\documentclass[11pt]{amsart}
\usepackage{amsmath,amsthm, amscd, amssymb, amsfonts, mathrsfs,mathtools,pb-diagram,pb-xy}
\usepackage{amsxtra, color}
\usepackage{enumitem}
\usepackage[all]{xy}
\usepackage{amssymb}
\usepackage{hyperref}
\usepackage[T2A,T1]{fontenc}
\usepackage[ansinew]{inputenc}
\usepackage{graphicx,fancyhdr}

\newcommand{\supera}[2]{{\mathbf A}(#1|#2)}
\newcommand{\superb}[2]{{\mathbf B}(#1|#2)}

\newcommand{\superd}[2]{{\mathbf D}(#1|#2)}
\newcommand{\superda}[1]{{\mathbf D}(2,1;#1)}
\newcommand{\superf}{{\mathbf F}(4)}
\newcommand{\superg}{{\mathbf G}(3)}

\newcommand{\dpn}{\widetilde{\mathcal{B}}}
\newcommand{\cP}{\mathcal{P}}

\newcommand{\Mat}{B}
\newcommand{\mat}{b}

\newcommand{\Matr}{C}
\newcommand{\matr}{c}

\newcommand{\Nt}{\mathtt{N}}

\newcommand{\lu}{\mathcal{L}}
\newcommand{\luq}{\lu_{\bq}}
\newcommand{\luqm}{\lu_{\bq^{(-1)}}}

\newcommand{\hopf}{\mathcal H}

\newcommand{\ydh}{{}^{H}_{H}\mathcal{YD}}
\newcommand{\ydgcq}{{}^{\Cset(\nu)\Gamma}_{\Cset(\nu)\Gamma}\mathcal{YD}}
\newcommand{\ydgcqp}{{}^{\Cset(\nu)\Gamma^+}_{\Cset(\nu)\Gamma^+}\mathcal{YD}}
\newcommand{\ydgcp}{{}^{\Cset \Gamma^+}_{\Cset \Gamma^+}\mathcal{YD}}
\newcommand{\ydgcqm}{\mathcal{YD}^{\Cset(\nu)\Gamma^-}_{\Cset(\nu)\Gamma^-}}
\newcommand{\ydgcm}{\mathcal{YD}^{\Cset \Gamma^-}_{\Cset \Gamma^-}}
\newcommand{\hyd}{\mathcal{YD}^{H}_{H}}
\newcommand{\kyd}{\mathcal{YD}^{K}_{K}}

\newcommand{\G}{\mathbb{G}}
\newcommand{\I}{\mathbb{I}}
\newcommand{\N}{\mathbb{N}}
\newcommand{\bq}{\mathfrak{q}}
\newcommand{\Pg}{\mathfrak{P}}
\newcommand{\bg}{\mathfrak{b}}
\newcommand{\h}{\mathfrak{h}}
\newcommand{\g}{\mathfrak{g}}
\newcommand{\n}{\mathfrak{n}}
\newcommand{\mm}{\mathfrak{m}}
\newcommand{\bp}{\mathfrak{p}}
\newcommand{\bpf}{\mathbf p}

\newcommand{\pt}{\mathtt{p}}

\newcommand{\Alg}{\operatorname{Alg}}
\newcommand{\ev}{\operatorname{ev}}
\newcommand{\res}{\operatorname{res}}
\newcommand{\Eq}{e}
\newcommand{\Fq}{f}

\newcommand{\Zc}{Z}

\numberwithin{equation}{section}
\theoremstyle{plain}
\newtheorem{maintheorem}{Theorem}
\newtheorem{theorem}{Theorem}[section]
\newtheorem{lemma}[theorem]{Lemma}
\newtheorem{definition-theorem}[theorem]{Definition-Theorem}
\newtheorem{proposition}[theorem]{Proposition}

\newtheorem{question}[theorem]{Question}
\newtheorem{corollary}[theorem]{Corollary}
\newtheorem{property}{}

\theoremstyle{definition}
\newtheorem{definition}[theorem]{Definition}
\newtheorem{example}[theorem]{Example}

\newtheorem{remark}[theorem]{Remark}
\newtheorem*{remark*}{Remark}
\newtheorem{remarks}[theorem]{Remarks}
\newtheorem{conjecture}[theorem]{Conjecture}

\newtheorem{genericityassumption}[theorem]{Non-degeneracy Assumption}
\newtheorem{step}{Step}

\newtheorem{case}{Case}

\newcommand\bth[1] { \begin{theorem}\label{t#1} }
\newcommand\ble[1] { \begin{lemma}\label{l#1} }

\newcommand\bpr[1] { \begin{proposition}\label{p#1} }
\newcommand\bqu[1] { \begin{question}\label{q#1} }
\newcommand\bco[1] { \begin{corollary}\label{c#1} }
\newcommand\bde[1] { \begin{definition}\label{d#1}\rm }
\newcommand\bex[1] { \begin{example}\label{e#1}\rm }
\newcommand\bre[1] { \begin{remark}\label{r#1}\rm }
\newcommand\bres[1] { \begin{remarks}\label{r#1}\rm }
\newcommand\bcj[1] { \begin{conjecture}\label{j#1}\rm }
\renewcommand {\eth} { \end{theorem} }
\newcommand{\ele} { \end{lemma} }

\newcommand{\epr} { \end{proposition} }
\newcommand{\equ} {\end{question} }
\newcommand{\eco} { \end{corollary} }
\newcommand{\ede} { \end{definition} }
\newcommand{\eex} { \end{example} }
\newcommand{\ere} { \end{remark} }
\newcommand{\eres} { \end{remarks} }
\newcommand{\ecj} { \end{conjecture} }
\newcommand{\enota} { \end{notation} }

\newcommand\pf{\begin{proof}}
\newcommand\epf{\end{proof}}

\newcommand{\Cset} {\mathbb C}
\newcommand{\Zset} {\mathbb Z}

\newcommand{\Kset} {\mathbb K}
\newcommand{\Fset}{\mathbb F}

\newcommand{\Z}{\mathbb Z}

\newcommand{\ku}{\mathbb C}

\newcommand{\Ac}{\mathcal{A}}          %mathcal

\newcommand{\toba}{{\mathcal B}}

\newcommand{\WP}{{\mathscr{P}}}

\newcommand{\Bgl}{\mathtt{wk}}

\newcommand{\Brown}{\mathtt{br}}

\newcommand{\fO}{\mathfrak O}

\newcommand{\Mg}{\mathfrak M}
\newcommand{\Ig}{\mathfrak I}

\newcommand{\CC} {\mathscr{C}}
\newcommand{\Ls}{\mathscr{L}}
\newcommand{\Ks}{\mathscr{K}}

\newcommand{\II} {{\mathcal{I}}}
\newcommand{\NN} {{\mathcal{N}}}
\newcommand{\QQ} {{\mathcal{Q}}}

\newcommand{\JJ} {{\mathcal{J}}}
\newcommand{\VV} {{\mathcal{V}}}

\newcommand{\Ss}{\mathcal{S}}
\newcommand{\LL}{\mathcal{L}}
\newcommand{\TT} {{\mathcal{T}}}
\newcommand{\ZZ}{\mathcal{Z}}
\newcommand{\Oc}{\mathcal{O}}
\newcommand{\EE}{\mathcal{E}}

\newcommand\cM{\mathcal{M}}
\newcommand\cW{\mathcal{W}}

\newcommand{\cX}{\mathcal{X}}

\newcommand{\ee}{\eta}
\newcommand{\ff}{\theta}

\newcommand{\zz}{\mathfrak{Z}}

\newcommand{\wfO}{\underline{\mathfrak O}}
\newcommand{\wbeta}{\underline{\beta}}
\newcommand{\walpha}{\underline{\alpha}}
\newcommand{\wgamma}{\underline{\gamma}}
\newcommand{\weta}{\underline{\eta}}
\newcommand{\wdelta}{\underline{\delta}}
\newcommand{\wmu}{\underline{\mu}}

\newcommand{\be} {\beta}
\newcommand{\vpi}{\varpi}

\newcommand{\Ga} {\Gamma}

\newcommand\mt  {\mapsto}
\newcommand\rcor {\rangle}
\newcommand\lcor {\langle}
\newcommand\ol {\overline}
\newcommand\wt {\widetilde}
\newcommand\wh {\widehat}

\newcommand{\id} {\operatorname{id}}

\newcommand{\Lie} {\operatorname{Lie}}

\newcommand\Acf {\mathsf{A}}

\newcommand{\bqf}{\mathbf{q}}
\newcommand{\qf}{\mathbf{q}}

\newcommand{\Aut}{\operatorname{Aut}}

\newcommand{\Der}{\operatorname{Der}}

\newcommand{\charr} {\operatorname{char}} 
  
\newcommand{\opp} {\operatorname{op}} 
\newcommand{\lcm}  {\operatorname{lcm}}

\newcommand{\ord} {\operatorname{ord}} 
\newcommand{\ad}{\operatorname{ad}}
\newcommand{\diag} {\operatorname{diag}} 
\newcommand{\Ker} {\operatorname{Ker}}

\newcommand{\End} {\operatorname{End}} 
 
\newcommand{\Hom}{\operatorname{Hom}}
\newcommand{\co}{\operatorname{co}}

\newcommand{\Inder}{\operatorname{Inn der}}

\newcommand{\Irr} {\operatorname{Irr}}

\renewcommand {\Im} {\operatorname{Im}}

\newcommand{\MaxSpec} {\operatorname{MaxSpec}}

\newcommand{\hopfq}{/ \hspace{-3pt} /}

\begin{document}
%%%%%%%%%%%%%%%%%%%%%%%%%%%%%%%%%%%%%%%%%%%%%%%%%%%%%%%%%%%%%%%%%%%%%%%%%%%
%%%%%%%%%%%%%%%%%%%%%%    Title    %%%%%%%%%%%%%%%%%%%%%%%%%%%%%%%%%%%%%%%%
\title[Poisson orders on large quantum groups]
{Poisson orders on large quantum groups}
\author[Nicol\'as Andruskiewitsch]{Nicol\'as Andruskiewitsch}
\address{\noindent FaMAF-CIEM (CONICET) \\ 
Universidad Nacional de C\'ordoba \\
Medina Allende s/n, Ciudad Universitaria \\
5000 C\'ordoba \\
Rep\'ublica Argentina
}
\email{nicolas.andruskiewitsch@unc.edu.ar}
\email{ivan.angiono@unc.edu.ar}
\author[Iv\'an Angiono]{Iv\'an Angiono}
%\address{}
%\email{}
\author[Milen Yakimov]{Milen Yakimov}
\address{
Department of Mathematics \\
Northeastern University \\
Boston, MA 02115
U.S.A.
}
\email{m.yakimov@northeastern.edu}
\date{}

\thanks{This material is based upon work supported by the National Science Foundation under
Grant No. DMS-1440140 while N. A. was in residence at the Mathematical Sciences
Research Institute in Berkeley, California, in the Spring 2020 semester. 
The work of N. A. and I. A. was partially supported by CONICET and Secyt (UNC). 
The work of M.Y was partially supported by NSF grants DMS-2131243 and DMS-2200762}

\begin{abstract} 
We develop a Poisson geometric framework for studying the representation theory of all contragredient quantum super groups at roots of unity. 
This is done in a uniform fashion by treating the larger class of quantum doubles of bozonizations of all distinguished pre-Nichols algebras \cite{Ang-dpn} 
belonging to a one-parameter family; we call these algebras \emph{large} quantum groups. We prove that each of these quantum algebras has a central Hopf subalgebra 
giving rise to a Poisson order in the sense of \cite{BG}. We describe explicitly the underlying Poisson algebraic groups and Poisson homogeneous spaces in terms of 
Borel subgroups of complex semisimple algebraic groups of adjoint type. The geometry of the Poisson algebraic groups and Poisson homogeneous spaces that are involved and 
its applications to the irreducible representations of the algebras $U_{\bq} \supset U_{\bq}^{\geqslant} \supset U_{\bq}^+$ are also described. Besides all (multiparameter) 
big quantum groups of De Concini--Kac--Procesi and big quantum super groups at roots of unity, our framework also contains the quantizations in characteristic 0 
of the 34-dimensional Kac-Weisfeler Lie algebras in characteristic 2 and the 10-dimensional Brown Lie algebras in characteristic 3. The previous approaches 
to the above problems relied on reductions to rank two cases and direct calculations of Poisson brackets, which is not possible in the super case since there are 
13 kinds of additional Serre relations on up to 4 generators. We use a new approach that relies on perfect pairings between restricted and non-restricted integral forms.
\end{abstract}
\maketitle

\setcounter{tocdepth}{1}
\tableofcontents

\section{Introduction}
\label{intro}

\subsection{Quantum groups and Poisson orders}
Let $\g$ be a complex finite-dimensional simple  Lie algebra and let
$\xi \in \ku$ be a root of 1 with some restrictions on its order depending on $\g$.
In the papers \cite{DK, DKP, DP} a quantized enveloping algebra
$U_{\xi}(\g)$ at  $\xi$
was introduced and studied; it is a version of the Drinfeld-Jimbo quantized 
universal enveloping algebra different from the one defined in \cite{L-fdHa-JAMS,L-roots of 1}.

The algebra $U_{\xi}(\g)$ 
is module-finite over a central Hopf subalgebra $Z_{\xi}(\g)$ and 
the corresponding small quantum group of Lusztig \cite{L-fdHa-JAMS,L-roots of 1} arises as the quotient 
$U_{\xi}(\g) \hopfq Z_{\xi}(\g)$ in the sense of Hopf algebras.
A geometric approach to the representation theory of $U_{\xi}(\g)$ was proposed in 
\cite{DP}, based on these facts.
The key ingredients of this approach are: 
\begin{itemize}[leftmargin=*]\renewcommand{\labelitemi}{$\circ$}
\item The existence of a Poisson structure on $Z_{\xi}(\g)$ so that the algebraic group $M$ corresponding to this algebra is a Poisson algebraic group, 
whose Lie bialgebra is dual to the standard Lie bialgebra structure on $\g$.
\item The Hamiltonian vector fields on $M$ extend to (explicit) derivations of $U_{\xi}(\g)$. 
\end{itemize}

The approach consists in packing the irreducible finite-dimensional representations of
$U_{\xi}(\g)$ along the symplectic leaves of $M$ and predicting their dimensions.
These ideas were distilled in the notion of Poisson order  in \cite{BG}, see Section \ref{sec:poisson}. The construction of a Poisson 
order structure on an algebra has substantial applications to the representation theory of the algebra: using this route
the irreducible representations of quantum function algebras were studied in \cite{DL}, 
the Azumaya loci of symplectic reflection algebras was described in \cite{BG}, the irreducible 
representations of the 3 and 4-dimensional PI Sklyanin algebras were fully classified in \cite{WWY1,WWY2}, 
the Azumaya loci of the multiplicative quiver varieties and quantum character varieties were studied in 
\cite{GJS}. See \cite[Part III]{BGb} for a comprehensive exposition of the applications 
the notion of Poisson order to the representation theory of quantum algebras at roots of unity. 

\subsection{Large quantum groups and pre-Nichols algebras}
\label{subsec:largeQG}
The main goal of this paper is to study by means of Poisson orders 
the representation theory of a larger class of Hopf algebras
introduced by the second author in \cite{A-presentation}
and studied in \cite{Ang-dpn}. They contain as special cases 
\begin{itemize}[leftmargin=*]\renewcommand{\labelitemi}{$\circ$}
\item all big quantum groups of De Concini--Kac--Procesi, 
\item all big contragredient quantum super groups,
\item and exceptional families that can be viewed as quantizations of
the universal enveloping algebras of simple Lie algebras in positive characteristic.
\end{itemize}
The keystone of the definition of
these Hopf algebras is the notion of \emph{distinguished pre-Nichols algebra}. 
It allows us to treat all of the above families uniformly without case-by-case considerations and 
computational arguments with quantum Serre relations. 

\smallbreak
Nichols algebras of diagonal type are essential for various classification problems of Hopf algebras.
Those of finite dimension were classified in the celebrated paper \cite{H-classif RS} while the defining relations
were provided in \cite{A-jems,A-presentation}. 
Let $\bq$ be a braiding matrix as in the list of \cite{H-classif RS} and let $\toba_{\bq}$
be the corresponding finite-dimensional Nichols algebra of diagonal type.
The distinguished pre-Nichols algebra $\dpn_{\bq}$
of $\toba_{\bq}$ is a covering of the latter defined by excluding 
the powers of the root vectors of Cartan type from its defining ideal.  
The Hopf algebras dealt with in the present paper are 
Drinfeld doubles of the bosonizations of the distinguished pre-Nichols algebras; they are denoted 
$U_\bq$, see \S \ref{subsec:Ubq}. They are shown to be module-finite over canonical central Hopf subalgebras $\Zc_{\bq}$,
which are the ones defined in \cite{Ang-dpn} if $\bq$ is not of type super $A$ and a one-dimensional extension of those in the 
super $A$ case, see \S \ref{subsec:Zq}. 

On the other hand, the graded dual  of $\dpn_{\bq}$ gives rise to a Lie algebra $\n_{\bq}$, which is either 0 or the 
nilpotent part of a semisimple Lie algebra $\g_{\bq}$ that is explicitly determined \cite{AAR3}.

\medbreak
We focus on Hopf algebras $U_\bq$ with a further restriction: the related Nichols algebra $\toba_{\bq}$ 
is deformable, i.e., belongs to a one-parameter family of Nichols algebras. 
We call them \emph{large quantum groups}. By inspection, the matrix $\bq$ is of one of three types:

\begin{enumerate}[leftmargin=*,label=\rm{(\alph*)}]
\item\label{item:Cartan-type} Cartan type (multiparameter versions of the quantum groups from \cite{DKP} without restrictions on $\xi$);

\smallbreak
\item\label{item:super-type} super type (multiparameter quantum groups associated to finite dimensional simple contragredient Lie superalgebras at roots of unity);

\smallbreak
\item\label{item:modular-type} modular types $\Bgl(4)$ or $\Brown(2)$ (quantizations at a root of unity 
in characteristic 0 of some simple Lie algebras in characteristics 2 and 3 respectively). 
\end{enumerate}

But it stems from the list in \cite{H-classif RS} that
there are finite-dimensional Nichols algebras of diagonal type that do not belong to such one-parameter families.

\begin{remark*} To be precise we need three technical assumptions:
\begin{enumerate} [leftmargin=*,label=\rm{(\roman*)}]
\item The base field is $\ku$ to have on hand symplectic leaves. 
(For other algebraically closed fields of characteristic 0, one can use symplectic cores \cite{BG} 
and argue as in \cite[\S 6.4]{WWY1}).  See Appendix \ref{subsec:symplectic-core} for a brief discussion of symplecitc leaves and cores.

\smallbreak
\item Condition \eqref{eq:condition-cartan-roots} is needed for the centrality of $Z_{\bq}$ in $U_{\bq}$.

\smallbreak
\item The Non-degeneracy Assumption \ref{genericityassumption} is used to identify some dual vector spaces in order to 
compute some Lie bialgebras. However this is not a constraint; each of the Hopf algebras $U_{\bq}$
that we consider can be obtained as a specialization from a family in many different ways. We prove in Proposition
\ref{prop:specialization-integer} that there  always  exist
ways that satisfy Assumption \ref{genericityassumption} and choose one such way. 
\end{enumerate}

\end{remark*}

See the Appendix \ref{sec:Nichols-specialization} and the survey \cite{AA-diag-survey} for full details on these algebras. 
We consider the chain of subalgebras $U_\bq^+ \subset U_\bq^{\geqslant} \subset U_\bq$ where
\begin{itemize}[leftmargin=*]\renewcommand{\labelitemi}{$\circ$}
\item $U_\bq^{\geqslant}$, the \emph{large quantum Borel subalgebra}, is identified with the bosonization of $\dpn_{\bq}$; 
\item $U_\bq^+$, the \emph{large quantum unipotent subalgebra}, is identified with $\dpn_{\bq}$.
\end{itemize}

Intersecting the central subalgebra $\Zc_{\bq}$ of $U_\bq$ gives the chain of central Hopf subalgebras 
\begin{equation}
\label{eq:nonsign-cent}
Z_\bq^+ \subset Z_\bq^{\geqslant} \subset Z_\bq.
\end{equation}
Each of these central Hopf subalgebras is actually isomorphic to a tensor product of a 
polynomial algebra and a Laurent polynomial algebra. We introduce the  algebraic groups 
$M_{\bq}$,  $M_{\bq}^\geqslant$ and $M_{\bq}^+$ as the maximal spectra of the Hopf algebras
$\Zc_{\bq}$, $\Zc_{\bq}^\geqslant$ and $ \Zc_{\bq}^+$,  respectively. 
We shall also need the opposite Borel $U_\bq^{\leqslant}$ and its central Hopf subalgebra
$\Zc_{\bq}^\leqslant$ with maximal spectrum $M_{\bq}^\leqslant$ and correspondingly $U_\bq^{-}$,
$\Zc_{\bq}^-$ and $M_{\bq}^-$.

\subsection{Main results} 
As said, this paper deals with the geometry of the Poisson algebraic group $M_{\bq}$ towards understanding
the representation theory of large quantum groups. 
This last question contain the description of the irreducible representations of contragredient quantum supergroups at roots of unity,
an important problem which is wide open even in the simplest case of $U_q({\mathfrak{sl}}(m|n))$. 
We present a foundation for a thorough investigation of these representations. 
We first summarize the main results in the following statement. Define the 
reductive Lie algebra
\[
\wt{\g}_{\bq} := 
\begin{cases}
\g_{\bq} \oplus \Cset, & \mbox{if $\bq$ is of type super $A$}
\\
\g_{\bq}, & \mbox{otherwise}.
\end{cases}
\]
Type super $A$ has the peculiarity that the rank of the Lie algebra $\g_{\bq}$ is one less than the rank of Nichols algebra $\toba_{\bq}$. Hence we need to  to enlarge the central subalgebra originally defined in \cite{Ang-dpn} adding an extra group-like element in order to have a central subalgebra $\Zc_{\bq}$  such that $U_\bq$ is module-finite over $\Zc_{\bq}$. 

Our first main result is the following:

\begin{maintheorem}\label{mainth:summary}
Let $U_{\bq}$ be a large quantum group as above. Then
\begin{enumerate} [leftmargin=*,label=\rm{(\alph*)}]
\item\label{item-mainth:poisson} The pair $(U_{\bq}, \Zc_{\bq})$ has the structure of a Poisson order in the sense of \cite{BG}.

\smallbreak
\item\label{item-mainth:solvable} 
The algebraic Poisson group $M_{\bq}$ is \emph{solvable}.
The Lie bialgebra of $M_{\bq}$ is dual to a Lie bialgebra structure on $\wt{\g}_{\bq}$
coming from an empty  Belavin--Drinfeld triple; the symplectic leaves of $M_{\bq}$
can be classified and related to conjugacy classes of the adjoint Lie group of $\wt{\g}_{\bq}$.

\smallbreak
\item\label{item-mainth:findimalg-reps} Every $z \in M_{\bq}$ with corresponding maximal ideal $\Mg_z$ gives rise to a finite-dimensional algebra 
$\hopf_z = U_{\bq}/  U_{\bq}\Mg_z$. Then $\hopf_z \simeq \hopf_{z'}$ whenever 
$z$ and $z'$ belong to the same symplectic leaf $S$. By abuse of notation we set $\hopf_S 
= \hopf_{z}$.
Every irreducible representation of $U_{\bq}$ is finite-dimensional and 
\begin{align*}
\Irr U_{\bq} = \bigcup_{S \text{ symplectic leaf of } M} \Irr \hopf_S.
\end{align*}
\end{enumerate}
\end{maintheorem}

Furthermore, we have analogous results for the pairs $(U_\bq^\geqslant, Z_\bq^\geqslant)$ and  $(U_\bq^+, Z_\bq^+)$.
Next we make more precise the claims of Theorem \ref{mainth:summary}. 
We fix a large quantum group  $U_{\bq}$.

\subsubsection{Poisson orders} We denote by $\ZZ(A)$ the center of an algebra $A$.
Because of the assumption that  $\toba_{\bq}$ is deformable in the class of Nichols algebras as mentioned above, 
we get Poisson order structures on the pairs $(U_\bq, \ZZ(U_\bq))$, $(U_\bq^\geqslant, \ZZ(U_\bq^\geqslant))$ and 
$(U_\bq^+, \ZZ(U_\bq^+))$ by specialization. As these centers are singular, it is more convenient to look at the
central subalgebras in \eqref{eq:nonsign-cent}. Part \ref{item-mainth:poisson} of Theorem \ref{mainth:summary}
is  included in the following result, see Theorem \ref{thm:Poisson-orders}. 

\begin{maintheorem}\label{mainthm:yak-A}
The pairs $(U_{\bq}, \Zc_{\bq})$, $(U_\bq^\geqslant, Z_\bq^\geqslant)$ and  $(U_\bq^+, Z_\bq^+)$
have Poisson order structures in the sense of \cite{BG} obtained from specialization.
\end{maintheorem}

Presently it is not known whether for the remaining braiding matrices $\bq$ in the list of \cite{H-classif RS} 
the pair $(U_{\bq}, \Zc_{\bq})$ has the structure of a Poisson order.
Indeed the other Nichols algebras of diagonal type with arithmetic root system
in the classification given in \cite{H-classif RS} do not admit such a one-parameter family
and for instance our proof of Theorem \ref{thm:Poisson-orders} does not generalize to them.

\subsubsection{Poisson algebraic groups and Lie bialgebras} Let $\h_{\bq}$ be a Cartan subalgebra of the semisimple Lie algebra $\g_{\bq}$ and extend it to a Cartan subalgebra
$\wt{\h}_{\bq}$ of the reductive Lie algebra $\wt{\g}_{\bq}$. We consider 
a Lie bialgebra structure on $\g_{\bq}$ that corresponds to the empty Belavin--Drinfeld triple \cite{BD}
and is explicitly defined in Theorem \ref{thm:mm*-Manin} and let
$\mm_{\bq}$ be the Lie bialgebra dual  to  $\wt{\g}_{\bq}\oplus \wt{\h}_{\bq}$, cf.  Eq. \eqref{eq:ident}.

Let $G_\bq$ be the semisimple algebraic group of adjoint type with  $\Lie G_{\bq} \simeq \g_{\bq}$. Denote
\[
\wt{G}_{\bq} := 
\begin{cases}
G_{\bq} \times \Cset^\times, & \mbox{if $\bq$ is of type super $A$}
\\
G_{\bq}, & \mbox{otherwise}.
\end{cases}
\]
For instance, when $U_{\bq} = U_q({\mathfrak{sl}}(m|n))$, 
$\wt{G}_{\bq} \simeq {\operatorname{PSL}}_m(\Cset) \times {\operatorname{PSL}}_n(\Cset) \times \Cset^\times$. 
Let $\wt{B}_{\bq}^{\pm}$ be a pair of opposite Borel subgroups of $\wt{G}_\bq$, $\wt{T}_{\bq} = \wt{B}_{\bq}^{+} \cap \wt{B}_{\bq}^{-}$ be the 
corresponding maximal torus and $N_{\bq}^\pm$ be the unipotent radicals of $\wt{B}_{\bq}^{+}$; we identify $N_{\bq}^{+} \simeq \wt{B}_{\bq}^{+} / \wt{T}_{\bq}$.

Here is a more precise statement of Theorem \ref{mainth:summary} Part \ref{item-mainth:solvable}, 
see Theorems \ref{thm:Ms-to-Bs} and \ref{Poisson0}.

\begin{maintheorem}\label{mainthm:yak-B}
\emph{(a)} The Poisson algebraic group $M_{\bq}$ is isomorphic 
to the product of two Borel subgroups of $\wt{G}_\bq$
and $\Lie M_{\bq} \simeq \mm_{\bq}$ as Lie bialgebras. 

\smallbreak 
\emph{(b)} The symplectic leaves of $M_{\bq}$ are in bijective correspondence  
with the conjugacy classes of $\wt{G}_\bq\times (\wt{T}_{\bq}/ \exp(\wt{Q}_\bq))$, where $\wt{Q}_\bq$ is a lattice 
in $\wt{\h}_\bq$; each leaf is isomorphic to an open
dense subset of the corresponding conjugacy class.
\end{maintheorem}

Note that the symplectic leaves are not algebraic varieties in general. 
The lattice $\wt{Q}_\bq$ is related to the continuous parameter of the corresponding Lie bialgebra, see Appendix \ref{subsec:B.1}.
In the case considered in \cite{DP} the Poisson structure is the standard one and $\wt{Q}_\bq$ coincides 
with the kernel of the exponential map restricted to the Cartan subalgebra.

Here are the promised versions for $M_{\bq}^\geqslant$ and $M_{\bq}^+$.

\begin{maintheorem}\label{mainthm:yak-C}
\emph{(a)} The Poisson algebraic group $M_{\bq}^\geqslant$ is isomorphic 
to the Borel subgroup $\wt{B}_{\bq}^{+}$. 
The Poisson structure is invariant under the left and right actions of  $\wt{T}_{\bq}$.

\smallbreak
\emph{(b)} The torus orbits of symplectic leaves of $M_{\bq}^{\geqslant}$ are 
the double Bruhat cells of $\wt{G}_{\bq}$ that lie in $\wt{B}_{\bq}^{+}$. 

\smallbreak
\emph{(c)} The  algebraic group $M_{\bq}^+$ is isomorphic 
to the unipotent radical $N_{\bq}^{+}$ of $\wt{B}_{\bq}^{+}$. 
It has a Poisson structure arising from the identification $N_{\bq}^{+} \simeq \wt{B}_{\bq}^{+} / \wt{T}_{\bq}$ which
is invariant under the left action of $\wt{T}_{\bq}$ 
and is a reduction of the Poisson structure on $\wt{B}_{\bq}^{+}$ from (a)
under the right action of $\wt{T}_{\bq}$.

\smallbreak
\emph{(d)} The torus orbits of symplectic leaves of $M_{\bq}^+$ are 
the open Richardson varieties of the flag variety $\wt{G}_{\bq}/\wt{B}_{\bq}^{+}$ that lie 
inside an open Schubert cell identified with $\wt{N}_{\bq}^{+}$. 
\end{maintheorem}

See Theorems \ref{Poisson+} and \ref{Poisson-unip}. We refer to \cite{FZ,KLS} 
for  information on double Bruhat cells and open Richardson varieties.
Here we recall briefly the definitions.

Let $v,w\in W$.
The corresponding double Bruhat cell 
  is $\wt{G}_{\bq}^{v,w} := \wt{B}^+_{\bq}v\wt{B}^+_{\bq}\cap \wt{B}^{-}_{\bq} w \wt{B}^{-}_{\bq}$. 
These cells form a partition:
$\wt{G}_{\bq} = \bigsqcup _{v,w\in W} \wt{G}_{\bq}^{v,w}$.

In turn the corresponding open Richardson variety is 
$\mathring{X}^v_w =   \mathring{X}^v \cap \mathring{X}_w \subset \wt{G}_{\bq}/ \wt{B}_{\bq}^+$
where $\mathring{X}_w = \wt{B}^-_{\bq}w / \wt{B}_{\bq}^+$ and $\mathring{X}^v  = \wt{B}^+_{\bq}v / \wt{B}_{\bq}^+$ are the Schubert cell and the opposite Schubert cell corresponding to $w$ and $v$ respectively.

\subsubsection{Representations}
 Since $U_{\bq}$ is a free $Z_{\bq}$-module of finite rank, it
is a PI-algebra. Let $V$ be an irreducible representation of $U_{\bq}$;
by the preceding $V$ is finite-dimensional and by the Schur Lemma, $Z_{\bq}$ acts on $V$ by 
some $z \in M_{\bq}$ (a central character) with corresponding maximal ideal $\Mg_z$.
Now the algebra $\hopf_z = U_{\bq}/  U_{\bq}\Mg_z$ is non-zero and finite-dimensional and 
$V$ becomes a $\hopf_z$-module. In other words the irreps of $U_{\bq}$ with central character $z$
are in bijective correspondence with the irreps of $\hopf_{z}$. Thus
\begin{align*}
\Irr U_{\bq} = \bigsqcup_{z \in M_{\bq}} \Irr \hopf_z.
\end{align*}
This circle of ideas is already present in \cite{DP}.
In this way, Part \ref{item-mainth:findimalg-reps} 
of Theorem \ref{mainth:summary} boils down to the following statement.

\begin{maintheorem}\label{mainthm:yak-Bc}
For every two points $z,z'$ in the same symplectic leaf of 
$M_{\bq}$, the algebras $\hopf_z$ and $\hopf_{z'}$ are isomorphic. In particular 
there is a dimension preserving bijection between the 
irreps of $U_{\bq}$ with central characters $z$ and $z'$. 
\end{maintheorem}

See Theorem \ref{Poisson0}. 
For instance, let $z = e$ be the identity of $M_{\bq}$. Then its symplectic leaf is $S =\{e\}$
and $\hopf_S = \hopf_e$ is the Drinfeld double of a suitable bosonization of the Nichols algebra $\toba_{\bq}$.
Assume that the matrix $\bq$ is of Cartan type. Then $\hopf_e$ is a variation of
the small quantum group of Lusztig (with an extra copy of the finite torus), 
with a notoriously difficult representation theory treated intensively in the literature.
Also, arguing as in \cite{DP} one concludes that $U_{\bq}$ is a maximal order. It is a domain, hence a prime algebra.
It follows that for generic $z$, $\hopf_{z}$ is semisimple by \cite[Theorem III.1.7]{BGb}.
But for super and modular types,  the representation theory of $\hopf_e$ is largely unknown, 
except for the somewhat standard fact that simple modules are classified by highest weights 
(but there is not even a conjecture for their characters). Also, $U_{\bq}$ is not a maximal order because 
it has nilpotent elements.

\bigbreak
We next write down the corresponding formulations for $M_{\bq}^\geqslant$, $M_{\bq}^\leqslant$, $M_{\bq}^+$ and $M_{\bq}^-$.
Let $\star \in \{\geqslant, \leqslant, +, -\}$. If $z \in M_{\bq}^{\star}$, then we denote by 
$\Mg_z^{\star}$ its maximal ideal in $Z_{\bq}^{\star}$ and
\begin{align}\label{eq:def-small-algebra}
\hopf_z^{\star} &= U_{\bq}^{\star}/ U_{\bq}^{\star} \left(\Mg_z ^{\star}\right).
\end{align}

Clearly these are finite-dimensional algebras.

\begin{maintheorem}\label{mainthm:yak-Cc}
\emph{(a)} For every $z,z'$ in the same double Bruhat cell inside $\wt{B}_{\bq}^{+}$, 
the algebras $\hopf^{\geqslant}_{z}$ and $\hopf^{\geqslant}_{z'}$ are isomorphic.
Analogously for $\hopf^{\leqslant}_{z}$ and $\hopf^{\leqslant}_{z'}$.

\smallbreak
\emph{(b)} For every $z,z'$ in the same open Richardson variety, 
the algebras $\hopf^{+}_{z}$ and $\hopf^{+}_{z'}$ are isomorphic.
Analogously for $\hopf^{-}_{z}$ and $\hopf^{-}_{z'}$.
\end{maintheorem}

See Theorems \ref{Poisson+} and \ref{Poisson-unip}. 

\medbreak
Notice also that $\hopf_{z}$ is a Hopf-Galois $\hopf_e$-object since $U_{\bq}$ is a 
cleft $\hopf_{\varepsilon}$-comodule algebra, see \S \ref{subsec:hopf-systems}.
Analogously,  
$\hopf_{z}^{\star}$ is a Hopf-Galois $\hopf_e^{\star}$-object  for $\star \in \{\geqslant, \leqslant\}$.

\subsection{Strategy and organization}
Our proofs of Theorems \ref{mainthm:yak-B}--\ref{mainthm:yak-Cc} follow a  different strategy from that of
\cite{DK, DKP, DP}. These papers  
rely on direct computations of Poisson brackets in terms of coordinates coming from Cartesian products of one-parameter unipotent 
groups and subsequent reductions to the rank 2 case. This approach does not work in
the more general context of \S \ref{subsec:largeQG} for several reasons, 
the simplest of which is that the quantum Serre relations for quantum supergroups or for quantum groups at $-1$
 involve more than two Chevalley generators. 

\smallbreak
Instead our approach is based on intrinsic properties of pairings between restricted and 
non-restricted integral forms of Hopf algebras. 
It does not rely on reduction to low rank cases. 
In particular, this approach provides new proofs of results in \cite{DK, DKP, DP}.
We expect that these ideas could be applied to other situations not covered in this paper.

\smallbreak
Next we overview briefly the main steps of the strategy:

\begin{step}
Let $\Cset(\nu)$ be the field of rational functions on $\nu$ and $\Acf$ the subalgebra  defined in \eqref{eq:A-alg}.
Since $\bq$ belongs to a family, there exists a chain of $\Cset(\nu)$-algebras
\begin{align*}
U_{\bqf}^+ \subset U_{\bqf}^{\geqslant} \subset U_{\bqf}
\end{align*}
and non-restricted integral forms over  $\Acf$  
\begin{align*}
U_{\bqf, \Acf}^+ \subset U_{\bqf, \Acf}^{\geqslant} \subset U_{\bqf, \Acf}
\end{align*}
such that the algebras $U_{\bq}^+ \subset U_{\bq}^{\geqslant} \subset U_{\bq}$ arise as specializations 
from these integral forms. This provides Poisson order structures on the pairs $(U_{\bq}^+, \ZZ(U_{\bq}^+))$, $(U_{\bq}^{\geqslant}, \ZZ(U_{\bq}^{\geqslant}))$
and  $(U_{\bq}, \ZZ(U_{\bq}))$. 
This step is carried out in Section \ref{sec:spec} in the framework of \cite{DP,BG} evoked in Section \ref{sec:poisson}.
\end{step}

\begin{step} We use Theorem \ref{thm:poisson-structure-ker} (on the restriction of Poisson order structures 
obtained from specialization to central subalgebras) to prove that the Poisson order structures on 
$(U_{\bq}^+, \ZZ(U_{\bq}^+))$ and $(U_{\bq}^{\geqslant}, \ZZ(U_{\bq}^{\geqslant}))$ restrict to 
$(U_{\bq}^+, Z_{\bq}^+)$ and $(U_{\bq}^{\geqslant}, Z_{\bq}^{\geqslant})$. 
To get a Poisson order structure on $(U_{\bq}, Z_{\bq})$ by restriction from $(U_{\bq}, \ZZ(U_{\bq}))$,
we need first to establish in Theorem \ref{th:Zq-invariant-under-Ti}
that the Weyl groupoid action preserves the central subalgebras $Z_{\bq}$. 
Along the way we also obtain that these Poisson structures 
on the  algebras $Z_{\bq}$ are equivariant under the Weyl groupoid. This step is carried out in Section \ref{sec:Pord}.
\end{step}

\begin{step} This is the matter of Section \ref{sec:PLgrps}
We introduce in \S \ref{subsec:Luszig-pm} non-restricted integral forms $U_{\bqf,\Acf}^{\res \pm}$ of $U_\bqf^\pm$ and $\Acf$-linear perfect pairings 
$U_{\bqf,\Acf}^{\res \pm}\times U_{\bqf,\Acf}^{\mp}\to \Acf$.  We prove 
\begin{enumerate}[leftmargin=*,label=\rm{(\roman*)}]
\item the specializations of $U_{\bqf,\Acf}^{\res \pm}$ are isomorphic to the Lusztig algebras 
defined in \cite{AAR-MRL}, see Proposition \ref{prop:evaluation-Ures};

\smallbreak
\item the cobrackets of the tangent Lie bialgebras to $M_{\bq}^\geqslant$ and 
$M_{\bq}^\leqslant$ are linearizations of those specializations, see Proposition \ref{prop:lem2}.
\end{enumerate}
In this way we control tangent Lie bialgebras intrinsically and consequently
we compute in Theorems \ref{thm:mm*} and \ref{thm:mm*-Manin} the tangent Lie bialgebras of the 
Poisson algebraic groups $M_{\bq}$, $M_{\bq}^\geqslant$ and $M_{\bq}^\leqslant$ by means of a Manin pair. 
Since these algebraic groups are connected we   
describe them as Poisson algebraic groups in terms of Borel subgroups of complex semisimple algebraic groups of adjoint type. 
Also, $M_{\bq}^\pm$ are presented as Poisson homogeneous spaces. 
\end{step}

\smallbreak
Finally, we discuss in Section \ref{sec:Poisson-geom} 
the Poisson geometry of the Poisson algebraic groups $M_{\bq}$, $M_{\bq}^\geqslant$ and the Poisson homogeneous 
space $M_{\bq}^+$, and the applications to the irreducible representations of $U_{\bq}$, $U_{\bq}^\geqslant$ and $U_{\bq}^+$. 

Besides, we discuss in Section \ref{sec:poisson} Poisson orders and their restrictions 
to central subalgebras, see Theorem \ref{thm:poisson-structure-ker}; 
Section \ref{sec:braided-bosonization} is devoted to preliminaries on Hopf algebra theory
while we present the main actors of this paper in Section \ref{sec:large}.

An in depth study of the restricted and nonrestricted integral forms of multiparameter 
quantum groups of Cartan type at roots of unity was carried out in \cite{GG}, based on the interpretation 
of those algebras as cocycle twists of the standard ones \cite{DK,DKP}. 
The authors completely describe the specialization at $1$ as the Poisson algebra of regular functions
on an explicit Poisson algebraic group and construct a Frobenius map which amounts to 
a Hopf algebra isomorphism between the specialization at $1$ and a central subalgebra 
of the specialization at a root of unity. Our results show that this is an isomorphism of Poisson algebras 
and that in the root of unity case it gives rise to a Poisson order structure.

\subsection*{Acknowledgements} This project started in visits of M. Y. to the University of  C\'ordoba in September 2017 and December 2018 supported by the program of guest professors of FaMAF. It was continued during visits  of I. A. and N. A. in February 2019 to the Lousiana State University. Progress on this material was reported at the plenary talk of I. A. at the XXIII \emph{Coloquio Latinoamericano de \'Algebra}, Mexico City (2019) 
and at the talk of M. Y. at the International conference on Hopf algebras, Nanjing (2019). 
We are grateful to the referee whose valuable suggestions helped us to improve the paper.

\subsection*{Notations}
The base field is $\ku$; all algebras, Hom's and tensor products are over $\ku$. 
If $t \in\N_0$, $n \in \N$ and $t < n$, then $\I_{t, n} := \{t, t+1,\dots,n\}$,
$\I_{n} := \I_{1, n}$.

For each integer $N>1$, let $\G_N$  be
the group of $N$-th roots of unity in $\Cset$ and let $\G'_N$ be its subset of primitive roots (of order $N$).
Also  $\G_{\infty} = \bigcup_{N \in \N} \G_N$, $\G'_{\infty} =\G_{\infty} - \{1\}$.

\section{Poisson orders and restrictions to central subalgebras} \label{sec:poisson}
This section contains background on Poisson orders, their construction from specializations, and their relations to 
Hopf algebras. We prove a general result on restrictions of Poisson orders to central subalgebras, Theorem \ref{thm:poisson-structure-ker}, which plays a key role later.
\subsection{Poisson orders}
Here we follow the exposition in \cite[Chapter 3, \S 11]{DP}. 
Consider
\begin{itemize}[leftmargin=*]\renewcommand{\labelitemi}{$\circ$}
\item a commutative $\ku$-algebra $A$ and $h \in A$ such that $A/h \simeq \ku$,

\item an $A$-algebra $U$ such that $h$ is not a zero divisor of $U$. The natural map $U \to U/(h)$ is denoted by $x \mapsto \overline{x}$.
\end{itemize}
For any $u \in U$ such that $\overline{u} \in \ZZ(U/(h))$ there is a linear map  $D_u \in \Hom U/(h)$ given by
\begin{align}\label{eq:der_a-gral}
D_u (y) &= \overline{\frac{[u,v]}{h}}, & \text{if } y &= \overline{v}.
\end{align}

\begin{proposition}\label{prop:der_a-gral} \cite[11.7]{DP} Let $u \in U$ such that $\overline{u} \in \ZZ(U/(h))$.
\begin{enumerate}[leftmargin=*,label=\rm{(\alph*)}]
\item\label{item:der_a-gral1} $D_u \in \Der U/(h)$. 

\medbreak
\item\label{item:der_a-gral2} Let $w \in U$. If $u' = u + hw$ so that $\overline{u} = \overline{u'}$, then $D_u  - D_{u'} = \ad \overline{w}$ is an inner derivation.
Conversely  the inner derivation $\ad \overline{w}$ coincides with $D_{hw}$.

\medbreak
\item\label{item:der_a-gral3}  Let $\varphi \in \Aut_{A-\rm{alg}} (U)$ and let $\overline{\varphi}$ be the induced automorphism of $U/(h)$. Then
\begin{align*}
\overline{\varphi} \circ D_u \circ \overline{\varphi}^{\, -1} &= D_{\varphi(u)}. 
\end{align*}

\medbreak
\item\label{item:der_a-gral4} There is natural Poisson structure on $\ZZ := \ZZ(U/(h))$ given by
\begin{align}\label{eq:Poisson-str-center}
\{x, y\} = D_u (y) &= \overline{\frac{[u,v]}{h}}, & \text{if } x &= \overline{u}, y = \overline{v}.
\end{align}

\medbreak
\item\label{item:der_a-gral5} 
The map $\varphi \mapsto \overline{\varphi}$ gives a group homomorphism $\Aut_{A-\rm{alg}} (U) \to \Aut_{\rm{Poisson}} (\ZZ)$.

\medbreak
\item\label{item:der_a-gral6}  $\LL = \{D_v: v \in U, \overline{v} \in \ZZ \}$
is a Lie subalgebra of $\Der U/(h)$. Indeed 
\begin{align*}
[D_u, D_v] &= D_{\frac{[u,v]}{h}}, & v \in U, \, \overline{v} &\in \ZZ.
\end{align*}

\medbreak\item\label{item:der_a-gral7} 
The Poisson structure gives rise to a Lie subalgebra $\LL'$ of $\Der\ZZ$ that fits into the complex  
\begin{align}\label{eq;exactseq-DCP}
\xymatrix@C=10pt{0 \ar[rr] & & \Inder (U/(h)) \ar[rr] & & \LL \ar[rr] & & \LL' \ar[rr] & & 0. }
\end{align}
The sequence \eqref{eq;exactseq-DCP} is exact if and only if the Poisson center of $\ZZ$ is trivial (i.e., there are no \emph{Casimir elements} except 0). \qed
\end{enumerate}
\end{proposition}

Brown and Gordon \cite{BG} axiomatized the ingredients of the above setting as follows:
\begin{definition}\label{def:Porders} A pair of $\ku$-algebras $(R, Z)$ is called a \emph{Poisson order} if $Z$ is a central subalgebra of $R$, $R$ is a  finitely generated $Z$-module  and the following  conditions hold:
\begin{enumerate}[leftmargin=*,label=\rm{(\alph*)}]
\item $Z$ is equipped a structure of Poisson algebra $\{\cdot, \cdot\}$; 
\item There exists a linear map $D : Z \to \Der_\ku(R)$ such that $D_z|_Z = \{ z, - \}$ for all $z \in Z$.
\end{enumerate}
\end{definition}
Reshetikhin, Voronov and Weinstein defined earlier a related notion of a {\em{Poisson fibered algebra}}, see \cite[Definition 2.1]{RVW}.  
In the above terminology, such an algebra is a Poisson order with the additional property that
\[
D_{z_1 z_2} (r) = z_1 D_{z_2}(r) + z_2 D_{z_1}(r) \quad \mbox{for all} \; \; z_1, z_2 \in Z, r \in R.
\]

Proposition \ref{prop:der_a-gral} proves that the pair $(U/(h), \ZZ(U/(h)))$ has a canonical 
structure of Poisson order when $U/(h)$ is module finite over $\ZZ(U/(h))$. The Poisson bracket on $\ZZ(U/(h))$ is given by \eqref{eq:Poisson-str-center}. 
The linear map $D$ is the map induced from the one in \eqref{eq:der_a-gral} by taking a linear section
of the canonical projection $U \to U/(h)$.
\medskip

The main application of Poisson orders for us is the following result, inspired by \cite[Cor. 11.8]{DP}, \cite[Cor. 9.2]{DL}. 
Assume that $R$ is affine, i.e., it is a finitely generated algebra (hence also $Z$ is affine).

\begin{theorem}\label{theorem:BG-reps} \cite[Theorem 4.1]{BG}
Let $(R, Z)$ be a Poisson order and 
$M:= \MaxSpec Z$. 
Given $x \in M$ with maximal ideal $\Mg_x$, let $\Ac_x \coloneqq R/\Mg_xR$, a finite dimensional algebra.
If $x$ and $y$ belong to the same symplectic core, then $\Ac_x  \simeq \Ac_y$ as algebras.
\end{theorem}

\subsection{Restrictions of Poisson orders from specializations}
In the setting of Proposition \ref{prop:der_a-gral} the center $\ZZ= \ZZ(U/(h))$ can be singular and is more useful to work with suitable subalgebras $\ZZ'$.
Next we prove a general fact for the construction of Poisson orders on pairs $(U/(h),\ZZ')$ for subalgebras $\ZZ'$ defined from algebra automorphisms and skew-derivations. For this purpose we fix:
\begin{itemize}[leftmargin=*]
\item $A$-algebra endomorphisms $\varsigma_i:U\to U$, $i\in\I$. We denote by $\overline{\varsigma}_i$ the corresponding $\mathbb C$-algebra endomorphisms of $U/(h)$ induced by $\varsigma_i$.
\item $A$-linear $(\id,\varsigma_i)$-derivations 
$\partial_i:U\to U$, $i\in\I$. We denote by $\overline{\partial}_i$ the corresponding $\mathbb C$-linear $(\id,\overline{\varsigma}_i)$-derivations induced by $\partial_i$.
\end{itemize}

\begin{theorem}\label{thm:poisson-structure-ker}
The Poisson order structure on $(U/(h), \ZZ(U/(h)))$ from Proposition \ref{prop:der_a-gral} 
restricts to a Poisson order structure on $(U/(h), \ZZ')$, where 
\begin{align}\label{eq:poisson-structure-subalgebra}
\ZZ' &:= \ZZ \cap \left(\cap_{i\in\I} \ker \overline{\partial}_i\right) \cap
\left(\cap_{i\in\I} \ker (\overline{\varsigma}_i-\id) \right).
\end{align}
\end{theorem}
\pf
We have to check that $\{\ZZ',\ZZ'\} \subset \ZZ'$. Let $x_j\in\ZZ'$ and $u_j\in U$ such that $x_j= \overline{u_j}$, $j=1,2$.
Fix $i\in\I$. As $\overline{\varsigma}_i(x_j)=x_j$ and $\partial_i(x_i)=0$, there are $v_j, w_j\in U$ such that
\begin{align*}
\varsigma_i (u_j)& = u_j + h\, v_j, & \partial_i (u_j) &= h\, w_j, & j&=1,2.
\end{align*}
Now we compute
\begin{align*}
\overline{\varsigma}_i \{x_1,x_2\} &= \overline{\varsigma}_i \left( \overline{\frac{[u_1,u_2]}{h}} \right)
= \overline{\frac{[\varsigma_i(u_1),\varsigma_i(u_2)]}{h}}
\\
&= \overline{\frac{[u_1,u_2]}{h}} +\overline{[u_1,v_2]} +\overline{[v_1,u_2]}
= \{x_1,x_2\} +[x_1,\overline{v_2}] +[\overline{v_1},x_2] = \{x_1,x_2\},
\\
\overline{\partial}_i \{x_1,x_2\} &=\overline{\partial}_i \left( \overline{\frac{[u_1,u_2]}{h}} \right)
= \overline{\frac{\partial_i(u_1)\varsigma_i(u_2) + u_1\partial_i(u_2) - \partial_i(u_2)\varsigma_i(u_1) - u_2\partial_i(u_1)}{h}}
\\
&= \overline{w_1(u_2 + h\, v_2) + u_1 w_2 - w_2 (u_1 + h\, v_1) - u_2w_1} =[x_1,\overline{w_2}] +[\overline{w_1},x_2] =0.
\end{align*}
Hence $\{x_1,x_2\}\in \ker \overline{\partial}_i \cap \ker (\overline{\varsigma}_i-\id)$ for all $i\in\I$ so $\{x_1,x_2\}\in\ZZ$.
\epf

\subsection{Poisson-Hopf algebras} 
Assume that in the above setting $U$ is a Hopf algebra over $A$. 
Then $U/(h)$ has a canonical structure of Hopf algebra over $\ku$.

Let $u \in U$ such that $\overline{u} \in \ZZ(U/(h))$ and furthermore $\Delta(\overline{u}) \in \ZZ \left( U/(h) \otimes U/(h) \right)$.
Then
\begin{align}\label{eq:Delta-D}
D_{\Delta (u)} \Delta (y) &= \Delta (D_u (y)),& y &\in U/(h).
\end{align}

\begin{proposition}\label{prop:der_a-gral-centralHopf} \cite[11.7]{DP} Let $B$ be a central Hopf subalgebra of $U/(h)$.Then 
\begin{align*}
T := \text{ minimal subalgebra of $\ZZ$ containing $B$ and closed under the Poisson bracket}
\end{align*}
is a central Hopf subalgebra of $U/(h)$, hence a Poisson-Hopf algebra.
\end{proposition}

We recall the elegant proof of \cite{DP}. 

\pf Apply \eqref{eq:Delta-D} to $y\in \ZZ$ and $x = \overline{u}$ to get $\Delta (\{x,y\}) = \{\Delta(x), \Delta(y)\}$ for all $x, y\in \ZZ$.
Hence $\widetilde T = \{t\in T: \Delta(t) \in T \otimes T \}$, which is a subalgebra containing $B$, is also closed under Poisson bracket; thus $\widetilde T = T$.
\epf

\section{Hopf algebras}\label{sec:braided-bosonization}
In this section we collect preliminaries on (braided) Hopf algebras (always with bijective antipode $\Ss$), bosonizations,
braided vectors spaces of diagonal type, Nichols algebras, Weyl groupoids, distinguished 
pre-Nichols algebras and Lusztig algebras. We refer to \cite{Rad-libro,A-leyva} for more information on Hopf algebras, Nichols algebras, Nichols algebras of diagonal type, respectively.

\subsection{Cleft comodule algebras}\label{subsec:hopf-systems}

Let  $\hopf$ be a Hopf algebra  with a central  Hopf subalgebra $Z$.
Given $z \in G = \Alg (Z, \ku)$ (the pro-algebraic group defined by $Z$), let
\begin{align*}
\Mg_z &= \ker z, & \Ig_z &= \hopf \Mg_z, & \hopf_z &= \hopf/  \Ig_z;
\end{align*}
thus $\hopf_{z}$ is an algebra (with multiplication $m_{z}$ and unit $u_{z}$)
and the natural projection $p_z : \hopf \to \hopf_z$ is an algebra map.  
Then $\hopf_{\varepsilon}$ is a quotient Hopf algebra of $\hopf$ and 
there is an exact sequence of Hopf algebras $Z \hookrightarrow \hopf \twoheadrightarrow \hopf_{\varepsilon}$.
Also  for any $z,z' \in G$ there are well-defined algebra morphisms
$\Delta_{z,z'}: \hopf_{zz'} \to \hopf_z \otimes \hopf_z'$ and in particular the maps
\begin{align*}
\varrho_{z} &:= \Delta_{z, \varepsilon}: \hopf_{z} \to \hopf_z \otimes \hopf_{\varepsilon}, &
\lambda_{z} := \Delta_{\varepsilon, z}: \hopf_{z} \to \hopf_{\varepsilon}\otimes \hopf_z,
\end{align*}
make $\hopf_{z}$ a $\hopf_{\varepsilon}$-bicomodule algebra for $z \in G$.
Clearly
\begin{align}
\varrho_{z} p_z &= (p_z \otimes p_\varepsilon) \Delta_{\hopf}, &
\lambda_{z} p_z = (p_\varepsilon \otimes p_z) \Delta_{\hopf}.
\end{align}

Recall that a right $K$-comodule algebra $A$ (over a Hopf algebra $K$)
is cleft if there exists a convolution-invertible morphism of $K$-comodules
$\chi: K \to A$.

\begin{lemma}\label{lemma:hopf-system-cleft}
If the $\hopf_{\varepsilon}$-comodule algebra $\hopf$ with coaction $\varrho = (\id \otimes p_{\varepsilon}) \Delta_{\hopf}$
is cleft, then so is $\hopf_{z}$ for any $z \in G$. In particular $\hopf_{z}$ is a Hopf-Galois object over $\hopf_{e}$.

If $\hopf$ is a pointed Hopf algebra, then $\hopf_{z}$ is $\hopf_{e}$-cleft for all $z \in G$.
\end{lemma}

\pf If $\chi: \hopf_{\varepsilon} \to \hopf$ is a morphism  of $\hopf$-comodules, then so is 
$\chi_{z} := p_{z}\chi: \hopf_{\varepsilon} \to \hopf_{z}$: 
\begin{align*}
(\chi_{z} \otimes \id) \varrho_{\varepsilon} &= (p_{z} \otimes \id)(\chi \otimes \id)\Delta_{\hopf_{\varepsilon}}
\\
&= (p_{z} \otimes \id) (\id \otimes p_{\varepsilon}) \Delta_{\hopf} \chi =  \varrho_{z} p_z \chi =
\varrho_{z} \chi_z.
\end{align*}
If $\chi$ is convolution-invertible, then so is $\chi_{z}$
since $p_{z}$ is an algebra map.

For the last statement, $\hopf$ is $\hopf_{\varepsilon}$-cleft by \cite[4.3]{Sc}, and then we apply the first part.
\epf

We refer to \cite{S} for Hopf-Galois objects. 
In the setting of Cayley--Hamilton Hopf algebras, which is a refinement of the above setting for the pair $(\hopf, Z)$, 
a tensor product decomposition of the 
irreducible representations of $\hopf_z$ was obtained in \cite{DPRR}. 
\subsection{Braided Hopf algebras and bosonization}

Recall that a braided vector space is a pair $(\VV,c)$ where $\VV$ is a vector space 
and $c \in GL(\VV \otimes \VV)$ is a solution of the braid equation: $(c\otimes \id)(\id \otimes c)(c\otimes \id) = (\id \otimes c)(c\otimes \id)(\id \otimes c)$. There are natural notions of morphisms of braided vector spaces and braided Hopf algebras (braided vector spaces with compatible algebra and coalgebra structures), see \cite{Tk1} for details. To distinguish comultiplications of braided Hopf algebras from those of Hopf algebras,
we use a variation of the Sweedler notation for the former: $\underline{\Delta} (r) = r^{(1)} \otimes r^{(2)}$.

\medbreak
Let $H$ be a Hopf algebra. Then the category of (left) Yetter-Drinfeld modules $\ydh$  is
a  braided tensor category and there is a forgetful functor from $\ydh$ to the category of braided vector spaces, namely 
$\VV \in \ydh$ goes to $(\VV,c)$ where $c \in GL(\VV \otimes \VV)$ is given by $c(v \otimes w) = v_{(-1)} \cdot w \otimes v_{(0)}$ in  Sweedler notation. This forgetful functor sends Hopf algebras in $\ydh$ to braided Hopf algebras. In turn Hopf algebras in $\ydh$ are noteworthy because of the Radford-Majid bosonization that provides a bijective correspondence between their collection and the collection of triples $(A, \pi, \iota)$ where
 $A \overset{\pi}{\underset{\iota}{\rightleftarrows}} H$ are morphisms of Hopf algebras with $\pi\iota = \id_H$. See  \cite{Rad-libro} for an exposition. More precisely, the correspondence sends the Hopf algebra $R\in \ydh$ to the bosonization $R\# H$ and the triple $(A, \pi, \iota)$ to
the algebra of right coinvariants $R = A^{\co \pi}$.

\medbreak Similar notions and results hold for the category of (right) Yetter-Drinfeld 
modules $\hyd$ consisting of right $H$-modules and right $H$-comodules $\VV$ satisfying the compatibility 
\begin{align*}
(v \cdot h)_{(0)} \otimes (v \cdot h)_{(1)} &= v_{(0)} \cdot h_{(2)} \otimes \Ss(h_{(1)}) v_{(1)}  h_{(3)},&
v&\in \VV, \; h \in H.
\end{align*}
For convenience of the reader we spell out the precise definitions. First, any $\VV \in \hyd$ becomes a braided vector space
with $c \in GL(\VV \otimes \VV)$ and its inverse given by 
\begin{align}\label{eq:braiding-hyd}
c(v \otimes w) &=  w_{(0)}\otimes v \cdot w_{(1)}, &
c^{-1}(v \otimes w) &=  w \cdot \Ss^{-1} (v_{(1)}) \otimes v_{(0)},& v,w&\in V.
\end{align} 
Let $(A, \pi, \iota)$ be a triple as before. Then the subalgebra of left coinvariants 
\begin{align*}
S = {}^{\co \pi}A =\{s\in A: (\pi\otimes \id)\Delta(s) = 1 \otimes s\}
\end{align*}
 becomes a Hopf algebra in $\hyd$ with right action $\cdot$, right
 coaction $\rho$ and comultiplication $\underline{\Delta}$ given by
\begin{align*}
s\cdot h &= \Ss(h_{(1)}) s h_{(2)},& \rho(s) &= (\id\otimes \pi)\Delta(s),&
\underline{\Delta} (s) &= s_{(1)} \otimes \underline{\vartheta}(s_{(2)}),& s&\in S, h \in H,
\end{align*}
where $\underline{\vartheta}: A \to S$ is given by $\underline{\vartheta}(a) = \pi(\Ss(a_{(1)}))a_{(2)}$, $a \in A$.
Conversely, the bosonization $H\# S$ of a Hopf algebra $S$ in $\hyd$ is the vector space 
$H\otimes S$ with the right smash product and coproduct. That is, given $s,\widetilde{s} \in S$ and 
$h, \widetilde{h} \in H$,
\begin{align*}
(h\# s) (\widetilde{h}\# \widetilde{s}) &= h\widetilde{h}_{(1)} \# (s \cdot\widetilde{h}_{(2)})  \widetilde{s},
&
\Delta (h\# s) &= h_{(1)} \# (s^{(1)})_{(0)} \otimes h_{(2)}(s^{(1)})_{(1)} \#  s^{(2)}.
\end{align*}

\subsection{Nichols algebras}
Let $\VV \in \ydh$. Then the tensor algebra $T(\VV)$ is naturally a Hopf algebra in $\ydh$. 
A \emph{pre-Nichols algebra} of $\VV$ is a  factor of $T(\VV)$ 
by a graded Hopf ideal in $\ydh$  supported in degrees $\geqslant 2$. 
The maximal Hopf ideal among those is denoted by $\JJ(\VV)$;
the Nichols algebra of $\VV$ is the quotient $\toba(\VV) = T(\VV) / \JJ(\VV)$.

The tensor algebra of a braided vector space $(\VV,c)$ is also a braided Hopf algebra in the sense of \cite{Tk1};
a \emph{pre-Nichols algebra} of $\VV$ is a  factor of $T(\VV)$ 
by a braided graded Hopf ideal  supported in degrees $\geqslant 2$. 
The maximal Hopf ideal among those is denoted $\widetilde{\JJ}(\VV)$;
the Nichols algebra of $\VV$ is the quotient $\toba(\VV) = T(\VV) / \widetilde{\JJ}(\VV)$.

These two structures are compatible, i.e., if $\VV \in \ydh$ and $(\VV,c)$ is the corresponding 
 braided vector space, then $\JJ(\VV) = \widetilde{\JJ}(\VV)$. But a pre-Nichols algebra of
$(\VV,c)$ does not necessarily come as the forgetful functor applied to a pre-Nichols algebra of
$\VV \in \ydh$.

\begin{remark}\label{rem:right-Nichols}
Let $H$ be cosemisimple, $\VV \in \ydh$ and $G = \toba(\VV) \# H = \oplus_{n\in \N_0} G^n$, where $G^n = \toba^n(\VV) \# H$. 
By other characterizations of Nichols algebras, we know that
\begin{enumerate} [leftmargin=*,label=\rm{(\alph*)}]
\item $\toba(\VV)$ is coradically graded and generated in degree 1;

\item $G$ is coradically graded and generated in degree 1.
\end{enumerate}
Since the projection $\pi: G \to H$ is graded, the subalgebra of left coinvariants 
$S = {}^{\co \pi}G$ inherits the grading of $G$; by a standard argument it is also coradically graded and generated in degree 1. Thus $S$ is a Nichols algebra in $\hyd$.
\end{remark} 

\subsection{Hopf skew-pairings of bosonizations}
Let $\langle \cdot , \cdot \rangle: M \times V \to \ku$ be a bilinear form 
between two vector spaces $M$ and $V$. We denote by 
$\langle \cdot , \cdot \rangle: (M \otimes M) \times (V \otimes V) \to \ku$ the bilinear form 
determined by
\begin{align}\label{eq:bilfor-tensorproduct}
\langle m \otimes m', v \otimes v'\rangle &= \langle m,v' \rangle \langle m', v \rangle,&
m,m' &\in M, \; v,v'\in V.
\end{align}

Let $H$ and $K$ be two Hopf algebras. 
A bilinear form $\langle \cdot , \cdot \rangle: K \times H \to \ku$ is a \emph{Hopf skew-pairing} (or skew-pairing of Hopf algebras) if for all for $k, k'\in K$, $h, h' \in H$, 
\begin{align}\label{eq:hopf-pairing}
\begin{aligned}
\langle k, hh'\rangle &= \langle \Delta^{\opp}(k), h\otimes h' \rangle,&
\langle kk', h\rangle &= \langle k \otimes k',\Delta(h) \rangle, & & 
\\
\langle k,1\rangle &= \varepsilon(k),&
\langle 1, h\rangle &= \varepsilon(h), &
\langle \Ss(k), h\rangle &= \langle k, \Ss(h)\rangle.
\end{aligned}
\end{align}
A skew-pairing of braided Hopf algebras is defined by \eqref{eq:hopf-pairing} but  with the convention
\begin{align*}
\underline{\Delta}^{\opp} = c^{-1} \underline{\Delta}.
\end{align*}

Let us fix a Hopf skew-pairing $\langle \cdot , \cdot \rangle: K \times H \to \ku$.
A \emph{YD-pairing} between $\cM \in \kyd$ and $\VV \in \ydh$ is a bilinear form 
$\langle \cdot , \cdot \rangle: \cM \times \VV \to \ku$ such that
\begin{align}\label{eq:YD-pairing-def}
&\begin{aligned}
\langle m \cdot k, v\rangle &= \langle k,v_{(-1)}\rangle \langle m , v_{(0)}\rangle,&
\\
\langle m, h \cdot v\rangle &= \langle m_{(1)}, h\rangle \langle m_{(0)}, v\rangle,
\end{aligned}&
m\in \cM, \; k &\in K, \; v \in \VV, \; h \in H.
\end{align}

We recall the following well-known result, whose proof is straightforward.

\begin{lemma}\label{lemma:bilfor-bosonization} Let $R$ be a Hopf algebra in $\ydh$, $S$ be a Hopf algebra in $\kyd$
and $\langle \cdot , \cdot \rangle:  (K \# S) \times (R \# H) \to \ku$ be a bilinear form such that 
\begin{align}\label{eq:bilfor-bosonization}
\langle ky, x h\rangle & = \langle k,h\rangle \langle y, x\rangle,&
y \in S, \; k &\in K, \; x \in R, \;h \in H.
\end{align}
Then the following are equivalent:
\begin{enumerate}[leftmargin=*,label=\rm{(\alph*)}]
\item\label{item:bilfor-bosonization1} $\langle \cdot , \cdot \rangle$ is a Hopf skew-pairing.

\item\label{item:bilfor-bosonization2}  The restriction of  $\langle \cdot , \cdot \rangle$ to $K \times H$ is a Hopf skew-pairing
and the restriction of  $\langle \cdot , \cdot \rangle$ to $S \times R$ is both a skew-pairing of braided Hopf algebras and a YD-pairing.\qed
\end{enumerate}
\end{lemma}

A YD-pairing between $\cM \in \kyd$ and $\VV \in \ydh$ extends canonically to a YD-pairing 
$\langle \cdot , \cdot \rangle: T(\cM) \times T(\VV) \to \ku$. This extension is actually 
a braided Hopf skew-pairing, i.e., it satisfies \eqref{eq:hopf-pairing} with respect to the braided
comultiplications. The bilinear form
\begin{align*}
\langle \cdot , \cdot \rangle: & (K \# T(\cM)) \times (T(\VV) \# H) \to \ku, &
\langle k \# y, x \# h\rangle & := \langle k,h\rangle \langle y, x\rangle,
\end{align*}
$y \in T(\cM)$, $k \in K$, $x \in T(\VV)$, $h \in H$ is a Hopf skew-pairing by Lemma \ref{lemma:bilfor-bosonization}.

\medbreak Assume that $\dim \cM < \infty$. Then the radical $T(\cM^*)^{\perp}$
with respect to $\langle \cdot , \cdot \rangle$ coincides with $\JJ(\cM)$. Hence, for any $\VV$ YD-paired with $\cM$ we have
\begin{align*}
T(\VV)^{\perp} \supseteq \JJ(\cM).
\end{align*}
Consequently, if $\dim \cM < \infty$ and $\dim \VV < \infty$, $\toba$ is a pre-Nichols algebra of $\cM$ in $\kyd$ and $\EE$ is a pre-Nichols algebra of $\VV$ in $\ydh$, 
then $\langle \cdot , \cdot \rangle$ 
descends to Hopf skew-pairings  
$\langle \cdot , \cdot \rangle: \toba \times \EE  \to \ku$ and
$\langle \cdot , \cdot \rangle: (K \# \toba) \times (\EE \# H) \to \ku$.

\subsection{Nichols algebras of diagonal type}\label{subsec:diag-type} 
We fix $\theta \in \N$ and set $\I = \I_{\theta}$.
Let $(V, c)$ be a (complex) braided vector space of diagonal type with braiding matrix 
\begin{align}\label{eq:braiding-matrix}
\bq = (q_{ij} ) \in \big( \Cset^\times \big)^{\I \times \I}
\end{align}
with respect to a basis $(x_i)_{i\in \I}$, 
i.e., $c(x_i \otimes x_j) =q_{ij} x_j \otimes x_i$ for all $i,j\in \I$.
We assume that $\dim \toba(V) < \infty$. These Nichols algebras are classified in \cite{H-classif RS}. 
Throughout the paper we will also assume that the Dynkin diagram of $\bq$ is connected, for simplicity of the exposition.

\medbreak
The canonical basis of $\Z^{\I}$ is denoted $\alpha_1, \dots, \alpha_\theta$.
The algebra $T(V)$ is $\Z^{\I}$-graded, with grading $\deg x_i = \alpha_i$, $i\in \I$.
This grading naturally specializes to the standard $\N_0$-grading.

\smallbreak
Let  $\bq:\Z^{\I}\times\Z^{\I}\to \Cset^\times$
be the $\Z$-bilinear forms associated to the matrix $\bq$, i.e.,
$\bq(\alpha_j,\alpha_k) :=q_{jk}$, $j,k \in\I$.
If $\alpha,\beta  \in \Z^{\I}$ and $i\in \I$, then we set
\begin{align}\label{eq:notation-qab}
q_{\alpha\beta} &= \bq(\alpha,\beta), & q_{\alpha\alpha} &= \bq(\alpha,\alpha),&
N_{\alpha} &= \ord q_{\alpha\alpha}, & N_{i} &= \ord q_{\alpha_i\alpha_i} = N_{\alpha_i}.
\end{align}

\begin{remark}\label{rem:sig-partial}
Every  $\Z^{\I}$-graded pre-Nichols algebra of $V$ admits algebra automorphisms $\varsigma_i^{\bq}$
and $(\id,\varsigma_i^{\bq})$-derivations $\partial_i^{\bq}$  for each $i\in\I$; that is,
\begin{align*}
\partial_i^{\bq}(xy) &= \partial_i^{\bq}(x)\varsigma_i^{\bq}(y) + x \partial_i^{\bq}(y), & &x,y\in T(V).
\end{align*}
\end{remark}

Indeed the algebra automorphism $\varsigma_i^{\bq}:T(V)\to T(V)$ is given by
\begin{align*}
\varsigma_i^{\bq}(x) &= \bq(\alpha_i,\beta) x, & x\in T(V) & \; \; \text{homogeneous of degree }\beta\in\Z^{\I}.
\end{align*}
The linear endomorphisms $\partial_i^{\bq}:T(V)\to T(V)$ are defined as follows.
Let $\Delta_{m,n}(x)$ be the homogeneous component of $\Delta(x) \in T(V)\otimes T(V)$ of degree $(m,n) \in \N_0^2$. 
Then
\begin{align*}
\Delta_{n-1,1}(x) & = \sum_{i \in\I} \partial_i^{\bq}(x) \otimes x_i, & &x\in T^n(V).
\end{align*}
It is easy to see that $\partial_i^{\bq}$ is a $(\id,\varsigma_i^{\bq})$-derivation.
If $\toba$ is a quotient of $T(V)$ by a $\Z^{\I}$-homogeneous ideal, then $\varsigma_i^{\bq}$ induces an algebra automorphism of $\toba$, also denoted by $\varsigma_i^{\bq}$, and  $\partial_i^{\bq}$ induces a $(\id,\varsigma_i^{\bq})$-derivation of $\toba$, also denoted by $\partial_i^{\bq}$.

\subsection{Weyl groupoids}\label{subsec:weyl-gpd}The notions of Weyl groupoid and generalized root systems 
were introduced in \cite{H-Weyl gpd,HY-groupoid}. We recall the main features needed later. 
Let  $(c_{ij}^{\bq})_{i,j\in \I}\in\Z^{\I\times\I}$ be the (generalized Cartan) matrix
defined by $c_{ii}^{\bq} := 2$ and
\begin{align}\label{eq:defcij}
c_{ij}^{\bq}&:= -\min \left\{ n \in \N_0: (n+1)_{q_{ii}}
(1-q_{ii}^n q_{ij}q_{ji} )=0 \right\},  & i & \neq j.
\end{align}
Let $i\in \I$. 
First, the reflection $s_i^{\bq}\in GL(\Z^\I)$ is given by 
\begin{align}\label{eq:siq-definition}
s_i^{\bq}(\alpha_j)&:=\alpha_j-c_{ij}^{\bq}\alpha_i, & &j\in \I.
\end{align}
Second, the matrix  $\rho_i(\bq)$ is given by
\begin{align}\label{eq:rhoiq-definition}
(\rho_i (\bq))_{jk}&:= \bq(s_i^{\bq}(\alpha_j),s_i^{\bq}(\alpha_k)) = q_{jk} q_{ik}^{-c_{ij}^{\bq}} q_{ji}^{-c_{ik}^{\bq}} q_{ii}^{c_{ij}^{\bq}c_{ik}^{\bq}}, & &j, k \in \I.
\end{align}
Finally, the braided vector space $\rho_i(V)$ is of diagonal type with matrix  $\rho_i (\bq)$. Set
\begin{align*}
\cX := \{\rho_{j_1} \dots \rho_{j_n}(\bq): j_1, \dots, j_n \in \I, n \in \N \}.
\end{align*}

\medbreak
The set $\cX$ is called the Weyl-equivalence class of $\bq$.
The set $\varDelta_+^{\bq}$ of \emph{positive roots} consists of the $\Z^{\I}$-degrees of the generators of a PBW-basis of $\toba_\bq$, counted with multiplicities.
Let $\varDelta^{\bq} := \varDelta_+^{\bq} \cup -\varDelta_+^{\bq}$. 
Then the  generalized  root system of $\bq$ is the fibration $\Delta \to  \cX$, where the fiber of $\rho_{j_1} \dots \rho_{j_N}(\bq)$ is $\Delta^{\rho_{j_1} \dots \rho_{j_N}(\bq)}$. The Weyl groupoid  $\cW_{\bq}$
of $\toba_\bq$ 
acts on this fibration, generalizing the classical Weyl group.
Here is another characterization of  $\varDelta_+^{\bq}$, valid because it is finite.
Let $\omega_0^{\bq} \in \cW_{\bq}$ be an element of maximal length and $\omega_0^{\bq}=\sigma_{i_1}^{\bq} \sigma_{i_2}\cdots \sigma_{i_\ell}$ 
be a reduced expression. Then 
\begin{align} \label{eq:betak}
\beta_k &:= s_{i_1}^{\bq}\cdots s_{i_{k-1}}(\alpha_{i_k}), & 
k\in\I_{\ell}
\end{align}
are pairwise different vectors and  $\varDelta_+^{\bq}=\{\beta_k : k\in\I_{\ell}\}$  \cite[Prop. 2.12]{CH-grpd-3-objects}, so $\vert \varDelta_+^{\bq} \vert = \ell$.

\subsection{Cartan roots  \cite{Ang-dpn}}\label{subsec:cartan-roots}
This important notion is crucial for our purposes. First, $i\in\I$ is a \emph{Cartan vertex} of $\bq$  if
\begin{align}\label{eq:cartan-vertex}
q_{ij}q_{ji} &= q_{ii}^{c_{ij}^{\bq}}, & \text{for all } j \neq i.
\end{align}
Then the set of \emph{Cartan roots} of $\bq$ is
\begin{align*}
\fO^{\bq} &= \{s_{i_1}^{\bq} s_{i_2} \dots s_{i_k}(\alpha_i) \in \varDelta^{\bq}:
i\in \I  \text{ is a Cartan vertex of } \rho_{i_k} \dots \rho_{i_2}\rho_{i_1}(\bq) \}.
\end{align*}

Set $\fO^{\bq}_+ = \fO^{\bq} \cap \N_{0}^{\theta}$. Recall \eqref{eq:notation-qab}
and set $\widetilde N_\beta :=  N_{\beta}$,  if $\beta\notin\fO^{\bq}$,
or else $\infty$ if $\beta\in\fO^{\bq}$. 

\medbreak
The set of Cartan roots gives rise to a root system up to a rescaling. Set
\begin{align}\label{eq:root-system-distinguished}
\wfO^{\bq} &= \{N^{\bq}_{\beta}\beta: \beta \in \fO^{\bq} \}, &  \wfO^{\bq}_+ &= \wfO^{\bq} \cap \N_{0}^{\theta},&
\wbeta &= N^{\bq}_{\beta} \beta, \ \beta \in \fO^{\bq}.
\end{align}

\begin{theorem}\label{thm:Cartan-roots} \cite[Theorem 3.6]{AAR3} The set $\wfO^{\bq}$ is either empty or 
a root system inside the real vector space generated by $\fO^{\bq}$. 
The set $\varPi^{\bq}$ of all indecomposable elements of $\wfO^{\bq}_+$ is a basis of this root system. \qed
\end{theorem}

Here $\wgamma \in \wfO^{\bq}_+$ is 
\emph{indecomposable} if it can not be represented as a non-trivial positive linear combination of elements of  $\wfO^{\bq}_+$.
Let $\g_{\bq}$ be either $0$ or the semisimple Lie algebra with root system as in Theorem \ref{thm:Cartan-roots}, accordingly.
We fix  a triangular decomposition 
\begin{align}
\label{eq:triangular}
\g_{\bq} &= \n_\bq^+ \oplus \h_{\bq} \oplus \n_\bq^-
\end{align}
and the Borel subalgebras $\bg_\bq^{\pm} = \h_{\bq} \oplus \n_{\bq}^\pm \subset \g_{\bq}$;
if $\g_{\bq} = 0$, then $\n_\bq^+ = \h_{\bq} = \n_\bq^- = 0$.
We denote the root lattice of $\g_{\bq}$ by
\begin{align}\label{eq:Cartan-root-lattice}
\QQ_{\bq} &:= \sum_{\gamma \in \wfO^{\bq}_+} \Z \gamma 
= \bigoplus_{\gamma \in \varPi^{\bq}} \Z \gamma.
\end{align}

\subsection{Distinguished pre-Nichols algebras}
\label{subsec:dpn}
The finite-dimensional Nichols algebras of diagonal type admit distinguished pre-Nichols algebras  introduced in \cite{A-presentation,Ang-dpn}. 
An ideal $\II(V)$ of $T(V)$ was introduced in \cite{Ang-dpn};
it is generated by all the defining relations of $\toba_{\bq}$ in \cite[Theorem 3.1]{A-presentation}, 
but excluding the power root vectors $x_\alpha^{N_\alpha}$, $\alpha\in\wfO_{\bq}$,
and adding some quantum Serre relations, see \cite{Ang-dpn} for the precise list of relations.

\begin{definition} \cite{Ang-dpn}
The distinguished pre-Nichols algebra $\dpn_{\bq}$ of $V$ is  the quotient
$\dpn_{\bq}=T(V)/\II(V)$. Since $\II(V)$ is a Hopf ideal, $\dpn_{\bq}$ is a braided Hopf algebra.
\end{definition}

 By Remark \ref{rem:sig-partial}, there are automorphisms $\varsigma_i^{\bq}$
and skew-derivations  $\partial_i^{\bq}$ of $\dpn_{\bq}$, $i \in \I$.

\subsection{Lusztig algebras}\label{subsubsec:Luszig-algs}

The \emph{Lusztig algebra} $\luq$ associated to $\bq$ is the graded dual of $\dpn_{\bq}$ \cite{AAR-MRL}.
Thus $\luq$ is a braided Hopf algebra  equipped with a bilinear form $( \, , \, ) : \luq \times \dpn_{\bq} \rightarrow \ku$,
which satisfies 
\begin{align}
\label{eq:pair-LqBq}
&( y, xx') = ( y^{(2)},x ) ( y^{(1)},x' ) \qquad \mbox{ and }
\qquad ( yy', x ) = ( y,x^{(2)} ) ( y',x^{(1)} )
\end{align}
for all $x,x' \in \dpn_{\bq}$, $y,y' \in \luq$.
Let $Z_{\bq}^+ = \,^{\co\vpi} \dpn_{\bq}$ be the subalgebra of coinvariants of the canonical projection
$$\vpi: \dpn_{\bq} \to \toba_{\bq}.$$ 
Then  $Z_{\bq}^+$ is a \emph{normal} Hopf subalgebra of $\dpn_{\bq}$ \cite[Theorem 29]{Ang-dpn} and 
we have an extension of braided Hopf algebras
$Z_{\bq}^+ \overset{\iota}{\hookrightarrow} \dpn_{\bq} \overset{\vpi}{\twoheadrightarrow} \toba_{\bq}$.
Taking graded duals, we obtain a new extension of braided Hopf algebras,  cf. \cite[Proposition 3.2]{AAR-Belgium}:
\begin{align}\label{eq:extension-braided-lu}
\toba_{\bq^t} \overset{\vpi^*}{\hookrightarrow} \luq \overset{\iota^*}{\twoheadrightarrow} \zz_{\bq},
\end{align}

\begin{remark}\label{rem:Lusztigalg-lie} Assume that \eqref{eq:condition-cartan-roots} below holds. Then 
the braided Hopf algebra $\zz_\bq$ is a Hopf algebra, isomorphic to the enveloping algebra 
of the Lie algebra $\cP(\zz_{\bq})$ \cite[3.3]{AAR-Belgium}.
Moreover $\cP(\zz_{\bq})\simeq  \n_\bq^-$ as in \eqref{eq:triangular} \cite{AAR3}. 
\end{remark}

\section{Large quantum groups}\label{sec:large}
In this section we  describe the \emph{large quantum groups}, i.e., Drinfeld doubles of  
bosonizations of the distinguished pre-Nichols algebras belonging to a one-parameter family;
these are the main focus of the paper. 
The large quantum Borel and unipotent subalgebras are also introduced here.
 Throughout the rest of the paper $\Gamma^{+}$ and $\Gamma^{-}$ denote free abelian groups of rank $\theta$ with bases denoted respectively $(K_i)_{i\in \I}$ and 
$(L_i)_{i\in \I}$. Let $\Gamma = \Gamma^{+} \times \Gamma^{-}$.

\subsection{Families of Nichols algebras}\label{subsec:prenichols-families}
From now on we assume that  $\bq$ belongs to a one-parameter family (except when explicitly stated otherwise). 
This means that there exists an indecomposable matrix
\begin{align}\label{eq:braiding-matrix-parameter}
\bqf = ( \qf_{ij}) \in   \big( \Cset[\nu^{\pm 1}]^\times \big)^{\I \times \I}
\end{align}
such that: 
\begin{itemize}[leftmargin=*]\renewcommand{\labelitemi}{$\circ$}
\item The Nichols algebra of the $\Cset(\nu)$-braided vector space of diagonal type $V_{\Cset(\nu)}$ with basis $(x_i)_{i\in \I}$ and braiding matrix \eqref{eq:braiding-matrix-parameter} has finite root system;
thus $\bqf$ is in the list of  \cite{H-classif RS}. 

\item There exists an open subset $\emptyset \neq O \subseteq \ku^{\times}$ such that for any $x \in O$, the matrix $\bqf(x)$ obtained by
evaluation $\nu \mapsto x$ has the same finite root system as $\bqf$.

\item There exists $\xi \in \G_{\infty}' \cap O$ such that $\bq = \bqf(\xi)$.
\end{itemize}

By inspection in \cite{H-classif RS}, all one-parameter families are listed in the Appendix \ref{sec:Nichols-specialization}.
We denote the Nichols algebras of $V$ and $V_{\Cset(\nu)}$, with braidings given by $\bq$, respectively $\bqf$, by
\begin{align*}
\toba_{\bq} &:= \toba(V)&  &\mbox{and}& \toba_{\bqf} &:= \toba(V_{\Cset(\nu)}).
\end{align*}
The defining relations and PBW-basis of $\toba_{\bq}$ and $\toba_{\bqf}$ are described in \cite{AA-diag-survey} 
over an algebraically closed  field of characteristic 0 but the same presentation and PBW-basis are valid over $\Cset(\nu)$. Indeed, apply to $\Fset = \Cset(\nu)$, $\Kset = \overline{\Cset(\nu)}$ the following remarks:

\begin{itemize}[leftmargin=*]\renewcommand{\labelitemi}{$\circ$}
\medbreak\item Let $\Kset/\Fset$ be a field extension and $(V,c)$ a braided $\Fset$-vector space. Then $(V\otimes_{\Fset} \Kset,c \otimes \id)$ is a braided $\Kset$-vector space  
and $\toba(V) \otimes_{\Fset} \Kset \simeq \toba(V \otimes_{\Fset} \Kset)$; use e.g.  quantum symmetrizers.

\medbreak\item Let $\Kset/\Fset$ be a faithfully flat extension of commutative rings.
Let $U$ be a $\Fset$-algebra with generators $(y_j)_{j\in J}$ and $U_{\Kset} = U \otimes_{\Fset} \Kset$ which is also generated by $(y_j)_{j\in J}$. 
Let $(r_{t})_{t\in T}$ be a set of elements in the tensor algebra over $\Fset$ of the free module $\Fset^{(J)}$. Then these are
defining relations of $U$ if and only if  they are defining relations of  $U_{\Kset}$.

\end{itemize}

\bigbreak

The discussions in \S \ref{subsec:diag-type} and \S \ref{subsec:weyl-gpd} apply to the matrix $\bqf$.
Let $\bqf:\Z^{\I}\times\Z^{\I}\to (\Cset[\nu^{\pm 1}])^\times$ as in  \S \ref{subsec:diag-type}; we also have the notation
$\bqf_{\alpha\beta}$ for $\alpha,\beta  \in \Z^{\I}$ as in \eqref{eq:notation-qab}. 
We denote by $\cW_{\bqf}$ the corresponding  Weyl groupoid, 
by $\rho_i (\bq)$ the related braiding matrices, etc.
As in  Remark \ref{rem:sig-partial}, there are  
$\varsigma_i^{\bqf} \in \Aut_{\rm{alg}}(\toba_{\bqf})$ and  $(\id,\varsigma_i^{\bqf})$-derivations
$\partial_i^{\bqf}:\toba_{\bqf}\to \toba_{\bqf}$, for every $i\in\I$.

\begin{remark}\label{rem:crucial} Let $\beta \in \Delta^{\bqf}$. Crucially,
$\beta$ is a Cartan root of $\bq$ if and only if $\ord\bqf_{\beta\beta}=\infty$.
\end{remark}

\subsection{The quantum group $U_\bqf$}\label{subsec:large-bqf} 
Here we work over $\Cset(\nu)$.
Let $W_{\Cset(\nu)}$ the $\Cset(\nu)$-vector space with basis $(y_i)_{i\in \I}$.
The group $\Gamma$ acts on $V_{\Cset(\nu)} \oplus W_{\Cset(\nu)}$ by 
\begin{align}\label{eq:double-action}
K_i \cdot x_j &= \qf_{ij} x_j,&  K_i \cdot y_j &= \qf_{ij}^{-1} y_j, & L_i \cdot x_j &= \qf_{ji} x_j,& L_i \cdot y_j &= \qf_{ji}^{-1} y_j,
\end{align}
$i,j \in \I$. The vector space $V_{\Cset(\nu)} \oplus W_{\Cset(\nu)}$ is $\Gamma$-graded by
\begin{align}\label{eq:double-grading}
\deg x_i &= K_i, & \deg y_i &= L_i, & i &\in \I.
\end{align}
Thus $V_{\Cset(\nu)} \oplus W_{\Cset(\nu)} \in \ydgcq$ with coaction given by the grading. In particular, 
$W_{\Cset(\nu)}$ is a braided vector space with braiding matrix $\bqf'$ where 
$\bqf'_{ij} = \bqf_{ji}^{-1}$, $i,j \in \I$.

We define $U_\bqf$ as the quotient Hopf algebra
of the bosonization $T(V_{\Cset(\nu)} \oplus W_{\Cset(\nu)}) \# \Cset(\nu)\Gamma$ modulo the ideal
generated by
\begin{align*}
&\JJ(V_{\Cset(\nu)}),&
&\JJ(W_{\Cset(\nu)}), &
&x_iy_{j} - \qf_{ij}^{-1} y_{j}x_i - \delta_{ij} (K_iL_i -
1), &  i,j &\in \I.
\end{align*}
The images of $x_i$, $y_i$, $K_i$ and $L_i$ in $U_\bqf$ will again be denoted by the same symbols. Let
$E_i:=x_{i}$, $F_i:=y_iL_{i}^{-1}$ in $U_\bqf$, $i \in \I$.
Then for all $i, j \in \I$ we have
\begin{align}
\label{eqn:g-con-Ei} K_iE_j &= \qf_{ij}E_jK_i,& L_iE_j &= \qf_{ji} E_jL_i,
\\\label{eqn:g-con-Fi}
K_iF_j &= \qf_{ij}^{-1}F_jK_i,& L_iF_j &= \qf_{ji}^{-1} F_jL_i,
\\ \label{eq:Ubqf-linking-EF}
E_iF_j - F_j E_i &= \delta_{ij} (K_i -L_i^{-1}),
\\ \label{eqn:comul-Ei}
\Delta(E_i) &= K_i \otimes E_i + E_i \otimes 1,&
\Delta(F_i) &=1 \otimes F_i + F_i \otimes L_i^{-1}.
\end{align}

We consider the following subalgebras of $U_\bqf$:
\begin{align*}
U_\bqf^{+0} &= \Cset(\nu) [K_i^{\pm 1}: i\in \I], & 
U_\bqf^{-0} &= \Cset(\nu) [L_i^{\pm 1}: i\in \I], \quad
U_\bqf^{0} = \Cset(\nu) [K_i^{\pm 1}, L_i^{\pm 1}: i\in \I],
\\
U_\bqf^{+} &= \Cset(\nu)\langle E_i: i\in \I \rangle, 
& U_\bqf^{-} &= \Cset(\nu)\langle F_i: i\in \I \rangle,
\\
U_\bqf^{\geqslant} &= \Cset(\nu)\langle E_i, K_i^{\pm1} : i\in \I \rangle, 
& U_\bqf^{\leqslant} &= \Cset(\nu)\langle F_i, L_i^{\pm 1}: i\in \I \rangle.
\end{align*}
The multiplication map induces linear isomorphisms
\begin{align}
\label{eq:triang-field-family}
U_\bqf \simeq U_\bqf^+ \otimes_{\Cset(\nu)} U_\bqf^0 \otimes_{\Cset(\nu)} U_\bqf^-
\simeq U_\bqf^{\geqslant} \otimes_{\Cset(\nu)} U_\bqf^{\leqslant}.
\end{align}

We have canonical isomorphisms of Hopf algebras
\begin{align*}
U_\bqf^{+0} & \simeq \Cset(\nu) \Ga^+, &
U_\bqf^{-0} & \simeq \Cset(\nu) \Ga^-,    &
U_\bqf^{0} & \simeq \Cset(\nu) \Ga. 
\end{align*}
The algebra $U_\bqf^{+}$ has a canonical structure of a Hopf algebra in $\ydgcqp$ and there are
isomorphisms of (braided) Hopf algebras
\begin{align*}
U_\bqf^{+} & \simeq \toba_{\bqf},&
U_\bqf^{\geqslant} &\simeq U_\bqf^{+}  \# U_\bqf^{+0},
\end{align*}
see e.g. \cite{ARS} for details.

Define the module $V_{\Cset(\nu)}^* \in \ydgcqm$ with basis $\{ x_i^* : i \in \I \}$ by
\begin{align*}
x_j^* \cdot L_i &= \qf_{ji} x_j^*,&  
\deg x_i^* = L_i^{-1},  && i, j \in \I. 
\end{align*}
Let $\pi^{-}: U_\bqf^{\leqslant} \to  U_\bqf^{-0}$ be the canonical Hopf algebra morphism; then
${}^{\co \pi^{-}}U_\bqf^{\leqslant} = U_\bqf^{-}$, cf. \cite[Corollary 3.9 (2)]{ARS}.
Hence  $U_\bqf^{-}$ has a canonical structure of a Hopf algebra in $\ydgcqm$.
By Remark \ref{rem:right-Nichols}, we have isomorphisms of (braided) Hopf algebras
\begin{align*}
U_\bqf^{-} & \simeq \toba (V_{\Cset(\nu)}^*) \simeq \toba_{\bqf^{(-1)}},&
U_\bqf^{\leqslant} &\simeq  U_\bqf^{-0} \# U_\bqf^{-}.
\end{align*}
Here $\bqf^{(-1)}$ means the matrix obtained by inverting every entry of $\bqf$. 

\medbreak
Now there is a unique Hopf skew-pairing 
$\langle \cdot , \cdot \rangle: U_\bqf^{\leqslant} \times U_\bqf^{\geqslant} \to \Cset(\nu)$ determined by
\begin{align*}
\langle L_i, K_j \rangle &= \bqf_{ji}^{-1},&
\langle F_i, E_j \rangle &= \delta_{ij},&
\langle L_i, E_j \rangle &= \langle F_i, K_j \rangle = 0,& i,j &\in \I,
\end{align*}
see \cite[Theorem 3.7]{ARS}. By  \cite[Theorem 3.11 (1)]{ARS}, we have
\begin{align*}
\langle x_-g_-, x_+g_+ \rangle &= \langle x_-, x_+ \rangle \langle g_-, g_+ \rangle,
& x_{\pm} &\in U_{\bqf}^{\pm},\; g_{\pm} \in \Gamma^{\pm}.
\end{align*}
The restriction $\langle \cdot , \cdot \rangle: U_\bqf^{-} \times U_\bqf^{+} \to \Cset(\nu)$
is non-degenerate by \cite[Theorem 3.11 (3)]{ARS} and is a 
Hopf skew-pairing of braided Hopf algebras by Lemma \ref{lemma:bilfor-bosonization}. 

\subsection{The large quantum group $U_\bq$}\label{subsec:Ubq}

Recall that $\bq  \in \big( \Cset^\times \big)^{\I \times \I}$  belongs to a one parameter family
given by a matrix $\bqf$, cf. \S \ref{subsec:prenichols-families}.

\begin{definition}\label{def:largeQG}
The \textbf{large quantum group} $U_\bq$ is the  Drinfeld double of the bosonization 
of the distinguished pre-Nichols algebra $\dpn_{\bq}$.
\end{definition}

The complex Hopf algebra $U_\bq$ was  defined  in \cite{Ang-dpn} for arbitrary $\bq$ with $\dim \toba_{\bq} < \infty$. 
Explicitly, let $W$ be the $\Cset$-vector space with basis $(y_i)_{i\in \I}$.
The group $\Gamma$ acts on $V  \oplus W$ by 
\begin{align*}
K_i \cdot x_j &= q_{ij} x_j,&  K_i \cdot y_j &= q_{ij}^{-1} y_j, & L_i \cdot x_j &= q_{ji} x_j,& L_i \cdot y_j &= q_{ji}^{-1} y_j,& i,j &\in \I.
\end{align*}
Now $V \oplus W$ is $\Gamma$-graded by
\eqref{eq:double-grading}, so $W$ is a braided vector space with braiding matrix $\bq'$ with entries
$q'_{ij}= q_{ji}^{-1}$ for $i,j \in \I$.
Recall the defining ideal $\II(V)$ of  $\dpn_{\bq}$.
Then  $U_\bq$ is the bosonization $T(V \oplus W) \# \Cset \Gamma$ modulo the ideal
generated by
\begin{align*}
 &\II(V),&  &\II(W),
&x_iy_{j} - q_{ij}^{-1} y_{j}x_i &- \delta_{ij} (K_iL_i -
1),&  i,j &\in \I.
\end{align*}

The images of $x_i$, $y_i$, $K_i$ and $L_i$ in $U_\bq$ will again be denoted by the same symbols. Let
$\Eq_i=x_{i}$, $\Fq_i=y_iL_{i}^{-1}$ in $U_\bq$, $i \in \I$.
Then for all $i, j \in \I$ we have
\begin{align}
\label{eqn:g-con-ei} K_i\Eq_j &= q_{ij}\Eq_jK_i,& L_i\Eq_j &= q_{ji} \Eq_jL_i,
\\\label{eqn:g-con-fi}
K_i\Fq_j &= q_{ij}^{-1}\Fq_jK_i,& L_i\Fq_j &= q_{ji}^{-1} \Fq_jL_i,
\\ \label{eq:Ubq-linking-ef}
\Eq_i\Fq_j - \Fq_j \Eq_i &= \delta_{ij} (K_i -L_i^{-1}),
\\ \label{eqn:comul-ei}
\Delta(\Eq_i) &= K_i \otimes \Eq_i + \Eq_i \otimes 1,&
\Delta(\Fq_i) &=1 \otimes \Fq_i + \Fq_i \otimes L_i^{-1}.
\end{align}

We consider the following subalgebras of $U_\bq$:
\begin{align*}
U_\bq^{+0} &= \Cset  [K_i^{\pm 1}: i\in \I], & 
U_\bq^{-0} &= \Cset  [L_i^{\pm 1}: i\in \I], &
U_\bq^{0} &= \Cset  [K_i^{\pm 1}, L_i^{\pm 1}: i\in \I],
\\
U_\bq^{+} &= \Cset \langle \Eq_i: i\in \I \rangle, 
& U_\bq^{-} &= \Cset \langle \Fq_i: i\in \I \rangle,
\\
U_\bq^{\geqslant} &= \Cset\langle \Eq_i, K_i^{\pm1} : i\in \I \rangle, 
& U_\bq^{\leqslant} &= \Cset\langle \Fq_i, L_i^{\pm 1}: i\in \I \rangle.
\end{align*}
\begin{definition}\label{def:largeBorel-unipotent}
The algebras $U_\bq^{\geqslant}$ and $U_\bq^{\leqslant}$ will be called \textbf{large quantum Borel algebras} and the algebras $U_\bq^{\pm}$  \textbf{large quantum unipotent algebras}. 
\end{definition}

The multiplication map induces the linear isomorphisms
\begin{align}
\label{eq:triang-field}
U_\bq \simeq U_\bq^+ \otimes_{\Cset} U_\bq^0 \otimes_{\Cset} U_\bq^{(-)} \simeq U_\bq^{\geqslant} \otimes_{\Cset} U_\bq^{\leqslant}.
\end{align}
We have canonical isomorphisms of Hopf algebras
\begin{align*}
U_\bq^{+0} & \simeq \Cset \Ga^+, &
U_\bq^{-0} & \simeq \Cset \Ga^-,    &
U_\bq^{0} & \simeq \Cset \Ga. 
\end{align*}
The algebra $U_\bq^{+}$ has a canonical structure of a Hopf algebra in $\ydgcp$.
We have  isomorphisms of (braided) Hopf algebras:
\begin{align*}
U_\bq^{+} & \simeq \dpn_{\bq},&
U_\bq^{\geqslant} &\simeq U_\bq^{+}  \# U_\bq^{+0},
\end{align*}
see \cite{Ang-dpn}. Define the module $V^* \in \ydgcm$ with basis $\{ x_i^* : i \in \I \}$ by
\begin{align*}
x_j^* \cdot L_i &= \bq_{ji} x_j^*,&  
\deg x_i^* = L_i^{-1},  && i, j \in \I. 
\end{align*}
Let $\pi^{-}: U_\bq^{\leqslant} \to  U_\bq^{-0}$ be the canonical Hopf algebra projection; 
then ${}^{\co \pi^{-}}U_\bq^{\leqslant} = U_\bq^{-}$ as in \cite[Corollary 3.9 (2)]{ARS}.
Hence  $U_\bq^{-}$ is a Hopf algebra in $\ydgcm$ and because of the defining relations of 
$U_\bq^{-}$, it
is isomorphic to the distinguished pre-Nichols algebra of $V^* \in \ydgcm$.
Combining the above, we get isomorphisms of (braided) Hopf algebras:
\begin{align}\label{eq:tobaq-1}
U_\bq^{-} & \simeq \dpn_{\bq^{(-1)}},&
U_\bq^{\leqslant} &\simeq  U_\bq^{-0} \# U_\bq^{-}.
\end{align}
Here, again, $\bq^{(-1)}$ denotes the matrix obtained by inverting every entry of $\bq$.

\subsection{Lusztig isomorphisms and root vectors}\label{subsec:lusztig-isom}
As in \cite[\S 3]{H-Lusztig} we consider
\begin{align}\label{eq:lambdaij-defn}
\lambda_{ij}^{\bqf} &=(\qf_{ii}^{-c_{ij}^{\bqf}} \qf_{ij}\qf_{ji})^{c_{ij}^{\bqf}} (-c_{ij}^{\bqf})_{\qf_{ii}}^!
\prod_{0\le s<-c_{ij}^{\bqf}} (\qf_{ii}^s \qf_{ij}\qf_{ji}-1) \in \Cset[\nu^{\pm 1}]^\times,
& i & \neq j \in\I. 
\end{align}
Notice that $\lambda_{ij}^{\bqf}  \neq 0$ by the definition \eqref{eq:defcij}. 

\medbreak
By \cite[Proposition 6.8]{H-Lusztig}, there exist algebra isomorphisms $T_i^{\bqf}:U_{\rho_i (\bqf)} \to U_{\bqf}$ such that
\begin{align}\label{eq:lusztig-isom-defn}
\begin{aligned}
T_i^{\bqf}(\underline K_j) &=  K_j K_i^{-c_{ij}^{\bqf}};
&
T_i^{\bqf}(\underline E_i) &= \begin{cases} 
F_i L_i, & j=i, \\ (\ad_c E_i)^{-c_{ij}^{\bqf}} E_j, & j\neq i,
\end{cases}
\\
T_i^{\bqf}(\underline L_j) &= L_j L_i^{-c_{ij}^{\bqf}};
&
T_i^{\bqf}(\underline F_i) &= \begin{cases} 
K_i^{-1} E_i, & j=i, \\ 
(\lambda_{ij}^{\bqf})^{-1} (\ad_c F_i)^{-c_{ij}^{\bqf}} F_j, & j\neq i,
\end{cases}
\end{aligned}
\end{align}
where the underlined letters denote the generators of  $U_{\rho_i (\bqf)}$.

\medbreak

Let $\omega_0^{\bqf}$ be the element of $\cW_{\bqf}$ 
of maximal lenght ending at $\bqf$ and $\omega_0^{\bqf}=\sigma_{i_1}^{\bqf}\sigma_{i_2}\cdots \sigma_{i_\ell}$ be a reduced expression. By 
\cite[Theorem 6.20]{H-Lusztig}, we have that
\begin{align}\label{eq:Ebeta-defn}
E_{\beta_k}&:=T_{i_1}^{\bqf} \dots T_{i_{k-1}} (E_{i_k}) \in U_{\bqf}^+, & 
F_{\beta_k}&:=T_{i_1}^{\bqf} \dots T_{i_{k-1}} (F_{i_k}) \in U_{\bqf}^-, &
&k \in \I_{\ell}.
\end{align}
By \cite[Theorem 4.5]{HY-shapov} the sets
\begin{align}\label{eq:PBW-basis-U+-U-}
\left\{  E_{\beta_1}^{n_1}E_{\beta_2}^{n_2} \dots E_{\beta_\ell}^{n_\ell} : 0\le n_j< \widetilde{N}_{\beta_j}, j \in \I_{\ell} \right\}, \,
\left\{  F_{\beta_1}^{m_1}F_{\beta_2}^{m_2} \dots F_{\beta_\ell}^{m_\ell} : 0\le m_j< \widetilde{N}_{\beta_j}, j \in \I_{\ell} \right\}
\end{align}
are bases of $U_{\bqf}^+$ and $U_{\bqf}^-$, respectively. Indeed, this follows from Property \ref{item:property-c} in the Appendix \ref{sec:Nichols-specialization} and Remark \ref{rem:crucial}. Thus the following set is a basis of $U_{\bqf}$:
\begin{align}\label{eq:PBW-basis-U}
\{ E_{\beta_1}^{m_1} \dots E_{\beta_\ell}^{m_\ell} K_1^{a_1} \dots K_{\theta}^{a_{\theta}}  L_1^{b_1} \dots L_{\theta}^{b_{\theta}} F_{\beta_1}^{n_1} \dots F_{\beta_\ell}^{n_\ell}:
0\le m_j, \, n_j< \widetilde{N}_{\beta}, \, a_i, \, b_i\in\Z \}.
\end{align}

\medbreak
We now turn to the algebras $U_{\bq}$. Let $\lambda_{ij}^{\bq}$ is defined as \eqref{eq:lambdaij-defn} with $\bq$ in place of $\bqf$. Notice that $\lambda_{ij}^{\bq} \neq 0$ by the definition \eqref{eq:defcij}. 
By \cite[Proposition 10]{Ang-dpn}, there exist algebra isomorphisms $T_i^{\bq}:U_{\rho_i (\bq)} \to U_{\bq}$ such that
\begin{align}\label{eq:lusztig-isom-defn-preNichols}
\begin{aligned}
T_i^{\bq}(\underline K_j) &= K_j K_i^{-c_{ij}^{\bq}};
&
T_i^{\bq}(\underline \Eq_i) &= \begin{cases} \Fq_i L_i, & j=i, \\ (\ad_c \Eq_i)^{-c_{ij}^{\bq}} \Eq_j, & j\neq i,
\end{cases}
\\ \\
T_i^{\bq}(\underline L_j) &= L_j L_i^{-c_{ij}^{\bq}};
&
T_i^{\bq}(\underline \Fq_i) &= \begin{cases} K_i^{-1} \Eq_i, & j=i, \\ 
(\lambda_{ij}^{\bq})^{-1} (\ad_c \Fq_i)^{-c_{ij}^{\bq}} \Fq_j, & j\neq i,
\end{cases}
\end{aligned}
\end{align}
The underlined letters denote the generators of $U_{\rho_i (\bq)}$.

\medbreak

Analogously, $\Eq_{\beta_k}=T_{i_1}^{\bq} \dots T_{i_{k-1}} (\Eq_{i_k})$ and $\Fq_{\beta_k}=T_{i_1}^{\bq} \dots T_{i_{k-1}} (\Fq_{i_k})$ belong to $U_{\bq}^+$ and $U_{\bq}^-$, respectively and by \cite[Theorem 11]{Ang-dpn} the sets
\begin{align}\label{eq:PBW-basis-preNicholsU+U-}
\big\{  \Eq_{\beta_1}^{n_1}\Eq_{\beta_2}^{n_2} \dots \Eq_{\beta_\ell}^{n_\ell} : 0\le n_i< \widetilde{N}_{\beta_i} \big\}
& &\text{and}
& &\big\{  \Fq_{\beta_1}^{m_1}\Fq_{\beta_2}^{m_2} \dots \Fq_{\beta_\ell}^{m_\ell} : 0\le m_j< \widetilde{N}_{\beta_j} \big\}
\end{align}
are bases of $U_{\bq}^+$ and $U_{\bq}^-$, respectively. Thus the following set is a basis of $U_{\bq}$:
\begin{align}\label{eq:PBW-basis-prenichols}
\big\{ \Eq_{\beta_1}^{m_1} \dots \Eq_{\beta_\ell}^{m_\ell} K_1^{a_1} \dots K_{\theta}^{a_{\theta}}  L_1^{b_1} \dots L_{\theta}^{b_{\theta}} \Fq_{\beta_1}^{n_1}  \dots \Fq_{\beta_\ell}^{n_\ell} : 0\le m_j, \, n_j< \widetilde{N}_{\beta_j}, \, a_i, \, b_i\in\Z \big\}.
\end{align}

\subsection{The central subalgebras $\Zc_{\bq}$, $\Zc_{\bq}^{\pm}$, $\Zc_{\bq}^{\geqslant}$, $\Zc_{\bq}^{\leqslant}$}\label{subsec:Zq}
\label{sec:Z}
In this subsection and the next $\bq$ does not need to be in a family, just $\dim \toba_{\bq} < \infty$ is assumed.
Set 
\begin{align}\label{eq:defn-lcm-Nbeta}
\Nt=\lcm \{N_{\beta}: \beta\in\varDelta^{\bq}_+\}.
\end{align}
To start with, we consider a subalgebra $\Zc_{\bq}$ of $U_{\bq}$.  Then $\Zc_{\bq}$ is generated by 
\begin{align}\label{eq:generators-Zc-Cartan-roots}
&\Eq_{\beta}^{N_{\beta}},& &\Fq_{\beta}^{N_{\beta}},& &K_{\beta}^{\pm N_{\beta}},& &L_{\beta}^{\pm N_{\beta}},& &\beta \in \fO^{\bq}_+,
\\\label{eq:generators-Zc-new}
&& && &K_{\beta}^{\pm\Nt},& &L_{\beta}^{\pm \Nt},& &\beta \in \varDelta^{\bq}_+;
\end{align}
this is a normal $\QQ_{\bq}$-graded Hopf subalgebra of $U_{\bq}$, \cite[Proposition 21, Theorem 33]{Ang-dpn}, which may be different from the one in \cite[p. 18]{Ang-dpn} 
since we add the generators in \eqref{eq:generators-Zc-new} what actually only affects the types $\supera{k-1}{\theta - k}$,
see Proposition \ref{prop:Zq-families}.
These new generators are necessary for $U_{\bq}$ to be finitely generated as $Z_{\bq}$-module.  

\medbreak
The following subalgebras of $\Zc_{\bq}$ are also needed: 
\begin{align*}
\Zc_{\bq}^{+} &= \Cset \langle \Eq_{\beta}^{N_{\beta}}: \beta \in \fO^{\bq}_{+} \rangle, & 
\Zc_{\bq}^{-} &= \Cset \langle \Fq_{\beta}^{N_{\beta}}: \beta \in \fO^{\bq}_{+} \rangle.
\end{align*}
Notice that  $\Zc_{\bq}^+$ coincides with the subalgebra introduced right after \eqref{eq:pair-LqBq}, see \cite[Theorem 29]{Ang-dpn}.
For $\Zc_{\bq}$ to be central in $U_{\bq}$ we need the following condition that 
\emph{we assume from now on:}
\begin{align}\label{eq:condition-cartan-roots}
&q_{\alpha\beta}^{N_\beta}=1, & \alpha \in \varDelta^{\bq}, \; \beta\in\fO^{\bq}.
\end{align}

\begin{remark}\label{rem:condition-cartan-roots-simple}
\begin{enumerate}[leftmargin=*,label=\rm{(\alph*)}]
\item If \eqref{eq:condition-cartan-roots} holds, then  $q_{\beta\alpha}^{N_\beta}=1$ \cite[Lemma 24]{Ang-dpn}.

\medbreak
\item Condition \eqref{eq:condition-cartan-roots} is equivalent to the following one:
\begin{align}\label{eq:condition-cartan-roots-simple}
&q_{\alpha_i\beta}^{N_\beta}=1, & \text{ for all }i\in\I, &\; \wbeta\in\varPi^{\bq}.
\end{align}
\end{enumerate}
The reduction to simple roots is clear. Since $q_{\alpha\beta}^{N_\beta}=q_{\alpha\wbeta}$ 
and $\varPi^{\bq}$ is a basis of the root system $\fO^{\bq}$, the reduction from $\fO^{\bq}_{+}$ to $\varPi^{\bq}$  holds.

\medbreak
\begin{enumerate}[leftmargin=*,resume,label=\rm{(\alph*)}]
\item\label{item:condition-cartan-roots-weyl} Let $i\in\I$.
Condition \eqref{eq:condition-cartan-roots} holds for $\bq$ if and only if it holds for
$\rho_i (\bq)$.
\end{enumerate}

Indeed, $\rho_i(\bq)_{\alpha\beta}=\bq_{s_i^{\bq}(\alpha)s_i^{\bq}(\beta)}$ for all $\alpha,\beta\in\Z^{\theta}$ by \eqref{eq:rhoiq-definition},
and by \cite[Lemma 2.3]{AAR3} we have $s_i^{\bq}(\fO^{\bq})=\fO^{\rho_i(\bq)}$, $N^{\rho_i(\bq)}_{s_i^{\bq}(\beta)}=N^{\bq}_{\beta}$ for all $\beta$.
\end{remark}

\medspace

When $\bq$ is symmetric, we can quotient the large quantum by a central group subalgebra to remove the extra Cartan generators as in quantum groups.
However the condition of $\bq$ being symmetric is not always compatible with \eqref{eq:condition-cartan-roots} as we see next.

\begin{example}
Assume that $\bq$ has Dynkin diagram $\xymatrix{\overset{-1}{\circ} \ar@{-}[r]^\xi & \overset{-1}{\circ} }$, $\xi\in\G_N'$, $N>2$: it is of super type $\supera{1}{0}$. In this case,
\begin{align*}
\varDelta^{\bq}_+ &=\{\alpha_1, \alpha_1+\alpha_2, \alpha_2 \}, 
& \fO^\bq_+ &=\{\alpha_1+\alpha_2\}.
\end{align*}
Condition \eqref{eq:condition-cartan-roots-simple} becomes
\begin{align*}
1&=(q_{11}q_{12})^N = (-q_{12})^N, &
1&=(q_{21}q_{22})^N = (-q_{21})^N& \iff q_{12}^N=(-1)^N = q_{21}^N.
\end{align*}

We have two possibilities: if $N$ is even, then $q_{12}=\xi^k$ for some $k\in\I_{N}$, so $q_{21}=\xi^{1-k}$, and $\bq$ is not symmetric.
If $N$ is odd, then $q_{12}=-\xi^k$ for some $k\in\I_{N}$, so $q_{21}=-\xi^{1-k}$. In this case $\bq$ is symmetric only when $k=\frac{N+1}{2}$.
\end{example}

We consider also the Hopf subalgebras
\begin{align*}
\Zc_{\bq}^{+0} &= \Cset \langle \{K_{\beta}^{\pm N_{\beta}}: \beta \in \fO^{\bq}_{+}\} \cup \{K_{\beta}^{\pm\Nt}:\beta \in \varDelta^{\bq}_+\}\rangle, &  
\Zc_{\bq}^{\geqslant} &= \Zc_{\bq}^{+}\Zc_{\bq}^{+0}, 
\\
\Zc_{\bq}^{-0} &= \Cset \langle \{L_{\beta}^{\pm N_{\beta}}: \beta \in \fO^{\bq}_{+}\} \cup \{L_{\beta}^{\pm\Nt}:\beta \in \varDelta^{\bq}_+\}\rangle, & \ \Zc_{\bq}^{\leqslant} &= \Zc_{\bq}^{-}\Zc_{\bq}^{-0}.
\\
\Zc_{\bq}^{0} &= \Zc_{\bq}^{+0}\Zc_{\bq}^{-0}.
&
\end{align*}

\begin{remark}\label{rem:Zq-properties} The  following properties hold:

\begin{enumerate}[leftmargin=*,label=\rm{(\alph*)}]
\item  \cite[Th. 23]{Ang-dpn}. $\Zc_{\bq}^{\pm}$ is a polynomial ring in variables $\Eq_{\beta}^{N_{\beta}}$, respectively $\Fq_{\beta}^{N_{\beta}}$, $\beta \in \fO^{\bq}_{+}$. 

\medbreak
\item The multiplication gives linear isomorphisms
$\Zc_{\bq}^{+} \otimes \Zc_{\bq}^{+0} \otimes \Zc_{\bq}^{-0} \otimes \Zc_{\bq}^{-} \simeq \Zc_{\bq} \simeq\Zc_{\bq}^\geqslant \otimes \Zc_{\bq}^\leqslant$.

\medbreak\item  Recall the skew-deriva\-tions $\partial^{\bq}_i$, $\partial^{\bq^{(-1)}}_i$ of  $U_{\bq}^{\pm}$, cf.  \eqref{eq:tobaq-1}. By \cite[Theorem 31]{Ang-dpn},
\begin{align}\label{eq:Z+-intersection-ker-derivations}
\Zc_{\bq}^{+} & = \bigcap_{i\in\I} \ker \partial_i^{\bq}, & \Zc_{\bq}^{-} &= \bigcap_{i\in\I} \ker \partial^{\bq^{(-1)}}_i.
\end{align}

\medbreak\item The algebras $U_{\bq}$, $U_{\bq}^{\geqslant}$, $U_{\bq}^{\leqslant}$ and $U_{\bq}^{\pm}$
are module finite over their central subalgebras $Z_{\bq}$, $Z_{\bq}^{\geqslant}$, $Z_{\bq}^{\leqslant}$ and $Z_{\bq}^{\pm}$; just consider the PBW-bases in \S \ref{subsec:lusztig-isom}.
\end{enumerate}

\end{remark}

\subsection{Action of the Weyl groupoid on $\Zc_{\bq}$}

Next we prove invariance of the central Hopf subalgebras $\Zc_{\bq}$ under the Lusztig isomorphisms $T_i^{\bq}: U_{\rho_i (\bq)}\to U_{\bq}$,
cf. \S \ref{subsec:lusztig-isom}.

\begin{theorem}\label{th:Zq-invariant-under-Ti} Let $i\in\I$. Then
$T_i^{\bq}$ restricts to an algebra isomorphism $T_i^{\bq}:\Zc_{\rho_i(\bq)}\to \Zc_{\bq}$.
\end{theorem}
\pf
By \eqref{eq:Z+-intersection-ker-derivations}, $\Zc_{\bq}$ (defined in terms of the root vectors which  
depend on the expression of $\omega_0^{\bq}$)
is indeed indedependent of such expression; in particular we may choose $\omega_0^{\bq} = \sigma_{i_1}^{\bq} \dots \sigma_{i_\ell}$ such that $i_1=i$. For simplicity we set $\bp=\rho_i(\bq)$. As $\sigma_{i_2}^{\bp} \dots \sigma_{i_\ell}$ is reduced, we may extend it to a reduced expression of $\omega^{\bp}$ \cite[Corollary 3]{HY-groupoid}:
\begin{align*}
\omega_0^{\bp} &= \sigma_{i_2}^{\bp} \dots \sigma_{i_\ell} \sigma_j & \text{for some } & j\in\I.
\end{align*}
We set $\beta_k'=\sigma_{i_1}(\beta_k)=\sigma_{i_2}^{\bp} \dots \sigma_{i_{k-1}}(\alpha_{i_k})$, $k\in\I_{2,\ell}$.
Hence
\begin{align*}
\{ \beta_k': k\in\I_{2,\ell}  \}
&=s_i^{\bp}\left(\varDelta^{\bq}_+ -\{\alpha_i\} \right) =
\varDelta^{\bp}_+ -\{\alpha_i\}.
\end{align*}
As $\sigma_{i_2}^{\bp} \dots \sigma_{i_\ell}(\alpha_j)\in \varDelta^{\bp}_+$, $\sigma_{i_2}^{\bp} \dots \sigma_{i_\ell}(\alpha_j) \neq \beta_k'$ for $k\in\I_{2,\ell}$, 
we have that $\sigma_{i_2}^{\bp} \dots \sigma_{i_\ell}(\alpha_j)=\alpha_i$. Then $\{N_{\beta'}: \beta'\in\varDelta^{\bp}_+\}=\{N_{\beta}: \beta\in\varDelta^{\bq}_+\}$, so
$\lcm \{N_{\beta'}: \beta'\in\varDelta^{\bp}_+\}=\Nt$, and
\begin{align*}
T_i^{\bq} \big(K_{i}^{\pm \Nt}\big) & =K_i^{\mp \Nt} \in \Zc_{\bq}, & 
T_i^{\bq} \big(K_{\beta_k'}^{\pm \Nt}\big) & = 
K_{s_i^{\bq}(\beta_k')}^{\pm \Nt} = K_{\beta_k}^{\pm \Nt} \in \Zc_{\bq}, & &k>1,
\\
T_i^{\bq} \big(K_{i}^{\pm \Nt}\big) & =K_i^{\mp \Nt} \in \Zc_{\bq}, & 
T_i^{\bq} \big(K_{\beta_k'}^{\pm \Nt}\big) & = 
K_{s_i^{\bq}(\beta_k')}^{\pm \Nt} = K_{\beta_k}^{\pm \Nt} \in \Zc_{\bq}, & &k>1.
\end{align*}

\medbreak
Let $\beta\in\fO^{\rho_i(\bq)}$. If $\beta=\beta_k'$ for some $k\in\I_{2,\ell}$, then
$s_i^{\bq}(\beta_k')=\beta_k$ and $N_{\beta_k'}=N_{\beta_k}$, hence 
\begin{align*}
T_i^{\bq} \big(K_{\beta_k'}^{\pm N_{\beta_k'}}\big) & = K_{s_i^{\bq}(\beta_k')}^{\pm N_{\beta_k'}} = K_{\beta_k}^{\pm N_{\beta_k}} \in \Zc_{\bq},
&
T_i^{\bq}\big(\Eq_{\beta_k'}^{N_{\beta_k'}}\big) & = T_{i_1}^{\bq} T_{i_2} \dots T_{i_{k-1}}(\Eq_{i_k}^{N_{\beta_k}}) = \Eq_{\beta_k}^{N_{\beta_k}} \in \Zc_{\bq}.
\end{align*}
Otherwise $\beta=\alpha_i$, so $i$ is a Cartan vertex and
\begin{align*}
T_i^{\bq}(K_{\beta}^{\pm N_{\beta}}) & = K_{i}^{\mp N_{\alpha_i}} \in \Zc_{\bq},
&
T_i^{\bq}(\Eq_{\beta}^{N_{\beta}}) & = 
T_i^{\bq}(\Eq_{i}^{N_{\alpha_i}})= (\Fq_i L_i)^{N_{\alpha_i}} 
= q_{ii}^{\binom{N_{\alpha_i}}{2}} \Fq_i^{N_{\alpha_i}} L_i^{N_{\alpha_i}}
\in \Zc_{\bq}.
\end{align*}
Analogously, $T_i^{\bq}(L_{\beta}^{\pm N_{\beta}}), T_i^{\bq}(\Fq_{\beta}^{N_{\beta}})\in \Zc_{\bq}$ for all $\beta\in\fO_{\rho_i(\bq)}$, so $T_i^{\bq}(\Zc_{\rho_i (\bq)}) \subseteq \Zc_{\bq}$.
Applying $T_i^{\bp}$ we get the opposite inclusion.
\epf

Let $\Lambda'_{\bq}$ be the subgroup of $\Gamma$ generated by $K_{\beta}^{\pm N_{\beta}}$, $L_{\beta}^{\pm N_{\beta}}$, $\beta \in \fO^{\bq}_{+}$ and let $\Lambda_{\bq}$ be the subgroup generated by $\Lambda'_{\bq}$ and $K_{\beta}^{\pm\Nt}$, $L_{\beta}^{\pm\Nt}$, $\beta \in \varDelta^{\bq}_+$: we have that $\Zc_{\bq}^0=\Cset \Lambda_{\bq}$.

\medbreak
Next we check that $\Lambda'_{\bq}=\Lambda_{\bq}$ for $\bq$ as in the Appendix, not of type super $A$, so the subalgebra $\Zc_{\bq}$ is generated by \eqref{eq:generators-Zc-Cartan-roots} and  coincides with the one in \cite[p. 18]{Ang-dpn}.

\begin{proposition}\label{prop:Zq-families}
Assume that $\bq$ belongs to one of the families in the Appendix \ref{sec:Nichols-specialization}.
\begin{enumerate}[leftmargin=*,label=\rm{(\alph*)}]
\item If $\bq$ is not of type $\supera{k-1}{\theta - k}$, $k \in \I_{\lfloor\frac{\theta+1}{2} \rfloor}$, then $\Lambda'_{\bq}=\Lambda_{\bq}$ and $\Zc_{\bq}^0$ is generated by \eqref{eq:generators-Zc-Cartan-roots}.
\item Let $\bq$ be of type $\supera{k-1}{\theta - k}$, $k \in \I_{\lfloor\frac{\theta+1}{2} \rfloor}$, $\eta\in\varDelta_+^{\bq}-\fO^{\bq}_+$. Then 
$\Lambda_{\bq}\simeq \Lambda'_{\bq} \times \langle K_{\eta}^{\Nt}\rangle \times \langle L_{\eta}^{\Nt} \rangle$, and 
$\Zc_{\bq}^0$ is generated by $\{K_{\beta}^{\pm N_\beta}, L_{\beta}^{\pm N_\beta} : \beta \in \varDelta^{\bq}_+ \}$ 
and $K_{\eta}^{\pm\Nt}$, $L_{\eta}^{\pm\Nt}$.
\end{enumerate}
\end{proposition}
\pf
First we notice that $T_i^{\bq}$ restricts to group isomorphisms $\Lambda_{\bq}\simeq \Lambda_{\rho_i\bq}$, $\Lambda_{\bq}'\simeq \Lambda_{\rho_i\bq}'$ since the restrictions $(T_i^{\bq})_{|U_{\rho_i\bq}^{+0}}:U_{\rho_i\bq}^{+0}\to U_{\bq}^{+0}$ and $(T_i^{\bq})_{|U_{\rho_i\bq}^{-0}}:U_{\rho_i\bq}^{-0}\to U_{\bq}^{-0}$ are given by $s_i^{\bq}$.  Hence it is enough to consider one matrix for each Weyl equivalence class.

\medbreak
Assume that $\bq$ is as in the Appendix \ref{sec:Nichols-specialization} and is not of type $\supera{k-1}{\theta - k}$. If $\bq$ is
of Cartan type, then $\varDelta_+^{\bq}=\fO^{\bq}_+$, hence $\Lambda'_{\bq}=\Lambda_{\bq}$.
For the other types we check the equality $\Lambda'_{\bq}=\Lambda_{\bq}$ case-by-case when the Dynkin diagram is the one in Tables \ref{tab:Appendix-super} and \ref{tab:Appendix-modular}.

\medbreak
Let $\bq$ be of type $\supera{k-1}{\theta - k}$ with Dynkin diagram in Table \ref{tab:Appendix-super}, $\eta\in\varDelta_+^{\bq}-\fO^{\bq}_+$. Then $\fO^{\bq}_+=\{\alpha_{ij}|i\le j<k \text{ or }k<i\le j\}$, $N_{\beta}=N$ for all $\beta\in \fO^{\bq}_+$,
$\Lambda_{\bq}'=\langle K_i^{N}, L_i^{N}| i\ne k \rangle$, $\eta=\alpha_{ij}$ for some $i\le k\le j$ and
$\Nt=N$ if $N$ is even while $\Nt=2N$ if $N$ is odd. Hence $K_k^{\Nt}$, $L_k^{\Nt}$ belong to the subgroup generated by $\Lambda'_{\bq}$, $K_{\eta}^{\Nt}$ and $L_{\eta}^{\Nt}$.
On the other hand, $\Lambda_{\bq}$ is generated by $\Lambda'_{\bq}$, $K_{k}^{\Nt}$ and $L_{k}^{\Nt}$, thus the statement follows.
\epf

\section{The specialization setting for large quantum groups}
\label{sec:spec}
In this section we introduce the non-restricted integral form of  $U_{\bqf}$ and prove that the  
large quantum group $U_{\bq}$ is a specialization of it. We also introduce restricted integral forms of the subalgebras $U_{\bqf}^\pm$
and establish pairing results for the corresponding specializations. The latter integral forms will play a key role 
in our treatment of Poisson order structures on the large quantum groups $U_{\bq}$ and their Borel and unipotent subalgebras.

\subsection{Integral forms} In order to implement the ideas of Section \ref{sec:poisson},
we need to consider forms over suitable rings, generalizing \cite{DP}. For simplicity, we set
\begin{align}
\label{eq:A-alg}
\Acf &:= \Cset[ \nu^{\pm 1}, (\qf_{ii}^s \qf_{ij}\qf_{ji}-1)^{-1}: i\neq j \in \I, 0\le s<-c_{ij}^{\bqf}] \subset \Cset(\nu).
\end{align}
By (3.9), $\qf_{ii}^s\qf_{ij}\qf_{ji} \neq 0$ for $0\le s<-c_{ij}^{\bqf}$.
	
We now define the \emph{(non-restricted) integral forms} as the $\Acf$-subalgebras
\begin{align*}
U_{\bqf, \Acf}^+ &= \Acf \lcor E_i: i \in \I \rcor \subset  U_{\bqf}^+,&
U_{\bqf, \Acf}^0 &= \Acf [K_i^{\pm 1}, L_i^{\pm 1}: i \in \I] \subset  U_{\bqf}^0,
\\
U_{\bqf, \Acf}^- &= \Acf \lcor F_i: i \in \I \rcor \subset  U_{\bqf}^-,&
U_{\bqf, \Acf} &= \Acf \lcor K_i^{\pm1}, L_i^{\pm 1}, E_i, F_i: i \in \I \rcor \subset  U_{\bqf},
\\
U_{\bqf, \Acf}^{\geqslant} &= U_{\bqf, \Acf}^+  \otimes_{\Acf} \Acf [K_i^{\pm 1}: i \in \I], 
&
U_{\bqf, \Acf}^{\leqslant} &= U_{\bqf, \Acf}^-  \otimes_{\Acf} \Acf [L_i^{\pm 1}: i \in \I].
\end{align*}
These  are crucial for our purposes. We have again a triangular decomposition
\begin{align}
\label{eq:triang-int}
U_{\bqf, \Acf}^+ \otimes_{\Acf} U_{\bqf, \Acf}^0 \otimes_{\Acf} U_{\bqf, \Acf}^- &\simeq U_{\bqf,\Acf}.
\end{align}
The surjectivity 
of this multiplication map follows from the cross relations  \eqref{eqn:g-con-Ei}, 
\eqref{eqn:g-con-Fi} and \eqref{eq:Ubqf-linking-EF},  while the injectivity follows from  \eqref{eq:triang-field-family}. Recall \eqref{eq:lambdaij-defn} for the next result.

\begin{lemma}\label{lemma:lambdaij-in-Acf}
For all $i\neq j$ in $\I$, $(\lambda_{ij}^{\bqf})^{-1}\in\Acf$.
\end{lemma}
\noindent \emph{Proof.} 
If $\qf_{ii}^{-c_{ij}^{\bqf}}\qf_{ij}\qf_{ji}=1$, then using that $\bqf_{ii} \in \Cset[\nu^{\pm1}]^\times$ we have
\begin{align*}
(\lambda_{ij}^{\bqf})^{-1}&
= (-1)^{c_{ij}^{\bqf}} q_{ii}^{c_{ij}^{\bqf} (c_{ij}^{\bqf}-1)}
(\qf_{ii}-1)^{-c_{ij}^{\bqf}} \prod_{0\le s<-c_{ij}^{\bqf}} (\qf_{ii}^{s}\qf_{ij}\qf_{ji}-1)^{-2}\in\Acf
\end{align*}
Otherwise $\qf_{ii}$ is a root of unity of order $1-c_{ij}^{\bqf}$, so because $(-c_{ij}^{\bqf})_{\qf_{ii}}^!\in\Cset^\times$, we have
\begin{align*}
(\lambda_{ij}^{\bqf})^{-1}
&= \frac{(\qf_{ii}^{-1}\qf_{ij}\qf_{ji})^{-c_{ij}^{\bqf}}} {(-c_{ij}^{\bqf})_{\qf_{ii}}^!} \prod_{0\le s<-c_{ij}^{\bqf}}
(\qf_{ii}^{s}\qf_{ij}\qf_{ji}-1)^{-1}\in\Acf.    \qed
\end{align*}

\begin{example}\label{exa:brown2}
Let $\bqf$ be of modular type $\Brown(2)$, respectively $\Bgl(4)$, see \S \ref{sec:by-diagram-modular-char3}.
Then  $\Acf=\Cset[\nu^{\pm 1}, (\nu - 1)^{-1}, (\nu -\zeta)^{-1}]$, respectively $\Acf=\Cset[\nu^{\pm 1}, (\nu -1)^{-1}, (\nu  +1)^{-1}]$. 
\end{example}

Recall the Hopf skew-pairing from \S \ref{subsec:large-bqf}.  We now define the \emph{restricted integral forms},
that also play a central role in this paper, as the $\Acf$-submodules
\begin{align}\label{eq:Uqres}
U_{\bqf, \Acf}^{\res -} &:= \{ y\in U_{\bqf}^{-} | \langle y, U_{\bqf,\Acf}^+ \rangle \subset \Acf \},&
U_{\bqf, \Acf}^{\res+} &:= \{ x\in U_{\bqf}^{+} | \langle U_{\bqf,\Acf}^-, x \rangle \subset \Acf \}.
\end{align}
Indeed, these are $\Acf$-subalgebras of $U_{\bqf}^-$ and $U_{\bqf}^+$, respectively. This 
follows from the fact that $U_{\bqf,\Acf}^{\pm}$ are braided Hopf subalgebras of $U_{\bqf}^{\pm}$ over $\Acf$ and the properties of Hopf skew-pairings.

\subsection{PBW-bases of integral forms}
Recall the Lusztig isomorphisms $T_i^{\bqf}$ from \S \ref{subsec:lusztig-isom}.

\begin{lemma}\label{lema:iso-int} 
\begin{enumerate}[leftmargin=*,label=\rm{(\alph*)}]
\item\label{item:Ti-lemma-restriction} $T_i^{\bqf}$  restricts to an $\Acf$-algebra isomorphism $T_{i}^{\bqf}:U_{\rho_i (\bqf), \Acf}\to U_{\bqf, \Acf}$, $i\in\I$.

\item\label{item:Ti-lemma-rootvectors} Let $\beta\in\varDelta_+$. Then $E_{\beta}, F_{\beta}\in U_{\bqf, \Acf}$.
\end{enumerate}
\end{lemma}

\pf \ref{item:Ti-lemma-restriction} follows from \eqref{eq:lusztig-isom-defn-preNichols} 
and Lemma \ref{lemma:lambdaij-in-Acf}, while \ref{item:Ti-lemma-rootvectors} from 
\ref{item:Ti-lemma-restriction} and \eqref{eq:Ebeta-defn}.
\epf

\begin{proposition}\label{prop:PBW-int} 
The sets \eqref{eq:PBW-basis-U+-U-} and \eqref{eq:PBW-basis-U} are $\Acf$-bases of $U_{\bqf, \Acf}^\pm$ and $U_{\bqf, \Acf}$, respectively. 
\end{proposition}

\begin{proof} We consider the case of $U_{\bqf, \Acf}^+$, the other being analogous. Let $Y$ be the set of 
PBW monomials of  $U_{\bqf}^+$ from \eqref{eq:PBW-basis-U+-U-}.
By Lemma \ref{lema:iso-int}, $Y \subset U_{\bqf, \Acf}^+$.
The defining relations of $U_{\bqf}^+$ involve products of $E_i$ with coefficients in $\Acf$, hence we may prove recursively that, for $j>k$, $E_{\beta_j}E_{\beta_k} \in \Acf Y$, the $\Acf$-module generated by $Y$, where each monomial in the expansion has letters $E_{\beta_t}$, $j>t>k$;
see the proof of \cite[Theorem 4.8]{HY-shapov}. Thus $\Acf Y$ is a left ideal containing $1$, so $\Acf Y =  U_{\bqf, \Acf}^+$.
This fact and the direct sum decomposition $U_{\bqf}^+ = \oplus_{y \in Y} \Cset(\nu) y$ imply that
$U_{\bqf, \Acf}^+ = \oplus_{y \in Y} \Acf y$.
\end{proof}

Recall the notation  $\widetilde{N}_{\beta}$ in \S \ref{subsec:cartan-roots}.
Next we consider  the quantum divided powers
\begin{align*}
F_{\beta_j}^{(n)} & =\frac{F_{\beta_j}^n}{(n)_{\bqf_{\beta_j\beta_j}}^!}, & E_{\beta_j}^{(n)}& =\frac{E_{\beta_j}^n}{(n)_{\bqf_{\beta_j\beta_j}}^!}, & &0\le n<\widetilde{N}_{\beta_j}.
\end{align*}

\begin{proposition}\label{prop:dual-A-bases} For $j \in \I_{\ell}$, let $n_j, m_j$ be such that 
$0 \le n_j, m_j< \widetilde{N}_{\beta_j}$. Then
\begin{align*}
\langle F_{\beta_1}^{(n_1)} \dots F_{\beta_\ell}^{(n_\ell)} , E_{\beta_1}^{m_1} \dots E_{\beta_\ell}^{m_\ell} \rangle &= 
\delta_{n_1 m_1} \ldots \delta_{m_\ell n_\ell}.
\end{align*}
\end{proposition}
\pf
Let $\eta_j=\langle F_{\beta_j},E_{\beta_j} \rangle$, $j\in\I_{\ell}$.
The same proof as \cite[Proposition 4.6]{AnY} shows that 
$$\langle F_{\beta_1}^{(n_1)} \dots F_{\beta_\ell}^{(n_\ell)} , E_{\beta_1}^{m_1} \dots E_{\beta_\ell}^{m_\ell} \rangle = 
\delta_{n_1 m_1} \ldots \delta_{m_\ell n_\ell}\eta_1^{n_1} \cdots \eta_\ell^{n_\ell},$$
As in \cite[4.7]{AnY}, we see that $\eta_j=1$:
here $\langle F_i,E_i\rangle=1$, there $\langle F_i,E_i\rangle=-1$ for $i\in\I$.
\epf

Propositions \ref{prop:PBW-int} and  \ref{prop:dual-A-bases}
imply the following:

\begin{corollary}\label{prop:PBW-int-res}
The following sets are $\Acf$-basis of $U_{\bqf,\Acf}^{\res-}$ and $U_{\bqf,\Acf}^{\res+}$, respectively:
\begin{align}\label{eq:PBW-basis-Ures+-Ures-}
&\left\{  F_{\beta_1}^{(n_1)} \dots F_{\beta_\ell}^{(n_\ell)} : 0\le n_j< \widetilde{N}_{\beta_j} \right\}
& &\text{and}
& &\left\{  E_{\beta_1}^{(m_1)} \dots E_{\beta_\ell}^{(m_\ell)} : 0\le m_j< \widetilde{N}_{\beta_j} \right\}.
\end{align}

\end{corollary}
\subsection{The specialization of $U_{\bqf, \Acf}$}
As explained in Property \ref{item:property-c} of the Appendix, there exists $\xi \in \G_{\infty}'$ such that 
$\bqf(\xi) = \bq$; we fix one such $\xi$.

\medbreak
We  consider the  setting in Section \ref{sec:poisson} assuming $R = \Acf$, $h=\nu - \xi$ and the $R$-algebra $A$ being either $U_{\bqf, \Acf}$ or its subalgebras $U_{\bqf, \Acf}^\pm$.
We claim that the map $\Cset[\nu^{\pm 1}] \to \Cset$, $\nu \mapsto \xi$ extends to an isomorphism
$\Acf/(\nu - \xi) \simeq \Cset$. For, if 
\begin{align*}
\qf_{ii}^s\qf_{ij}\qf_{ji}-1 & \mapsto q_{ii}^sq_{ij}q_{ji}-1 =0
&& \text{for some }i\neq j, \; 0\le s<-c_{ij}^{\bqf}, 
\end{align*} 
then $0\le -c_{ij}^{\bq} \le s<-c_{ij}^{\bqf}$, which contradicts Property \ref{item:property-c} of the Appendix \ref{sec:Nichols-specialization}.
Here and below we will use the bar notation $x \mt \overline{x}$ for specializations.

\begin{theorem}\label{thm:Uq-special}
There are Hopf algebra (respectively, braided Hopf algebra) isomorphisms
\begin{align*}
&\Xi_{\bq}:U_{\bq} \to U_{\bqf, \Acf}/(\nu - \xi) 
&& \mbox{and} 
&& \Xi_{\bq}|_{U_{\bq}^\pm}: U_{\bq}^\pm \to U_{\bqf, \Acf}^\pm /(\nu - \xi)
\end{align*}
given by $\Eq_i \mt \ol{E}_i, \Fq_i \mt \ol{F}_i, K_i^{\pm 1} \mt \ol{K}_i^{\, \pm 1}, L_i^{\pm 1} \mt \ol{L}_i^{\, \pm 1}$ 
for all $i \in \I$. For each $i\in\I$, the following diagram is commutative:
\begin{align}\label{eq:eval-commutes-Ti}
\begin{aligned}
\xymatrix@C=50pt{U_{\rho_i(\bq)} \ar@{->}[r]^{\Xi_{\rho_i(\bq)}\qquad} \ar@{->}[d]_{T_{i}^{\bq}} & 
U_{\rho_i(\bq)f, \Acf}/(\nu - \xi) \ar@{->}[d]^{T_{i}^{\bqf}}
\\ U_{\bq} \ar@{->}[r]^{\Xi_{\bq}\qquad } & U_{\bqf, \Acf}/(\nu - \xi). }
\end{aligned}\end{align} 
\end{theorem}

\begin{proof}
The defining relations of $U_{\bq}$ hold in $U_{\bqf, \Acf}/(\nu - \xi)$ by the definition of $U_{\bq}$ in
 \cite{Ang-dpn} and the presentation of $U_{\bqf}$ in \cite{A-presentation}.
 Therefore, the map $\Xi_{\bq}$ as above is well-defined. Moreover $\Xi_{\bq}$ is surjective, since 
$\ol{E}_i$, $\ol{F}_i$, $\ol{K}_i^{\, \pm 1}$, $\ol{L}_i^{\, \pm 1}$ generate $U_{\bqf, \Acf}/(\nu - \xi)$ as $\Cset$-algebra.

Now we check that \eqref{eq:eval-commutes-Ti} is a commutative diagram. Indeed, since Property \ref{item:property-c} in the Appendix \ref{sec:Nichols-specialization} holds and $\bqf \mt \bq$ under the evaluation map, we have that
\begin{align*}
\Xi_{\bq}\circ T_i^{\bq}(e_j)&=(\ad_c \ol{E}_i)^{-c_{ij}^{\bqf}} \ol{E}_j = T_i^{\bqf}\circ \Xi_{\rho_i(\bq)}(e_j), \\
\Xi_{\bq}\circ T_i^{\bq}(f_j)&=(\ad_c \ol{F}_i)^{-c_{ij}^{\bqf}} \ol{F}_j = T_i^{\bqf}\circ \Xi_{\rho_i(\bq)}(f_j)
\end{align*}
for $j\neq i$. Since $\Xi_{\bq}\circ T_i^{\bq}(X) = T_i^{\bqf}\circ \Xi_{\rho_i(\bq)}(X)$ for $X\in\{e_i,f_i,K_j^{\pm1},L_j^{\pm1}\}$, the claim follows.
By \eqref{eq:eval-commutes-Ti}, $\Xi_{\bq}(E_{\beta})=\ol{E}_{\beta}$ and $\Xi_{\bq}(F_{\beta})=\ol{F}_{\beta}$
for all $\beta\in\varDelta_+$. Hence $\Xi_{\bq}$ sends the PBW basis of $U_{\bq}$ to that of $U_{\bqf, \Acf}/(\nu - \xi)$, so $\Xi_{\bq}$,  and its restrictions to $U_{\bq}^{\pm}$, are isomorphisms. 
Clearly $\Xi_{\bq}$ (and its restrictions) are isomorphisms of (braided) Hopf algebras.
\end{proof}

\subsection{The specialization of $U_{\bqf, \Acf}^{\res\pm}$}
\label{subsec:Luszig-pm}
Recall the Lusztig algebra $\luq$ \S \ref{subsubsec:Luszig-algs} and the identification of
$\dpn_{\bq}$ with $U_\bq^+$ as in \S \ref{subsec:Ubq}.
For $\beta \in \fO^{\bq}$, $n \in \N_0$, define $\ee_{\beta}^{(n)}\in\luq$ such that
\begin{align}\label{eq:fdual-definition}
( \ee_{\beta_j}^{(n)}, \Eq_{\beta_1}^{m_1} \dots \Eq_{\beta_\ell}^{m_{\ell}} ) &= \begin{cases}
1, & m_j=n, \, m_k=0 \text{ for }k\neq j, \\ 0, & \text{otherwise}, 
\end{cases}
\end{align}
By \cite[Proposition 4.6]{AAR-MRL}, the  set
\begin{align*}
\{ \ee_{\beta_1}^{(n_1)} \cdots \ee_{\beta_\ell}^{(n_\ell)} :  0\le n_j <\widetilde{N}_{\beta_j} \}
\end{align*}
is a basis of $\luq$ and the algebra $\luq$ is generated by
\begin{align*}
\{ \ee_{\alpha_i}: i\in\I \}  \cup  \{\ee_{\beta}^{(N_\beta)}: \wbeta\in \varPi^{\bq} \}.
\end{align*}
The Lie algebra $\n_\bq^-$ from \eqref{eq:triangular} has a $\Cset$-basis 
$\{ \iota^*(\ee_\beta^{(N_\beta)}) : \beta \in \fO^{\bq} \}$ and set of simple root vectors $\{ \iota^*(\ee_\beta^{(N_\beta)}) : \wbeta \in \varPi^{\bq} \}$.
Similar results hold for the Lusztig algebra $\luqm$ associated to $\dpn_{\bq^{(-)}} \simeq U_\bq^{-}$. The
corresponding elements of $\luqm$, defined as in \eqref{eq:fdual-definition} using $\Fq_{\beta}^{m}$ instead of $\Eq_{\beta}^{m}$, 
will be denoted by $\ff_\beta^{(n)}$, where
$\beta \in \fO^{\bq}$ and $n \in \N_0$.
\begin{remark}\label{rem:iso} The Lie algebras associated to $\luq$ and $\luqm$ as in Remark \ref{rem:Lusztigalg-lie} are isomorphic to each
other, see the list in the Appendix \ref{sec:Nichols-specialization}. Hence we have a Lie algebra isomorphism 
\begin{align}
\label{eq:nq-isom}
\n_{{\bq}^{(-1)}}^- \simeq \n_{\bq}^+
\end{align}
where $\iota^*(\theta_\beta^{(N_\beta)}) \in \n_{{\bq}^{(-1)}}^+$, $\wbeta \in \varPi^{\bq}$ are mapped to the simple root vectors of $\n_{\bq}^-$.
\end{remark}

For a braided Hopf algebra $B$ denote the braided opposite algebra $B^{\underline{\opp}}$ with product 
$\mu^{\opp}:= \mu c^{-1}$ where $\mu : B \times B \to B$ is the product in $B$.

\begin{proposition}\label{prop:evaluation-Ures}
There are $\Cset$-algebra anti-isomorphisms 
\begin{align*}
&\begin{aligned}
\phi^- : U_{\bqf, \Acf}^{\res - }/(\nu - \xi) & \to \luq, 
\\
\phi^+ : \Big( U_{\bqf, \Acf}^{\res +}/(\nu - \xi) \Big)^{\underline{\opp}} & \to \luqm,
\end{aligned}
& &\text{given by} 
&\begin{aligned} \overline{F_{\beta}^{(n)}} & \mapsto \ee_{\beta}^{(n)}, 
\\
\overline{E_{\beta}^{(n)}} & \mapsto \ff_{\beta}^{(n)},
\end{aligned}
& &\beta &\in\fO^{\bq}, \ n\in\N_0.
\end{align*}

\end{proposition}

\pf
We prove the statement in the minus case, the plus case is analogous. 
By Proposition \ref{prop:dual-A-bases} the Hopf skew-pairing 
$\langle \, , \, \rangle:U_{\bqf}^{-}\times U_{\bqf}^{+}\to \Cset(\nu)$
restricts to a perfect 
pairing 
\begin{align*}
\langle \, , \, \rangle:U_{\bqf,\Acf}^{\res-}\times U_{\bqf,\Acf}^{+}\to \Acf.
\end{align*}
Since $U_{\bqf,\Acf}^{+}/(\nu - \xi) \simeq U_{\bq}^{+}$ as braided Hopf algebras, the latter pairing induces a non-degenerate pairing 
$\langle \, , \, \rangle: \big(U_{\bqf,\Acf}^{\res-}/(\nu - \xi)\big) \times U_{\bq}^{+}\to \Cset$ such that
\begin{align}
\label{eq:pair1}
\langle y y', x \rangle &= \langle y \otimes y',\Delta(x) \rangle, 
&& y, y' \in U_{\bqf,\Acf}^{\res-}/(\nu - \xi), x \in U_{\bq}^{+}
\end{align}
and we have the commutative diagram
\begin{align*}
\xymatrix@C=50pt{U_{\bqf,\Acf}^{\res-}\times U_{\bqf,\Acf}^{+} \ar@{->}[r] \ar@{->}[d] & \Acf \ar@{->}[d]
\\ \big(U_{\bqf,\Acf}^{\res-}/(\nu - \xi)\big) \times U_{\bq}^{+} \ar@{->}[r] & \Cset }
\end{align*}
By the definition of $\luq$, we have a canonical vector space isomorphism 
\begin{align*}
\phi^- : U_{\bqf, \Acf}^{\res - }/(\nu - \xi) \to \luq && \mbox{such that} &&
\langle Y, x \rangle &= ( \phi^-(Y), x ) 
\end{align*}
for all 
$Y \in U_{\bqf, \Acf}^{\res - }/(\nu - \xi), x \in U_{\bq}^{+}$.
Comparing \eqref{eq:pair-LqBq} and \eqref{eq:pair1}, we see that $\phi^-$ is an algebra anti-isomorphism.
Using again Proposition \ref{prop:dual-A-bases} and the definition \eqref{eq:fdual-definition} of $\ee_{\beta}^{(n)}$, we get that $\phi^-$ is given by 
$\overline{F_{\beta}^{(n)}} \mapsto \ee_{\beta}^{(n)}$ for  $\beta\in\fO^{\bq}$, $n\in\N_0$.
\epf

\section{Poisson orders on large quantum groups}\label{sec:Pord}

By Theorem  \ref{thm:Uq-special}, the large quantum group $U_{\bq}$ fits in the context of 
Section \ref{sec:poisson} and consequently the pair $(U_{\bq}, \ZZ(U_{\bq}))$ inherits a structure of Poisson order from deformation theory.
However the  Poisson algebra $\ZZ(U_{\bq})$ is often singular. 
We prove that the central Hopf subalgebra $\Zc_{\bq}$ introduced in \S \ref{subsec:Zq} (which is of course regular)
is a Poisson subalgebra of $\ZZ(U_{\bq})$ of the same dimension.
Thus $(U_{\bq}, \Zc_{\bq})$ has a structure of Poisson order that 
restricts to the corresponding large quantum Borel  and  unipotent algebras.

\subsection{Poisson structure on $\Zc_{\bq}$}
\label{sec:6.1}

We show that $\Zc_{\bq}$,  $\Zc_{\bq}^{\geqslant}$, $\Zc_{\bq}^{\leqslant}$, $\Zc_{\bq}^+$ and $\Zc_{\bq}^{-}$
are Poisson subalgebras of $\ZZ(U_{\bq})$, respectively
$\ZZ(U_{\bq}^{\geqslant})$, $\ZZ(U_{\bq}^{\leqslant})$, $\ZZ(U_{\bq}^{+})$ and $\ZZ(U_{\bq}^{-})$. 

\smallbreak
The coefficients of Poisson brackets that we use will be expressed in terms of a square matrix $\WP^{\bqf} \in \Cset^{\varPi^{\bq} \times \varPi^{\bq}}$. 
Furthermore, in the next section we will show that the Cartan matrix of the semisimple Lie algebra $\g_{\bq}$ is also expressed in terms of the 
entries of this matrix. The matrix $\WP^{\bqf}$ is defined as follows.
Let $\beta, \gamma \in\fO^{\bq}_{+}$. As $q_{\beta\gamma} = \bqf_{\beta\gamma}(\xi)$, \eqref{eq:condition-cartan-roots} implies 
that there exists $\wp^{\bqf}_{\beta\gamma}(\nu)\in\Acf$ such that
\begin{align}\label{eq:qalphabeta-factorization}
1-\bqf_{\beta\gamma}^{N_{\beta}N_{\gamma}} &= (\nu - \xi) \wp^{\bqf}_{\beta\gamma}(\nu).
\end{align}
Recall the notation $\wbeta$ from \eqref{eq:root-system-distinguished} and
the set $\wfO^{\bq}$ from Theorem \ref{thm:Cartan-roots}. 
Define
\begin{align}\label{eq:matrix-P}
\WP^{\bqf}:=(\wp^{\bqf}_{\beta\gamma}(\xi))_{\wbeta,\wgamma\in\varPi^{\bq}}.
\end{align}

We  distinguish two cases, namely whether $\bq$ is of type $\supera{k-1}{\theta-k}$ or not. In the first case
we need one more generator to have finite type of $U_{\bq}$ as $Z_{\bq}$-module, and \emph{a fortiori}
as  $Z(U_{\bq})$-module.

\begin{case}
 $\bq$ is of type $\supera{k-1}{\theta-k}$ with Dynkin diagram as in Table \ref{tab:Appendix-super}. Set
\begin{align}\label{eq:eta}
\eta & := \sum_{1 \le i \le k} \alpha_i+ \sum_{k<i\le \theta} (i-k+1) \alpha_i, & \weta &:= N_\eta \eta,
\end{align}
recall \eqref{eq:betak}. It is easy to see that $N_\eta = \Nt$, recall \eqref{eq:defn-lcm-Nbeta}. Denote
\begin{align}\label{eq:def-wtPi-1}
\widetilde{\varPi}^{\bq}:= \varPi^{\bq} \sqcup\{\weta\} =
\{\walpha_i|i\ne k\}\sqcup\{\weta\}.
\end{align}
\end{case}

We have the following:
\begin{enumerate}[leftmargin=*,label=\rm{(\Roman*)}]
\item By Proposition \ref{prop:Zq-families}, $\Zc_{\bq}$ is generated by \eqref{eq:generators-Zc-Cartan-roots} and $K_{\eta}^{\pm N_\eta}$, $L_{\eta}^{\pm N_\eta}$.
\item By direct computation, $q_{\walpha_i\weta}=q_{\weta\walpha_i}=1$. Arguing as in \eqref{eq:qalphabeta-factorization}, there exist $\wp^{\bqf}_{\alpha_i\eta}(\nu)$, $\wp^{\bqf}_{\eta\alpha_i}(\nu)\in\Acf$ such that
\begin{align*}
1-\bqf_{\walpha_i\weta} &= (\nu - \xi) \wp^{\bqf}_{\alpha_i\eta}(\nu), &
1-\bqf_{\weta\walpha_i} &= (\nu - \xi) \wp^{\bqf}_{\eta\alpha_i}(\nu).
\end{align*}
We check that the following equality holds:
\begin{align}
\label{peta0}
&\wp^{\bqf}_{\alpha_i\eta}(\xi)+\wp^{\bqf}_{\eta\alpha_i}(\xi)=0, & &i\ne k.
\end{align}

\item Similarly, $q_{\weta\weta}=1$, so there exist $\wp^{\bqf}_{\eta\eta}(\nu)\in\Acf$ such that
$1-\bqf_{\weta\weta}= (\nu - \xi) \wp^{\bqf}_{\eta\eta}(\nu)$.
By direct computation,
\begin{align}
\label{petanot0}
\wp^{\bqf}_{\eta\eta}(\xi)&=-N_\eta^2\xi \left( \theta-k+(\theta-k+1)^2 \right) \ne 0.
\end{align}
\end{enumerate}

\begin{case}
 $\bq$ is not of type $\supera{k-1}{\theta-k}$. Set 
\begin{align}\label{eq:def-wtPi-2}
\widetilde{\varPi}^{\bq}:= \varPi^{\bq}.
\end{align}
\end{case}

It follows from Proposition \ref{prop:Zq-families} that
\begin{align}
\label{zqu-Laurent}
&Z_\bq^{+0} = \Cset[ K_\mu^{N_\mu} : \wmu \in \wt{\varPi}_\bq],
&Z_\bq^{-0} = \Cset[ L_\mu^{N_\mu} : \wmu \in \wt{\varPi}_\bq].
\end{align}

\begin{lemma}\label{lem:matrix-P-Weyl-gpd}
Let $i\in \I$. Then $\WP^{\rho_i (\bqf)}=\WP^{\bqf}$.
\end{lemma}
\pf
First, $\varPi^{\rho_i(\bq)}=s_i^{\bq}(\varPi^{\bq})$. Thus $\wp^{\rho_i(\bqf)}_{s_i^{\bq}(\beta)s_i^{\bq}(\gamma)}(\nu) =\wp^{\bqf}_{\beta\gamma}(\nu)$ for $\wbeta,\wgamma\in\varPi^{\bq}$ by  \eqref{eq:rhoiq-definition}.
\epf

The next theorem is the main result of this section.

\begin{theorem}\label{thm:Poisson-orders} There are structures of Poisson order on the pairs
\begin{align}\label{eq:Poisson-orders}
&(U_{\bq}, Z_{\bq}),& &(U_{\bq}^{\geqslant}, Z_{\bq}^{\geqslant}), & &(U_{\bq}^{\leqslant}, Z_{\bq}^{\leqslant}), & &
(U_{\bq}^{+}, Z_{\bq}^{+})& & \text{ and } & &(U_{\bq}^{-}, Z_{\bq}^{-})
\end{align} 
arising by restriction from the Poisson order on the corresponding algebra and its center
with Poisson bracket \eqref{eq:Poisson-str-center}. 
The central algebras $Z_{\bq}$, $Z_{\bq}^{\geqslant}$ and $Z_{\bq}^{\leqslant}$ are Poisson-Hopf while 
$Z_{\bq}^{\pm}$ are coideal Poisson subalgebras over the former.
\end{theorem}

In each of the pairs in \eqref{eq:Poisson-orders} the second algebra is central  to the first. 
Using PBW bases one easily shows that in each case the first algebra is a finitely generated module over the second one 
in the pair; here the introduction of the  generators \eqref{eq:generators-Zc-new}  for  super type A is essential. 
Because of Theorem  \ref{thm:Uq-special} and Proposition \ref{prop:der_a-gral-centralHopf} we are reduced to prove:

\begin{proposition}\label{prop:Poisson-Zq}
The subalgebras $\Zc_{\bq}^{\pm}$, $\Zc_{\bq}^{\geqslant}$, $\Zc_{\bq}^{\leqslant}$ and $\Zc_{\bq}$
are Poisson subalgebras of $\ZZ(U_{\bq}^{\pm})$, $\ZZ(U_{\bq}^{\geqslant})$,  $\ZZ(U_{\bq}^{\leqslant})$ and
$\ZZ(U_{\bq})$, respectively, under the Poisson bracket \eqref{eq:Poisson-str-center}.
\end{proposition}

Observe that $\Zc_{\bq}^{\pm}$, $\Zc_{\bq}^{\geqslant}$ and $\Zc_{\bq}^{\leqslant}$
are Poisson subalgebras of $\Zc_{\bq}$.

\medbreak
\noindent \emph{Proof.} 
We apply Theorem \ref{thm:poisson-structure-ker} to the algebra $U_{\bqf, \Acf}^+$, the automorphisms $\varsigma_i^{\bqf}$ and the $(\id,\varsigma_i^{\bqf})$-derivations 
$\partial_i^{\bqf}$, $i\in\I$ to conclude that $\ZZ'$ defined as in \eqref{eq:poisson-structure-subalgebra}
is a Poisson subalgebra of $\ZZ(U_{\bq}^+)$. Now we have that
\begin{align*}
\overline{\varsigma}_i^{\;  \bqf} & \overset{\star}{=} \varsigma_i^{\bq}, & 
\overline{\partial}_i^{\; \bqf} & \overset{\ast}{=} \partial_i^{\bq}, & 
\Zc_{\bq}^+ &\subset \cap_{i\in\I} \ker (\varsigma_i^{\bq}-\id).
\end{align*}
The equality $\star$ holds since $\bq = \bqf(\xi)$, while $\ast$ holds because both skew-derivations act in the same way on the generators of $U_{\bq}^+$. 
The inclusion holds since $\Zc_{\bq}^+ \subset \ZZ(U_{\bq})$: indeed  $\varsigma_i^{\bq}(x)=K_ix K_i^{-1} =x$ for all $x \in Z_{\bq}^+$. 
From this inclusion and \eqref{eq:Z+-intersection-ker-derivations}, it follows that $\ZZ'=\Zc_{\bq}^+$. The proof for $\Zc_{\bq}^-$ is analogous.
The restriction of the Poisson structure to $\Zc_{\bq}^{0\pm}$ vanishes by the definition \eqref{eq:Poisson-str-center}.

\medbreak

Next we prove the statement for $\Zc_{\bq}^{\geqslant}$. 
Let $\beta,\gamma\in \fO^{\bq}_{+}$. We have
\begin{align*}
\{ \Eq_{\beta}^{N_{\beta}}, K_{\gamma}^{N_{\gamma}} \} &= 
\overline{\frac{[E_{\beta}^{N_{\beta}}, K_{\gamma}^{N_{\gamma}}]}{\nu - \xi}} = \overline{\frac{1-\bqf_{\beta\gamma}^{N_{\beta}N_{\gamma}}}{\nu - \xi} E_{\beta}^{N_{\beta}}K_{\gamma}^{N_{\gamma}}} 
\overset{\eqref{eq:qalphabeta-factorization}}{=} \wp^{\bqf}_{\beta \gamma} (\xi) \Eq_{\beta}^{N_{\beta}} K_{\gamma}^{N_{\gamma}}\in \Zc_{\bq}^\geqslant.
\end{align*}
If $\bq$ is of type $\supera{k-1}{\theta - k}$ and $\eta$ is as in \eqref{eq:eta}, 
then $K_{\eta}^{N_\eta}\in \ZZ(U_{\bq})$ and $\{ \Eq_{\beta}^{N_{\beta}}, K_{\eta}^{N_\eta} \} \in 
\Cset \Eq_{\beta}^{N_{\beta}} K_{\eta}^{N_\eta} \subset
\Zc_{\bq}^\geqslant$.
This proves the claim in light of Proposition \ref{prop:Zq-families} and Property (I) in \S \ref{sec:6.1}. 
Similarly, one shows that
\begin{align*}
\{ \Eq_{\beta}^{N_{\beta}}, L_{\mu}^{N_{\mu}}\} &\in \Cset \Eq_{\beta}^{N_{\beta}} L_{\mu}^{N_{\mu}},&
\{ \Fq_{\beta}^{N_{\beta}}, K_{\mu}^{N_{\mu}} \} &\in \Cset \Eq_{\beta}^{N_{\beta}} K_{\mu}^{N_{\mu}}, &
\{ \Fq_{\beta}^{N_{\beta}}, L_{\mu}^{N_{\mu}} \} &\in \Cset \Fq_{\beta}^{N_{\beta}} L_{\mu}^{N_{\mu}},
\end{align*}
for all $\beta \in  \fO^{\bq}_{+}, \wmu \in \widetilde{\varPi}^{\bq}$. 
This finishes the proof for $\Zc_{\bq}^{\leqslant}$ and reduces that of $\Zc_{\bq}$
to proving that $\{ \Eq_{\beta}^{N_{\beta}}, \Fq_{\gamma}^{N_{\gamma}} \}  \in \Zc_{\bq}$ for all  $\beta, \gamma \in  \fO^{\bq}_{+}$. 
For this we use the enumeration of the positive roots using the longest element of the Weyl groupoid. 
First we assume that $\beta=\beta_j$, $\gamma=\beta_k$ for $1\le j<k\le \ell$. Let $\bp=\rho_{i_j} \dots \rho_{i_1}(\bq)$, $\gamma'= s_{i_j}^{\bp} \dots s_{i_1}(\gamma)$, 
so $N_{\gamma'}=N_{\gamma}$. We have that
\begin{align*}
\{ \Eq_{\beta}^{N_{\beta}}, \Fq_{\gamma}^{N_{\gamma}} \} &= 
\overline{\frac{[E_{\beta}^{N_{\beta}}, F_{\gamma}^{N_{\gamma}}]}{\nu - \xi}} =
\overline{\frac{T_{i_1}^{\bqf} \dots T_{i_j}([K_{i_j}^{-N_{i_j}} F_{i_{j}}^{N_{i_j}}, F_{\gamma'}^{N_{\gamma'}}])}{\nu - \xi}}
\\ & \overset{\eqref{eq:eval-commutes-Ti}}{=}
T_{i_1}^{\bq} \dots T_{i_j}
\left(\overline{\frac{[K_{i_j}^{-N_{i_j}} F_{i_{j}}^{N_{i_j}}, F_{\gamma'}^{N_{\gamma'}}]}{\nu - \xi}}\right) = T_{i_1}^{\bq} \dots T_{i_j} \left( \left\{ K_{i_j}^{-N_{i_j}} \Fq_{i_{j}}^{N_{i_j}}, \Fq_{\gamma'}^{N_{\gamma'}} \right\} \right).
\end{align*}
By the statements already proved, $\left\{ K_{i_j}^{-N_{i_j}} \Fq_{i_{j}}^{N_{i_j}}, \Fq_{\gamma'}^{N_{\gamma'}} \right\}\in\Zc_{\bp}$. Hence
\begin{align*}
\{ \Eq_{\beta}^{N_{\beta}}, \Fq_{\gamma}^{N_{\gamma}} \} & \in
T_{i_1}^{\bq} \dots T_{i_j} (\Zc_{\bp}) \overset{\text{Theorem \ref{th:Zq-invariant-under-Ti}}}{=} \Zc_{\bq}.
\end{align*}
The case $j>k$ is proved analogously. Now assume that $\beta=\gamma$. We start with the case $\beta=\alpha_i$ for some $i\in\I$ (a simple Cartan root). Using \eqref{eq:Ubqf-linking-EF} we prove recursively that
\begin{align}\label{eq:conmut-EN-FN}
[E_i^N,F_i^N] &= \sum_{t=1}^{N} (t)_{\qf_{ii}}^! \binom{N}{t}^2_{\qf_{ii}} F_i^{N-t} \prod_{s=0}^{t-1} \left(  K_i \qf_{ii}^{2t-2N-s} -L_i^{-1} \right) E_i^{N-t}, & &N\in\N.
\end{align}

Let $t\in\I_{N_i-1}$. As $q_{ii}$ is a primitive $N_i$-th root of unity and $\overline{\qf_{ii}} = q_{ii}$, 
\begin{align*}
\wp^{\bqf}_{\alpha_i\alpha_i}(\xi) &= \overline{\frac{1 - \qf_{ii}^{N_i^2}}{\nu - \xi}}  
= \overline{\frac{(1 - \qf_{ii}^{N_i})( 1 + \qf_{ii}^{N_i} + \cdots + \qf_{ii}^{N_i(N_i-1)}) }{\nu - \xi}}  = N_i \overline{\frac{1 - \qf_{ii}^{N_i}}{\nu - \xi}}.
\end{align*}
Hence,
\begin{align}
\label{eq:s-beta-spec}
\overline{\frac{(N_i)_{\qf_{ii}}^!}{\nu - \xi}} &= \overline{\frac{1 - \qf_{ii}^{N_i}}{\nu - \xi}} \cdot \frac{(1 - q_{ii}) \ldots (1- q_{ii}^{N_i -1})}{(1- q_{ii})^{N_i}} = 
\frac{\wp^{\bqf}_{\alpha_i\alpha_i}(\xi)}{(1-q_{ii})^{N_i}} \cdot
\end{align}
From this we obtain,
\begin{align*}
\{ \Eq_{i}^{N_i}, \Fq_{i}^{N_i} \} &= \overline{\frac{[E_{i}^{N_{i}}, F_{i}^{N_{i}}]}{\nu - \xi}} =
\overline{\frac{(N_i)_{\qf_{ii}}^!}{\nu - \xi}}
\prod_{s=0}^{N_i-1} (K_i q_{ii}^{-s}-L_i)
=
\frac{-\wp^{\bqf}_{\alpha_i\alpha_i}(\xi)}{(q_{ii}-1)^{N_i}}
(K_i^{N_i}-L_i^{-N_i})\in\Zc_{\bq}.
\end{align*}
Next, if $\beta$ is not simple, say $\beta=\beta_j$ for some $j \in \I_{\ell}$, then using 
Theorem \ref{th:Zq-invariant-under-Ti} and \eqref{eq:eval-commutes-Ti} again
\begin{align}
\label{eq:bracket-e-f}
\{ \Eq_{\beta}^{N_{\beta}}, \Fq_{\beta}^{N_{\beta}} \} &=
T_{i_1}^{\bq} \dots T_{i_{j-1}}
\Big(\overline{\frac{[E_{i_{j}}^{N_{i_j}}, F_{i_j}^{N_{i_j}}]}{\nu - \xi}}\Big)
=\frac{-\wp^{\bqf}_{\beta\beta}(\xi)}{(q_{\beta\beta}-1)^{N_{\beta}}}
(K_{\beta}^{N_{\beta}}-L_{\beta}^{N_{\beta}}) \in \Zc_{\bq}. \qed
\end{align}

\section{The associated Poisson algebraic groups}\label{sec:PLgrps}
In this section we describe the Poisson algebraic groups that correspond to the Poisson-Hopf algebras 
$Z_{\bq}$, $Z_{\bq}^{\geqslant}$ and $Z_{\bq}^{\leqslant}$.
We prove that, as algebraic groups, they are isomorphic to Borel subgroups of connected semisimple algebraic groups but of adjoint type
(and not of simply connected type as in previous works) and direct products of such Borel subgroups.
The dual Lie bialgebras of the three tangent Lie algebras are proved to constitute a Manin triple, 
the ample Lie algebra in which is reductive. It is shown that the resulting Lie bialgebra structures are 
the ones from the Belavin--Drinfeld classification \cite{BD} for the empty BD-triple
 and arbitrary choice of the continuous parameters. 
The results completely determine the 
Poisson structures on the three kinds of algebraic groups in question.
  
\subsection{The positive and negative parts of the dual tangent Lie bialgebra of $M_{\bq}$} 
Let $M_{\bq}$, $M_{\bq}^\pm$, $M_{\bq}^{\pm 0}$, $M_{\bq}^\geqslant$ and $M_{\bq}^\leqslant$ be the complex 
algebraic groups which are 
equal to the maximal spectra of the commutative Hopf algebras
$\Zc_{\bq}$, $\Zc_{\bq}^\pm$, $\Zc_{\bq}^{\pm 0}$, $\Zc_{\bq}^\geqslant$ and $ \Zc_{\bq}^\leqslant$, respectively. 
Here the Hopf algebra structures on $\Zc_{\bq}^\pm$ are the restrictions of the braided Hopf algebra structures
on $U^\pm_{\bq}$ to $\Zc_{\bq}^\pm$ \cite{Ang-dpn}.

Since $\Zc_{\bq}$ is a finitely generated Poisson-Hopf algebra which is an integral domain, 
$M_{\bq}$ is a connected Poisson algebraic group (see \S \ref{subsec:B.2} for background). 
Analogously, $M_{\bq}^\geqslant$, $M_{\bq}^\leqslant$ and  $M_{\bq}^{\pm 0}$ are connected Poisson algebraic groups, 
and $M_{\bq}^\pm$ are connected unipotent algebraic groups. The latter are not Poisson algebraic groups;
they are isomorphic to certain Poisson homogeneous spaces for $M_{\bq}^\geqslant$ and $M_{\bq}^\leqslant$ (see \S \ref{Poisson-geom-unip}). 
The tensor product decompositions 
$Z_{\bq} \simeq Z_{\bq}^\geqslant \otimes Z_{\bq}^\leqslant$ 
from \S \ref{sec:Z} give rise to the decomposition of algebraic groups 
\begin{equation}
\label{M-decomp}
M_{\bq} \simeq M_{\bq}^\geqslant \times M_{\bq}^\leqslant.
\end{equation}
This is not a direct product decomposition of Poisson algebraic groups (because $Z_{\bq} \simeq Z_{\bq}^\geqslant \otimes Z_{\bq}^\leqslant$
is a tensor product decomposition of commutative but not Poisson algebras). However,
the canonical projections $M_{\bq} \twoheadrightarrow M_{\bq}^\geqslant$ and $M_{\bq} \twoheadrightarrow M_{\bq}^\leqslant$
are homomorphisms of Poisson algebraic groups because $Z_{\bq}^\geqslant$ and $Z_{\bq}^\leqslant$ are 
Poisson-Hopf subalgebras of $Z_{\bq}$. 

Denote by $\mm_{\bq}$, $\mm_{\bq}^{\geqslant}$ and $\mm_{\bq}^{\leqslant}$ the tangent Lie bialgebras of  $M_{\bq}$, $M_{\bq}^\geqslant$ and $M_{\bq}^\leqslant$ 
(see the Appendix \ref{subsec:B.1} for background and notations). Eq. \eqref{M-decomp} gives rise to the direct sum decomposition of Lie algebras
\[
\mm_{\bq} \simeq \mm_{\bq}^\leqslant \oplus \mm_{\bq}^\geqslant.
\]
The Lie coalgebra structure on $\mm_{\bq}$, fully described below, has cross terms.
The dual of the tangent Lie bialgebra $\mm_{\bq}^* = T^*_1 M_{\bq}$ is computed as the linearization at the identity element $1$ of $M_{\bq}$
of its Poisson structure by using \eqref{eq:linear}. 
The maximal ideal $\Mg_1$ of $\Cset[M_{\bq}] \simeq \Zc_{\bq}$ of functions vanishing at 1 coincides with the 
augmentation ideal of $\Zc_{\bq}$. In the proofs below we will use the identification $T^*_1 M \simeq \Mg_1 / \Mg_1^2$
where the differential $d_1(g)$ of a function $g \in \Cset[M_{\bq}]$ at $1 \in M_{\bq}$
is sent to the class of $g - g(1)$ in $\Mg_1 / \Mg_1^2$ for $g \in \Cset[M_{\bq}]$.
The Lie algebra $\mm_{\bq}^*$ has the $\Cset$-basis:
\begin{align}
\label{eq:T*-basis}
&\big{\{} d_1 (\Eq_\be^{N_\be}),  d_1 (\Fq_\be^{N_\be}), d_1 (K_\mu^{N_\mu}), d_1 (L_\mu^{N_\mu}) 
: \be \in \fO^{\bq}_{+}, \wmu \in \wt{\varPi}^{\bq}
\big{\}}.
\end{align}
By Proposition \ref{prop:Poisson-Zq}, the subspaces 
\[
(\mm_{\bq}^+)^{*} := \oplus_{\be \in \fO^{\bq}_{+}} \Cset \, d_1(\Eq_\be^{N_\be}) \quad \mbox{and} \quad 
(\mm_{\bq}^-)^{*} := \oplus_{\be \in \fO^{\bq}_{+}} \Cset \,  d_1(\Fq_\be^{N_\be})
\]
are Lie subalgebras of $\mm_{\bq}^*$. The dual Lie bialgebras $(\mm_{\bq}^{\geqslant})^{*}$ and $(\mm_{\bq}^{\leqslant})^{*}$ 
are canonically identified with the Lie sub-bialgebras of $\mm_{\bq}^*$
\begin{equation}
\label{m><ident}
(\mm_{\bq}^+)^{*} \oplus  \big( \oplus_{\wmu \in \wt{\varPi}^{\bq}} d_1(K_\mu^{N_\mu}) \big) \quad \mbox{and} \quad
(\mm_{\bq}^-)^{*} \oplus  \big( \oplus_{\wmu \in \wt{\varPi}^{\bq}} d_1(L_\mu^{N_\mu}) \big). 
\end{equation}

Recall the notation from \S \ref{subsec:Luszig-pm}.
It follows from the triangular decomposition \eqref{eq:triangular} of the semisimple Lie algebra 
$\g_{\bq}$ associated to  $\bq$ that 
the set of simple roots of  $\g_{\bq}$ can be identified with $\varPi^{\bq}$. 
Denote the entries of the Cartan matrix of $\g_{\bq}$  by 
\begin{align*}
&c_{\wbeta \, \wgamma}, &&\wbeta, \wgamma \in \varPi^{\bq}.
\end{align*}
Throughout the section we will assume the identification $\n_{{\bq}^{(-1)}}^- \simeq \n_{\bq}^+$ from \eqref{eq:nq-isom}, so 
$\g_{\bq} = \n_{\bq}^+ \oplus \h_{\bq} \oplus \n_{\bq}^-$ will be identified with $\n_{\bq^{(-1)}}^- \oplus \h_{\bq} \oplus \n_{\bq}^-$.
By the definitions of $\n_{{\bq}^{(-1)}}^-$ and $\n_{\bq}^-$, $\g_{\bq}$ has a set of Chevalley generators 
\begin{align*}
&\{ x_\beta, y_\beta, h_\beta : \wbeta \in \varPi^{\bq} \}
\end{align*}
such that $x_\beta \in \Cset^{\times} \iota^*(\ff_\beta)$ and $y_\beta \in \Cset^{\times} \iota^*(\ee_\beta)$, respectively
(here $\ee_{\beta}$ and $\ff_{\beta}$ are defined in \eqref{eq:fdual-definition} and the subsequent paragraph).
 In this way the 
root lattice of $\g_{\bq}$ is identified with $\QQ_{\bq}$ by setting $\deg x_\beta = - \deg y_\beta = N_\beta \beta$, 
$\deg h_\beta =0$ for $\wbeta \in \varPi^{\bq}$.

\medbreak
We will need the following reductive Lie algebra
\[
\wt{\g}_{\bq} := 
\begin{cases}
\g_{\bq} \oplus \Cset, & \mbox{if $\bq$ is of type $\supera{k-1}{\theta-k}$}
\\
\g_{\bq}, & \mbox{otherwise}.
\end{cases}
\]
See the comments after Theorem \ref{thm:Poisson-orders}.
By \eqref{eq:triangular}, it has the triangular decomposition 
\[
\wt{\g}_{\bq} =  \n_\bq^+ \oplus \wt{\h}_{\bq} \oplus \n_\bq^-
\]
where the Cartan subalgebra is given by 
\[
\wt{\h}_{\bq} := 
\begin{cases}
\h_{\bq} \oplus \Cset, & \mbox{if $\bq$ is of type $\supera{k-1}{\theta-k}$}
\\
\h_{\bq}, & \mbox{otherwise}.
\end{cases}
\]
In the $\supera{k-1}{\theta-k}$ case denote by $h_\eta$ a non-zero central element of $\wt{\g}_{\bq}$. 
In that case we have $\wt{\g}_{\bq} = \g_{\bq} \oplus \Cset h_\eta$ and 
$\wt{\h}_{\bq} = \h_{\bq} \oplus \Cset h_\eta$.

\begin{proposition}\label{prop:lem2} We have a $\QQ_{\bq}$-graded Lie algebra isomorphism 
$(\mm_{\bq}^{\pm})^{*} \simeq \n_{\bq}^\pm$ given by 
\begin{align*}
&d_1 ( \Eq_\be^{N_\be} ) \mt s_{\beta} \iota^*( \ff_\be^{(N_\be)}), && \mbox{respectively}
&& d_1 ( \Fq_\be^{N_\be} ) \mt - s_{\beta}\iota^*( \ee_\be^{(N_\be)})
\end{align*}
for all $\be \in \fO^{\bq}_{+}$, where $s_{\beta}:= \dfrac{\wp^{\bqf}_{\beta\beta}(\xi)}{(1-q_{\beta\beta})^{N_{\beta}}}$.
In the plus case we use the identification \eqref{eq:nq-isom}.
\end{proposition}

\pf First we prove the minus case. Let $\beta,\gamma\in\fO^{\bq}_{+}$. Since $\overline{F_{\beta}^{N_{\beta}}}=\Fq_{\beta}^{N_\beta}, \overline{F_{\gamma}^{N_{\gamma}}}=\Fq_{\gamma}^{N_\gamma} \in\Zc_{\bq}$ and the subalgebra $\Zc_{\bq}$ is closed under the Poisson bracket $\{\cdot , \cdot\}$ by Proposition \ref{prop:Poisson-Zq}, using Proposition \ref{prop:PBW-int} we obtain
\begin{align}\label{eq:commutator-U-}
[F_{\beta}^{N_{\beta}}, F_{\gamma}^{N_{\gamma}}] \equiv \sum_{\delta\in\fO^{\bq}_{+}}(\nu - \xi)a_{\beta\gamma}^{\delta}(\nu) F_{\delta}^{N_{\delta}}
+(\nu - \xi) g_{\beta\gamma} \mod (\nu - \xi)^2 U_{\bqf,\Acf}^{-}, 
\end{align}
where $a_{\beta\gamma}^{\delta}(\nu)\in\Acf$ and $g_{\beta\gamma}$ is a non-commutative polynomial in $\{F_{\delta}^{N_{\delta}} : \delta \in \fO^{\bq}_{+}\}$ involving monomials of degree $\ge 2$. Since $U_{\bqf}$ is $\Z^{\I}$-graded, the sum in the right-hand side has at most one non-zero term, when $N_{\beta}\beta + N_{\gamma}\gamma = N_{\delta}\delta$ for some $\delta\in\fO^{\bq}_{+}$. Therefore
\begin{align}
\label{eq:commutator-U-eval}
[d_1(\Fq_{\beta}^{N_{\beta}}), d_1(\Fq_{\gamma}^{N_{\gamma}})] &=
d_1 \big( \{ \Fq_{\beta}^{N_{\beta}}, \Fq_{\gamma}^{N_{\gamma}}\} \big)
=d_1  \Bigg( \, \overline{\frac{[F_{\beta}^{N_{\beta}}, F_{\gamma}^{N_{\gamma}}]}{\nu - \xi}} \, \Bigg)
\\
&=\sum_{\delta\in\fO^{\bq}_{+}} a_{\beta\gamma}^{\delta}(\xi) d_1(\Fq_{\delta}^{N_{\delta}})
+ d_1 (g_{\beta\gamma}) =\sum_{\delta\in\fO^{\bq}_{+}} a_{\beta\gamma}^{\delta}(\xi) d_1(\Fq_{\delta}^{N_{\delta}}),
\notag
\end{align}
because $g_{\beta \gamma}\in{\Mg_1}^2$. From \eqref{eq:commutator-U-} and since $U_{\bqf,\Acf}^{-}$ is $\N_0$-graded connected, we see that
\begin{align*}
[F_{\beta}^{(N_{\beta})}, F_{\gamma}^{(N_{\gamma})}] \equiv \sum_{\delta\in\fO^{\bq}_{+}} a_{\beta\gamma}^{\delta}(\nu)   
\frac{(\nu - \xi)(N_{\delta})_{q_{\delta\delta}}^!}{(N_{\beta})_{q_{\beta\beta}}^!(N_{\gamma})_{q_{\gamma\gamma}}^!}F_{\delta}^{(N_{\delta})}
\mod (\nu - \xi) U_{\bqf,\Acf}^{\res-}.
\end{align*}
It follows from \eqref{eq:s-beta-spec} that
\begin{align*}
\overline{\frac{(N_{\beta})_{q_{\beta\beta}}^!}{\nu - \xi}} &=
\frac{\wp^{\bqf}_{\beta\beta}(\xi)}{(1-q_{\beta\beta})^{N_{\beta}}} =s_{\beta}.
\end{align*}
Hence in $U_{\bqf,\Acf}^{\res-}/(\nu - \xi)$ we have
\begin{align}
\label{eq:commutator-divided-powers}
\begin{aligned}
[s_{\beta} \overline{F_{\beta}^{(N_{\beta})}}, s_{\gamma} \overline{F_{\gamma}^{(N_{\gamma})}}] &= \sum_{\delta\in\fO^{\bq}_{+}} a_{\beta\gamma}^{\delta}(\xi) s_{\delta} \overline{F_{\delta}^{(N_{\delta})}}
\\
& =\begin{cases}
a_{\beta\gamma}^{\delta}(\xi) s_{\delta} \overline{F_{\delta}^{(N_{\delta})}} & \text{ if }\exists \delta\in\fO^{\bq}_{+}: \wdelta=\wbeta+\wgamma,
\\
0, & \text{otherwise}.
\end{cases}
\end{aligned}
\end{align}
The statement of the lemma follows from this identity, \eqref{eq:commutator-U-eval} and Proposition \ref{prop:evaluation-Ures}.
The plus case is proved analogously, using Remark \ref{rem:iso} and  that $\bq_{\beta \gamma} =1$ for all $\beta, \gamma \in \wfO^{\bq}_+$.
\epf

The last part of the proof gives the following fact about the structure of Lusztig algebras which is of independent interest. Recall
$\ee_{\beta}$ defined in \eqref{eq:fdual-definition}.

\begin{corollary} The braided Hopf algebra projection $\iota^* : \luq \twoheadrightarrow U(\n_{\bq}^-)$ (recall \eqref{eq:extension-braided-lu})
has an algebra section $U(\n_{\bq}^-) \to \luq$ given by
\begin{align*}
\iota^{*}(\ee_{\beta}^{(N_{\beta})}) &\mapsto \ee_{\beta}^{(N_{\beta})}, & &\beta\in\fO^{\bq}_{+}.
\end{align*}
\end{corollary}
\pf
By Proposition \ref{prop:evaluation-Ures} and \eqref{eq:commutator-divided-powers},
\begin{align}\label{eq:commutator-divided-powers-bis}
[\ee_{\beta}^{(N_{\beta})}, \ee_{\gamma}^{(N_{\gamma})}] &=\begin{cases}
\tfrac{a_{\beta\gamma}^{\delta}(\xi) s_{\delta} }{s_{\beta}s_{\gamma}}\ee_{\delta}^{(N_{\delta})}, & \text{ if }\exists \delta\in\fO^{\bq}_{+}: \wdelta=\wbeta+\wgamma,
\\
0, & \text{otherwise}.
\end{cases}
\end{align}

On the other hand, set $\mathtt{x}_{\wbeta}:=\iota^{*}(\ee_{\beta}^{(N_{\beta})})$. As $\n_{\bq}^-$ is the positive part of $\g_{\bq}$ and each $\mathtt{x}_{\wbeta}$ has weight $\wbeta$, there exist $a_{\wbeta\,\wgamma}\in\Bbbk$, $\wbeta,\wgamma\in\wfO^{\bq}_+$, such that 
\begin{align}\label{eq:defn-rels-n-bq}
[\mathtt{x}_{\wbeta},\mathtt{x}_{\wgamma}] &= \begin{cases}
a_{\wbeta\,\wgamma} \mathtt{x}_{\wbeta+\wgamma}, & \wbeta+\wgamma\in \wfO^\bq_{+},
\\
0, & \wbeta+\wgamma\notin \wfO^\bq_{+}.
\end{cases}
\end{align}
Applying $\iota^*$ to \eqref{eq:commutator-divided-powers-bis} we obtain that $a_{\wbeta\,\wgamma}=\tfrac{a_{\beta\gamma}^{\delta}(\xi) s_{\delta} }{s_{\beta}s_{\gamma}}$ 
for each pair $\wbeta,\wgamma\in\wfO^{\bq}_+$ such that $\wbeta+\wgamma\in \wfO^\bq_{+}$. 
Therefore the existence of the algebra map $U(\n_{\bq}^-) \to \luq$ as above follows since $U(\n_{\bq}^-)$ is presented by generators $\mathtt{x}_{\wbeta}$, $\wbeta\in\wfO^{\bq}_+$, and relations \eqref{eq:defn-rels-n-bq}, and the corresponding relations for $\ee_{\beta}^{(N_{\beta})}$ hold in $\luq$
by \eqref{eq:commutator-divided-powers-bis}.
\epf

\subsection{The dual tangent Lie bialgebra of $M_{\bq}$} 
\begin{lemma}
\label{lem:brackets} The following equalities hold in the Lie algebra $\mm_{\bq}^*$:
\begin{align*}
[ d_1(\Eq_\beta^{N_\beta}), d_1(\Fq_\gamma^{N_\gamma}) ] &= - \delta_{\beta \gamma}
\frac{\wp^{\bqf}_{\beta \beta}(\xi)}{(q_{\beta\beta}-1)^{N_\beta}}
(d_1(K_\beta^{N_\beta}) + d_1( L_\beta^{N_\beta})), &&
\wbeta, \wgamma \in \varPi^{\bq},
\end{align*}
and
\begin{align*}
[ d_1(K_\mu^{N_\mu}), d_1(\Eq_\beta^{N_\beta}) ] &= - \wp^{\bqf}_{\mu \beta}(\xi) d_1(\Eq_\beta^{N_\beta}), &
[ d_1(K_\mu^{N_\mu}), d_1(\Fq_\beta^{N_\beta}) ] &= \wp^{\bqf}_{\mu \beta}(\xi) d_1(\Fq_\beta^{N_\beta}),
\\
[ d_1(L_\mu^{N_\mu}), d_1(\Eq_\beta^{N_\beta}) ] &= - \wp^{\bqf}_{\beta \mu}(\xi) d_1(\Eq_\beta^{N_\beta}), &
[ d_1(L_\mu^{N_\mu}), d_1(\Fq_\beta^{N_\beta}) ] &=  \wp^{\bqf}_{\beta \mu}(\xi) d_1(\Fq_\beta^{N_\beta})
\end{align*}
for all $\beta \in \fO^{\bq}_{+}$, $\wmu \in \wt{\varPi}^{\bq}$
\end{lemma}
\pf 
The case of $\wbeta \neq \wgamma \in \varPi^{\bq}$ of the first identity follows from the fact that \eqref{eq:T*-basis}  is a basis of the Lie algebra $\mm_{\bq}^*$ 
and that the latter is $\QQ_{\bq}$-graded. 
The case $\wbeta = \wgamma \in \varPi^{\bq}$ is a consequence of \eqref{eq:bracket-e-f} since
$d_1( L_\beta^{N_\beta}) = - d_1( L_\beta^{-N_\beta})$, which in turn follows since the value of $L_\beta^{N_\beta}$
at the identity of $M_{\bq}$ equals $1$.
The other four identities follow from \eqref{eqn:g-con-Ei}--\eqref{eqn:g-con-Fi}.
\epf

Since the polynomials $\nu^n - a$ are separable over $\Cset$ for $a \neq 0$, we infer from \eqref{eq:qalphabeta-factorization} that 
\begin{align*}
&\wp^{\bqf}_{\beta \beta}(\xi)  \neq 0 && \mbox{for all} \; \; \be \in \fO^{\bq}_{+}.
\end{align*}

\begin{theorem}\label{thm:mm*} \begin{enumerate}[leftmargin=*,label=\rm{(\alph*)}]
\item\label{item:thm-a} The Cartan matrix of the semisimple Lie algebra $\g_{\bq}$ is given by
\begin{align*}
c_{\wbeta \, \wgamma} = \frac{ \wp^{\bqf}_{\beta \gamma}(\xi)  
+  \wp^{\bqf}_{\gamma \beta}(\xi) }{ \wp^{\bqf}_{\beta \beta}(\xi)}, && \wbeta, \wgamma \in \varPi^{\bq}.
\end{align*} 
\medbreak\item\label{item:thm-b} There is a ($\QQ_{\bq}$-graded) Lie algebra isomorphism $\wt{\g}_{\bq} \oplus \wt{\h}_{\bq} \simeq \mm_{\bq}^*$ such that
\begin{align*}
x_\beta \mt  d_1(\Fq_\beta^{N_\beta}), \; \; y_\beta \mt \frac{(q_\beta -1)^{N_\beta}}{\wp^{\bqf}_{\beta \beta}(\xi)^2} d_1(\Eq_\beta^{N_\beta}), \; \; 
h_\mu \mt \frac{1}{ \wp^{\bqf}_{\mu \mu}(\xi)} (d_1(K_\mu^{N_\mu}) + d_1( L_\mu^{N_\mu}))
\end{align*}
for $\wbeta \in \varPi^{\bq}$, $\wmu \in \wt{\varPi}^{\bq}$,
and $\wt{\h}_{\bq}$ maps to the subspace 
\[
\big\{ 
\sum_{{\mu} \in \wt{\varPi}^{\bq}} a_\mu d_1(K_\mu^{N_\mu}) + b_\mu d_1( L_\mu^{N_\mu})  : 
\sum_{{\wmu} \in \wt{\varPi}^{\bq}}  \wp^{\bqf}_{\mu \gamma}(\xi) a_\mu + \wp^{\bqf}_{\gamma \mu}(\xi) b_\mu =0, 
\forall \wgamma \in \wt{\varPi}^{\bq}
\big\}
\]
of the abelian Lie algebra $\oplus_{\wbeta \in \wt{\varPi^{\bq}}} ( \Cset d_1(K_\beta^{N_\beta}) + \Cset d_1( L_\beta^{N_\beta}) )$.
\end{enumerate}
\end{theorem}
\pf (a) For $\wbeta \in \varPi^{\bq}$ and $\wmu \in \wt{\varPi}^{\bq}$,
define the following elements of $\mm_{\bq}^*$: 
\begin{align*}
\wh{x}_\beta :=  d_1(\Fq_\gamma^{N_\gamma}), \; \; \wh{y}_\beta := \frac{(q_\beta -1)^{N_\beta}}{\wp^{\bqf}_{\beta \beta}(\xi)^2} d_1(\Eq_\beta^{N_\beta}), \; \; 
\wh{h}_\mu := \frac{1}{ \wp^{\bqf}_{\mu \mu}(\xi)} (d_1(K_\mu^{N_\mu}) + d_1( L_\mu^{N_\mu}))
\end{align*} 
and the Lie subalgebra $\g_\bq(\beta) := \Cset \wh{x}_\beta \oplus \Cset \wh{h}_\beta \oplus \Cset \wh{y}_\beta$. Lemma \ref{lem:brackets} 
implies that, for all $\wbeta \in \varPi^{\bq}$, $[\wh{h}_\be, \wh{x}_\beta] = 2 \wh{x}_\beta$,  $[\wh{h}_\be, \wh{y}_\beta] = - 2 \wh{y }_\beta$, 
$[\wh{x}_\be, \wh{y}_\beta] = \wh{h}_\beta$, so
$\g_{\bq}(\beta) \simeq {\mathfrak{sl}}_2$. 

Now take $\wbeta \neq \wgamma \in \varPi^{\bq}$ and consider $\g_{\bq}$ as a $\g_{\bq}(\beta)$-module under the adjoint action. 
It follows from Lemma \ref{lem:brackets} that 
\begin{align*}
&[\wh{x}_\beta, \wh{y}_\gamma] =0 \quad \mbox{and} \quad [\wh{h}_\beta,  \wh{y}_\gamma] 
= - \frac{ \wp^{\bqf}_{\beta \gamma}(\xi)  +  \wp^{\bqf}_{\gamma \beta}(\xi) }{ \wp^{\bqf}_{\beta \beta}(\xi)} \wh{x}_\gamma, 
\end{align*}
so $\wh{y}_\gamma$ is a highest weight vector for $\g_{\bq}(\beta) \simeq {\mathfrak{sl}}_2$ of weight 
$- \frac{ \wp^{\bqf}_{\beta \gamma}(\xi)  +  \wp^{\bqf}_{\gamma \beta}(\xi) }{ \wp^{\bqf}_{\beta \beta}(\xi)} \omega$ where 
$\omega$ denotes the fundamental weight of  ${\mathfrak{sl}}_2$. The isomorphism of Proposition \ref{prop:lem2} and the Serre relations in 
$\n_{\bq}^-$ imply that 
\begin{align*}
&\ad^{- c_{\beta \gamma} +1}_{\wh{y}_\beta} (\wh{y}_\gamma) =0 \quad \mbox{and} \quad 
\ad^{j}_{\wh{y}_\beta} (\wh{y}_\gamma)  \neq 0 \; \; \mbox{for} \; \; j \leqslant - c_{\beta \gamma}.
\end{align*}
Hence, $\ad^{- c_{\beta \gamma}}_{\wh{y}_\beta} (\wh{y}_\gamma)$ is the lowest weight vector of the (irreducible)
$\g_{\bq}(\beta)$-module generated by $\wh{y}_\gamma$, which forces 
\[
c_{\wbeta \, \wgamma} = \frac{ \wp^{\bqf}_{\beta \gamma}(\xi)  
+  \wp^{\bqf}_{\gamma \beta}(\xi)}{ \wp^{\bqf}_{\beta \beta}(\xi)} \cdot
\]
This proves part (a). It also proves that the assignment $x_\beta \mt \wh{x}_\beta$, 
$y_\beta \mt \wh{y}_\beta$, $h_\beta \mt \wh{h}_\beta$ for $\wbeta \in \varPi^{\bq}$ defines 
a  $\QQ_{\bq}$-graded Lie algebra homomorphism $\varphi : \g_{\bq} \to \mm_{\bq}^*$ which is an 
embedding by Proposition \ref{prop:lem2} and the linear independence of 
$\{ d_1(K_\be^{N_\be}), d_1(L_\be^{N_\be}) : \be \in \fO^{\bq}_{+} \}$. 
Here we use the canonical isomorphism $\n_{\bq}^\pm \to \n_{\bq}^\mp$
obtained by restricting the Chevalley involution of $\g_{\bq}^\pm$.

If $\bq$ is of type $\supera{k-1}{\theta-k}$, then $\wh{h}_\eta$ is in the center of $\mm_{\bq}^*$ by 
\eqref{peta0} and Lemma \ref{lem:brackets}. Furthermore, $\wh{h}_\eta \notin \varphi (\g_{\bq})$ 
by the definition of $d(K_\mu^{N_\mu})$ and $d(L_\mu^{N_\mu})$ for $\wmu \in \wt{\varPi}^{\bq}$. 
Hence, $\varphi$ extends to an embedding
\begin{equation}
\label{eq:mu}
\varphi : \wt{\g}_{\bq} \to \mm_{\bq}^*
\end{equation}
by setting $\varphi(h_\eta) := \wh{h}_\eta$ if $\bq$ is of type $\supera{k-1}{\theta-k}$. Denote
\[
(\mm_{\bq}^{0})^{*}  := \oplus_{\wmu \in \wt{\varPi}^{\bq}} \Big( \Cset d_1(K_\mu^{N_\mu}) \oplus \Cset d_1(L_\mu^{N_\mu}) \Big).
\]
Let $(\mm_{\bq}^{0})'$ be the intersection of the kernels of the functionals 
$\{ l_\gamma : \wgamma \in \wt{\varPi}^{\bq} \}$ on $(\mm_{\bq}^{0})^{*}$ given by
\begin{align*}
&l_\gamma( d_1(K_\mu^{N_\mu })) := \wp^{\bqf}_{\mu \gamma}(\xi), 
&& l_\gamma( d_1(L_\mu^{N_\mu }) ) := \wp^{\bqf}_{\gamma \mu}(\xi).
\end{align*}

Part (a) of the theorem, the constructed embedding \eqref{eq:mu}, and eq. \eqref{petanot0} imply that 
$(\mm_{\bq}^{0})'  \cap \Im \varphi =0$. Hence, $\dim (\mm_{\bq}^{0})' \leq \dim (\mm_{\bq}^{0})^* - \dim \wt{\h}_\bq = \dim \wt{\h}_{\bq}$. 

Since the number of the above functionals equals $|\wt{\varPi}^{\bq}| = \dim  \wt{\h}_{\bq} $, we have 
$$\dim (\mm_{\bq}^{0})' \geqslant \dim  \h_{\bq}.$$
It follows from part (a) that $\dim (\mm_{\bq}^{0})' \leqslant \dim  \h_{\bq}$, Hence
$$\dim \h_{\bq} \geq \dim (\mm_{\bq}^{0})^* - \dim \wt{\h}_\bq = \dim \wt{\h}_{\bq}.$$
Therefore $\dim  (\mm_{\bq}^{0})' = \dim \wt{\h_{\bq}} = \dim (\mm_{\bq}^{0})^*/2$ and $\mm_{\bq}^* = \varphi(\wt{\g}_\eta) \oplus (\mm_{\bq}^{0})'$. 
Taking a vector space isomorphism $\h_{\bq} \simeq (\mm_{\bq}^{0})'$ and combining it with the embedding $\varphi$, gives the needed Lie algebra isomorphism for part (b).
\epf

Let $(\cdot , \cdot )$ be the invariant symmetric bilinear form on $\g_{\bq}$ for which the induced form on the dual of the Cartan subalgebra of $\g_{\bq}$ 
satisfies $(\underline{\beta}, \underline{\beta} ) =2$ for short roots $\underline{\beta} \in \varPi^\bq$. 
We extend it to a non-degenerate invariant symmetric bilinear form on $\wt{\g}_{\bq}$, 
where in the case when $\bq$ is of type $\supera{k-1}{\theta-k}$ we let 
\begin{align*}
&(h_\eta, \g_\bq) =0,
&&(h_\eta, h_\eta)=2.
\end{align*}
In this case we identify 
\[
\wt{\h}_\bq \cong \h_\bq^* \oplus \Cset \weta,
\]
where $\weta \in \wt{\h}_\bq^*$ is such that $\langle \weta, h_\eta \rangle =2$ and $\langle \weta, h_{\wbeta} \rangle =0$ 
for all $\wbeta \in \varPi^\bq$.

We will identify $\wt{\h}_{\bq}$ with $\wt{\h}_{\bq}^*$
via the bilinear form $(.,.)$. Define the scalars
\[
\kappa_\mu : =  2 \wp^{\bqf}_{\mu \mu}(\xi) (\underline{\mu} , \underline{\mu})^{-1} 
\quad \quad \mbox{for} \quad \wmu \in \wt{\varPi}^\bq.
\]
For $\wbeta \in \varPi^\bq$, the scalar $\kappa_\beta$ 
only depends on the simple factor of $\wt{\g}_{\bq}$ of which $\wbeta \in \varPi^\bq$ is a root,
because by Theorem \ref{thm:mm*}(a), 
\begin{equation}
\label{c-to-p}
c_{\beta \gamma} = \frac{ \wp^{\bqf}_{\beta \gamma}(\xi)  +  \wp^{\bqf}_{\gamma \beta}(\xi) }{ \wp^{\bqf}_{\beta \beta}(\xi)} = 
\frac{2 (\underline{\beta}, \underline{\gamma})}{(\underline{\beta} , \underline{\beta})} \cdot
\end{equation}
If $\bq$ is of type $\supera{k-1}{\theta-k}$, then $( \weta, \weta) =2$.

Proposition \ref{prop:specialization-integer} (i) tells us that 
each large quantum group $U_{\bq}$ is realized as a specialization of an integral form of a one-parameter quantum 
group  $U_\bqf$ in infinitely many different ways parametrized by integers $t_{ij} \in\mathbb Z$ for $i<j\in\I$.
Furthermore, by part (ii) of that proposition, for a generic choice of the parameters $t_{ij} \in \mathbb Z$,  $i<j\in\I$, 
the matrix with entries $\wp^{\bqf}_{\beta \gamma}(\xi)$ for $\wbeta, \wgamma \in \varPi^{\bq}$ is non-degenerate.
In the remaining part of the paper we will assume the following:
\medskip

\begin{genericityassumption}\label{genericityassumption}
The specialization parameters $t_{ij} \in \mathbb Z$,  $i<j\in\I$ in 
Proposition \ref{prop:specialization-integer} are chosen in such a way that 
the matrix  $\WP^{\bqf}$ in \eqref{eq:matrix-P} is non-degenerate.
\end{genericityassumption}

It follows from \eqref{peta0}--\eqref{petanot0} that the matrix 
\[
\wt{\WP}^{\bqf}:=(\wp^{\bqf}_{\mu\gamma}(\xi))_{\wmu,\wgamma\in \wt{\varPi}^{\bq}}
\]
is invertible.

\begin{remark}
\label{rem:ident}
In what follows we will identify the Lie algebras 
\begin{equation}
\label{eq:ident}
\mm_{\bq}^* \simeq \wt{\g}_{\bq} \oplus \wt{\h}_{\bq}
\end{equation}
via 
the isomorphism from Theorem \ref{thm:mm*}. In particular, $x_\beta, y_\beta, h_\mu$ for $\wbeta \in \varPi^{\bq}, \wmu \in \wt{\varPi}^{\bq}$
will be viewed as elements of $\mm_{\bq}^*$.
We also fix the identification of abelian Lie algebras
\begin{equation}
\label{eq:h-ident}
\big\{ 
\sum_{\wmu \in \wt{\varPi}^{\bq}} a_\mu d_1(K_\mu^{N_\mu}) + b_\mu d_1( L_\mu^{N_\mu})  : 
\sum_{\wmu \in \wt{\varPi}^{\bq}}  \wp^{\bqf}_{\mu \gamma}(\xi) a_\mu + \wp^{\bqf}_{\gamma \mu}(\xi) b_\mu =0, 
\forall \wgamma \in \wt{\varPi}^{\bq}
\big\} \simeq \wt{\h}_{\bq}
\end{equation}
for Theorem  \ref{thm:mm*}(b)
by sending 
$\sum_{{\underline{\mu}} \in \wt{\varPi}^{\bq}} a_\mu d_1(K_\mu^{N_\mu}) + b_\mu d_1( L_\mu^{N_\mu})  \mt 
\sum_{{\underline{\mu}} \in \wt{\varPi}^{\bq}} b_\mu \kappa_\mu \underline{\mu}$,
using the identification of  $\wt{\h}_{\bq}$ with $\wt{\h}_{\bq}^*$ via the form $(.,.)$.
Since both Lie algebras in \eqref{eq:h-ident} have the same dimensions, we only need to show that this map is injective. An element in its kernel 
has $b_\mu=0$ for all $\wmu \in \wt{\varPi}^{\bq}$ and thus, 
$\sum  \wp^{\bqf}_{\mu \gamma}(\xi) a_\mu =0$ for all $\wgamma \in \wt{\varPi}^{\bq}$.
The Non-degeneracy Assumption \ref{genericityassumption} implies that $a_\mu =0$ for $\wmu \in \wt{\varPi}^{\bq}$.
\end{remark}
 
Consider the Borel subalgebras $\wt{\bg}_{\bq}^\pm := \n_\bq^\pm \oplus \wt{\h}_\bq$ of $\wt{\g}_{\bq}$. We have
\begin{equation}
\label{eq:m-sub-bh}
(\mm_{\bq}^{\geqslant})^{*} \subset \wt{\bg}_{\bq}^- \oplus \wt{\h}_{\bq} \quad \mbox{and} \quad 
(\mm_{\bq}^{\leqslant})^{*} \subset \wt{\bg}_{\bq}^+ \oplus \wt{\h}_{\bq}
\end{equation}
in the identification \eqref{m><ident} of $(\mm_{\bq}^{\geqslant})^{*}$ and $(\mm_{\bq}^{\leqslant})^{*}$ 
with Lie subalgebras of $\mm_{\bq}^*$.
Using the Non-degeneracy Assumption \ref{genericityassumption} one more time, we obtain
that the projection into the first component $\mm_{\bq}^* \simeq \wt{\g}_{\bq} \oplus \wt{\h}_{\bq} \to \wt{\g}_{\bq}$ restricts to the Lie algebra isomorphisms 
\begin{equation}
\label{eq:m-iso-b}
(\mm_{\bq}^{\geqslant})^{*} \simeq \wt{\bg}_{\bq}^- \quad \mbox{and} \quad 
(\mm_{\bq}^{\leqslant})^{*}  \simeq \wt{\bg}_{\bq}^+. 
\end{equation}
We next describe the embeddings \eqref{eq:m-sub-bh}. Denote the linear maps $\wt{\WP}, \wt{\WP}^{\, \mathrm{T}} \in \End(\wt{\h}_{\bq})$:
\begin{align}\label{eq:def-PT}
&\wt{\WP}(\underline{\mu}) := \sum_{{\underline{\gamma}} \in \wt{\varPi}^{\bq}}  \wp^{\bqf}_{\mu \gamma}(\xi) \underline{ \gamma}, 
&&\wt{\WP}^{\, \mathrm{T}}(\underline{\mu}): = \sum_{{\underline{\gamma}} \in \wt{\varPi}^{\bq}}  \wp^{\bqf}_{\gamma \mu}(\xi) \underline{ \gamma}.
\end{align}
Because of the Non-degeneracy Assumption \ref{genericityassumption}, the matrix $\wt{\WP}^{\bqf}$ is invertible, 
and thus both endomorphisms are invertible.

Denote by $(\!(\cdot , \cdot )\!)$ the invariant symmetric bilinear form on $\wt{\g}_{\bq}$, which is a rescaling 
of $(\cdot , \cdot )$ by $\kappa_{\mu}^{-1}$ on each simple factor of $\wt{\g}_{\bq}$ and on the one-dimensional 
center of $\wt{\g}_{\bq}$ if $\bq$ is of type $\supera{k-1}{\theta-k}$. It satisfies
\begin{align*}
&(\!( d_1(K_\mu^{N_\mu}) + d_1(L_\mu^{N_\mu}), d_1(K_\gamma^{N_\gamma}) + d_1(L_\gamma^{N_\gamma}) )\!)  =
\wp^{\bqf}_{\mu \mu}(\xi) \wp^{\bqf}_{\gamma \gamma}(\xi) (\!( h_\mu, h_\gamma )\!) 
\\
&= \wp^{\bqf}_{\mu \mu}(\xi) \wp^{\bqf}_{\gamma \gamma}(\xi) \kappa_\gamma^{-1} \frac{2 c_{\mu \gamma}}{ ( \underline{\gamma}, \underline{\gamma})} = 
\wp^{\bqf}_{\mu \gamma}(\xi)  +  \wp^{\bqf}_{\gamma \mu}(\xi), \quad \forall \wmu, \wgamma \in \wt{\varPi}^\bq.
\cdot
\end{align*}
This implies that the form $(\!(\cdot , \cdot )\!)$ has a unique extension to an invariant symmetric bilinear form on $\mm_{\bq}^*$ such that
\begin{align} 
&(\!( d_1(K_\mu^{N_\mu}),  d_1(L_\gamma^{N_\gamma}) )\!)  = \wp^{\bqf}_{\mu \gamma}(\xi), 
\label{form-LK}
\\
&(\!( d_1(K_\mu^{N_\mu}), d_1(L_\mu^{N_\mu}) )\!) = (\!( d_1(K_\gamma^{N_\gamma}), d_1(L_\gamma^{N_\gamma}) )\!)  = 0
\label{form-KK}
\end{align} 
for $\wmu, \wgamma \in \wt{\varPi}^{\bq}$. The Non-degeneracy Assumption \ref{genericityassumption} implies that the bilinear 
form $(\!(\cdot , \cdot )\!)$ on $\mm_{\bq}^*$ is non-degenerate. 

One easily verifies that the orthogonal complement in $\mm_{\bq}^*$ of $\wt{\g}_{\bq}$ equals $\wt{\h}_{\bq}$.
\begin{proposition}
\label{prop:m*} 
For all large quantum groups $U_{\bq}$ satisfying the Non-degeneracy Assumption \ref{genericityassumption}, 
the subalgebras $(\mm_{\bq}^{\geqslant})^{*} \subset \wt{\bg}_{\bq}^- \oplus \wt{\h}_{\bq}$ 
and $(\mm_{\bq}^{\leqslant})^{*} \subset \wt{\bg}_{\bq}^+ \oplus \wt{\h}_{\bq}$ are given by
\begin{align*}
&(\mm_{\bq}^{\geqslant})^{*} = \{ (y +h, - h) : y \in \n_{\bq}^-, h \in \wt{\h}_{\bq} \}, 
&&(\mm_{\bq}^{\leqslant})^{*} = \{ (x +h, \WP^{-1} \WP^{\, \mathrm{T}}(h)) : x \in \n_{\bq}^+, h \in \wt{\h}_{\bq} \}.
\end{align*}
\end{proposition}
\begin{proof} Denote the first (abelian) Lie algebra in \eqref{eq:h-ident} by $\wt{\h}_\bq^{(2)}$. Fix
\begin{align*}
&h := \sum_{{\underline{\mu}} \in \wt{\varPi}^{\bq}} c_\mu (d_1(K_\mu^{N_\mu}) + d_1( L_\mu^{N_\mu}) ),
&&h_1 := \sum_{{\underline{\mu}} \in \wt{\varPi}^{\bq}} a_\mu d_1(K_\mu^{N_\mu}), 
&&&h_2 := \sum_{{\underline{\mu}} \in \wt{\varPi}^{\bq}} b_\mu d_1(L_\mu^{N_\mu}).
\end{align*}
If $h_1 + h_2 \in \h_\bq^{(2)}$, then $l_\gamma(h_1) = - l_\gamma(h_2)$ for all $\underline{\gamma} \in \wt{\varPi}^{\bq}$, which is equivalent to
\begin{equation}
\label{eq:PPT}
\wt{\WP} \big( \sum_{{\underline{\mu}} \in \wt{\varPi}^{\bq}} a_\mu \underline{\mu} \big) = 
- \wt{\WP}^{\, \mathrm{T}} \big( \sum_{{\underline{\mu}} \in \wt{\varPi}^{\bq}} b_\mu \underline{\mu} \big).
\end{equation}

By Theorem \ref{thm:mm*}(b), in the identification \eqref{eq:ident}, $d_1(K_\mu^{N_\mu}) + d_1( L_\mu^{N_\mu})$ 
corresponds to $\kappa_\mu  \underline{\mu}$ for all $\wmu \in \wt{\varPi}^\bq$.
Hence, the first statement of the proposition is equivalent to proving that for all
$h$, $h_1$, $h_2$ as above, if $h_1 + h_2 \in \wt{\h}_\bq^{(2)}$ and $h + h_1 + h_2 \in (\mm_{\bq}^{\geqslant})^{*}$, 
then $c_\mu = - b_\mu$ for $\underline{\mu} \in \wt{\varPi}^{\bq}$.  
From the condition $h_1 + h_2 \in \wt{\h}_\bq^{(2)}$ we obtain $(\!( d_1( L_\gamma^{N_\gamma}), h + h_2 )\!) =0 $ for all $\underline{\gamma} \in \wt{\varPi}^{\bq}$.
Thus 
\begin{align*}
&\sum_{{\underline{\mu}} \in \wt{\varPi}^{\bq}}  \wp^{\bqf}_{\mu \gamma}(\xi) (c_\mu + b_\mu) =0,
&& \forall {\underline{\gamma}} \in \wt{\varPi}^{\bq}.
\end{align*}
Now the first statement of the proposition follows from the  Non-degeneracy Assumption \ref{genericityassumption}.
The second one follows from the first by interchanging the roles of $(\mm_{\bq}^{\geqslant})^{*}$ and $(\mm_{\bq}^{\leqslant})^{*}$ and applying 
\eqref{eq:PPT}.
\end{proof}

We next describe the Lie coalgebra structure on $\mm_{\bq}^*$ and the 
corresponding Manin triple; see \S \ref{subsec:B.1} for background. 

\begin{theorem}\label{thm:mm*-Manin}
 For every choice of the specialization parameters 
$t_{ij} \in \mathbb Z$ satisfying the Non-degeneracy Assumption \ref{genericityassumption} the following hold:
\begin{enumerate}[leftmargin=*,label=\rm{(\alph*)}]
\item\label{item:thm-Manin-a} 
The Lie coalgebra structure of the Lie bialgebra $\mm_{\bq}^*$ is given by
\begin{align*}
&\delta(x_\beta) = d_1(L_\beta^{N_\beta}) \wedge x_\beta, 
&&\delta(y_\beta) = d_1(K_\beta^{N_\beta}) \wedge y_\beta,
&&\delta( d_1(K_\mu^{N_\mu}) ) = \delta ( d_1(L_\mu^{N_\mu}) ) =0
\end{align*}
for all $\wbeta \in \varPi^{\bq}, \wmu \in \wt{\varPi}^{\bq}$.
\medbreak
\item\label{item:thm-Manin-b} With respect to the bilinear form 
$(\!(\cdot , \cdot )\!)$, $(\mm_{\bq}^*,  (\mm_{\bq}^{\leqslant})^{*}, (\mm_{\bq}^{\geqslant})^{*})$ is a Manin triple.
\medbreak
\item\label{item:thm-Manin-c} The Lie coalgebra structures of $(\mm_{\bq}^{\geqslant})^{*}$ and $(\mm_{\bq}^{\leqslant})^{*}$ satisfy
\begin{align}
&(\!( \delta(y), x_1 \otimes x_2 )\!) = - (\!( y, [x_1, x_2] )\!),
&&(\!( \delta(x), y_1 \otimes y_2 )\!) = (\!( x, [y_1, y_2] )\!)
\label{form-ident}
\end{align}
for all $x, x_1, x_2 \in (\mm_{\bq}^{\leqslant})^{*}$ and $y, y_1, y_2 \in (\mm_{\bq}^{\geqslant})^{*}$. 
\end{enumerate}
\end{theorem}
\begin{remark}
\label{rem:bialg-str} 
\begin{enumerate}[leftmargin=*,label=\rm{(\alph*)}]
\item Part (a) of the theorem uniquely determines the Lie coalgebra structures of $\mm_{\bq}^*$, $(\mm_{\bq}^{\leqslant})^{*}$ and $(\mm_{\bq}^{\geqslant})^{*}$, 
since the set 
\[
\{x_\beta, y_\beta, d_1(L_\mu^{N_\mu}), d_1(L_\mu^{N_\mu}) : \wbeta \in \varPi^{\bq}, \wmu \in \wt{\varPi}^{\bq} \}
\]
and its appropriate subsets generate the Lie algebras $\mm_{\bq}^*$, $(\mm_{\bq}^{\leqslant})^{*}$ and $(\mm_{\bq}^{\geqslant})^{*}$.
\item By part (c) of the theorem, the Lie coalgebra structures of $\mm_{\bq}^*$, $(\mm_{\bq}^{\leqslant})^{*}$ and $(\mm_{\bq}^{\geqslant})^{*}$, 
are precisely the ones that are associated to a Manin triple as in Remark \ref{Manin-Drinfeld}(c). In particular,
we have the isomorphism of Lie bialgebras
\begin{align}
\label{eq:Lie-bialg}
&\mm_{\bq}^* \simeq D((\mm_{\bq}^{\leqslant})^{*}), 
&& (\mm_{\bq}^{\geqslant})^{*} \simeq ( ((\mm_{\bq}^{\leqslant})^{*})^*)^{\mathrm{op}} \simeq (\mm_{\bq}^{\leqslant})^{\mathrm{op}}.
\end{align}
\item The Lie bialgebra structures on the reductive Lie algebras $\mm_{\bq}^* \simeq \wt{\g}_{\bq} \oplus \wt{\h}_{\bq}$ from part (a) of 
the theorem correspond to empty Belavin--Drinfeld triples 
 and arbitrary
choice of the continuous parameters in their classification \cite{BD}.  
\end{enumerate} 
\end{remark}
\begin{proof}[Proof of Theorem \ref{thm:mm*-Manin}] Part (a) follows from Lemma \ref{lem:tanget-coalg} and the identities
\begin{align}
\label{eq:copod-eNfN}
\Delta(\Eq_\beta^{N_\beta}) &= K_\beta^{N_\beta} \otimes \Eq_\beta^{N_\beta} + \Eq_\beta^{N_\beta} \otimes 1,&
\Delta(\Fq_\beta^{N_\beta}) &=1 \otimes \Fq_\beta^{N_\beta} + \Fq_\beta^{N_\beta} \otimes L_\beta^{-N_\beta}.
\end{align}
for $\wbeta \in \varPi^{\bq}$ and the fact that $K_\mu^{N_\mu}$ and $L_\mu^{N_\mu}$ for $\wmu \in \wt{\varPi}^{\bq}$
are group-like elements. 

(b) The subalgebras $(\mm_{\bq}^{\geqslant})^{*}$ and $(\mm_{\bq}^{\leqslant})^{*}$ are orthogonal to their nilradicals because of the embeddings 
\eqref{eq:m-sub-bh}.
This, combined with \eqref{form-KK}, implies that they are isotropic subalgebras of $\mm_{\bq}^*$ with respect to the form 
$(\!(\cdot , \cdot )\!)$. The direct sum decomposition $\mm_{\bq} \simeq \mm_{\bq}^\leqslant \oplus \mm_{\bq}^\geqslant$ yields the desired result.

(c) Part (a) of the theorem and the isomorphism in Theorem \ref{thm:mm*}(b) imply at once the validity of the identities 
\eqref{form-ident} for $y = d_1(e_\beta^{N_\beta})$, $y= d_1(K_\mu^{N_\mu})$, 
$x = d_1(f_\beta^{N_\beta})$, $x= d_1(L_\mu^{N_\mu})$, where $\wbeta \in \varPi^{\bq}$ $\wmu \in \wt{\varPi}^{\bq}$,
and for all possible choices of $x_1, x_2, y_2, y_2$.
The general case follows by induction on root height when $x$, $y$ are chosen to be root vectors 
by using the invariance of the bilinear form $(\!(\cdot , \cdot )\!)$.
\end{proof}
\subsection{The Poisson algebraic groups $M_{\bq}^\geqslant$ and $M_{\bq}^\leqslant$} 
Combining the isomorphisms \eqref{eq:m-iso-b} and \eqref{eq:Lie-bialg}, we get the Lie algebra isomorphisms
\begin{equation}
\label{eq:m-isos}
\mm_{\bq}^{\geqslant} \simeq ((\mm_{\bq}^{\leqslant})^{*})_{\mathrm{op}} \simeq (\wt{\bg}_{\bq}^+)_{\mathrm{op}} \simeq \wt{\bg}_{\bq}^+ \quad \mbox{and} 
\quad
\mm_{\bq}^{\leqslant} \simeq (\mm_{\bq}^{\geqslant})^{*} \simeq \wt{\bg}_{\bq}^-,
\end{equation} 
where $(.)_{\mathrm{op}}$ stands for the opposite Lie algebra structure and $(\wt{\bg}_{\bq}^+)_{\mathrm{op}} \simeq \wt{\bg}_{\bq}^+$
is the standard Lie algebra isomorphism $x \mt -x$. 
The proof of Proposition \ref{prop:m*} shows that the corresponding pull back maps on the level of duals send 
\begin{align}
\label{eq:dual-iso}
&\wbeta \mt - d_1(K_\beta^{N_\beta}), 
&&\wbeta \mt d_1(L_\beta^{N_\beta}), 
&& \forall \wbeta \in \varPi^{\bq}.
\end{align}
The scalars $\kappa_\beta$ do not appear here because the form $(\!(\cdot , \cdot )\!)$ is a rescaling of the form $(\cdot , \cdot )$ 
by $\kappa_{\mu}^{-1}$ on each simple factor of $\wt{\g}_{\bq}$ and on the one-dimensional 
center of $\wt{\g}_{\bq}$ if $\bq$ is of type $\supera{k-1}{\theta-k}$.

Denote by $G_{\bq}$ the adjoint semisimple algebraic group with Lie algebra $\g_{\bq}$. Let 
\[
\wt{G}_\bq = 
\begin{cases}
G_\bq \times \Cset^\times, & \mbox{if $\bq$ is of type $\supera{k-1}{\theta-k}$}
\\
G_\bq, & \mbox{otherwise}.
\end{cases}
\]
In the former case the exponential map
\begin{equation}
\label{exp-wt}
\exp : \wt{\g}_\bq \to \wt{G}_\bq \quad \mbox{is given by} \quad
\exp( x + c h_\eta) = (\exp (x), \exp(c)),  
\end{equation}
where in the first component in the right hand side we use the exponential 
map $\exp : \g_\bq \to G_\bq$, and $x \in \g_\bq$, $c \in \Cset$.  
Denote by $\wt{B}^\pm_{\bq}$ the Borel subgroups of $\wt{G}_\bq$ corresponding to $\wt{\bg}_{\bq}^\pm$
and by $B^\pm_{\bq}$ the Borel subgroups of $G_\bq$ corresponding to $\wt{\bg}_{\bq}^\pm \cap \g_\bq$.
We have $\wt{B}_\bq^\pm = B_\bq^\pm \times \Cset^\times$ if $\bq$ is of type $\supera{k-1}{\theta-k}$,
and $\wt{B}_\bq^\pm = B_\bq^\pm$ otherwise.

Let $\wt{T}_{\bq} := \wt{B}^+_{\bq} \cap \wt{B}^-_{\bq}$ be the corresponding maximal torus of $\wt{G}_{\bq}$. 
Denote by $N_{\bq}^\pm \subset G_\bq$ the unipotent radicals of $\wt{B}^\pm_{\bq}$.

The groups of group-like elements of $Z_{\bq}^\geqslant$ and $Z_{\bq}^\leqslant$ are the free abelian groups 
on $K_\mu^{\pm N_\mu}$,  $\wbeta \in \wt{\varPi}^{\bq}$ and $L_\mu^{\pm N_\mu}$, $\mu \in \wt{\varPi}^{\bq}$, 
respectively. 
\begin{theorem}\label{thm:Ms-to-Bs}  For every choice of the specialization parameters 
$t_{ij} \in \mathbb Z$ satisfying the Non-degeneracy Assumption \ref{genericityassumption}, 
the Lie algebra isomorphisms \eqref{eq:m-isos} integrate to isomorphisms of algebraic groups.
\[
\tau_+ : M_{\bq}^\geqslant \stackrel{\simeq}{\longrightarrow} \wt{B}_{\bq}^+ 
\quad \mbox{and} \quad \tau_- : M_{\bq}^\leqslant \stackrel{\simeq}{\longrightarrow} \wt{B}_{\bq}^-.
\]
\end{theorem} 

Theorem \ref{thm:Ms-to-Bs} describes explicitly the algebraic groups $M_{\bq}^\geqslant$ and $M_{\bq}^\leqslant$. As an algebraic group, 
$M_{\bq} \simeq \wt{B}_{\bq}^+ \times \wt{B}_{\bq}^-$. The Poisson structures on $M_{\bq}^\geqslant$, $M_{\bq}^\leqslant$ and $M_{\bq}$
are the unique Poisson algebraic group structures that integrate the Lie bialgebras $\mm_{\bq}^\geqslant$, $\mm_{\bq}^\leqslant$ and $\mm_{\bq}$, whose
dual Lie bialgebras are described in Theorem \ref{thm:mm*-Manin}.

\begin{proof} We prove the first statement, the second being analogous.
Since $G_{\bq}$ is of adjoint type, the Borel subgroup $B_{\bq}^+$ is canonically identified with the identity component of $\Aut(\bg_{\bq}^+)$. The adjoint action of 
$M_{\bq}^\geqslant$ on $\mm_{\bq}^\geqslant \simeq \bg_{\bq}^+$ induces a surjective homomorphism 
$\tau_+^1 : M_{\bq}^\geqslant  \twoheadrightarrow B_{\bq}^+$. If $\bq$ is of type $\supera{k-1}{\theta-k}$, then we also have a canonical 
surjective homomorphism $\tau_+^2 : M_{\bq}^\geqslant  \twoheadrightarrow \Cset^\times$, whose pull back 
map $\Cset[\Cset^\times] \cong \Cset[\chi^{\pm 1}] \to \Cset [ M_{\bq}^\geqslant]$
is given by $\chi \mapsto K_\eta^{N_\eta}$, where $\chi$ is the identity character of $\Cset^\times$.
Define the homomorphism $\tau_+ : M_{\bq}^\geqslant  \twoheadrightarrow \wt{B}_{\bq}^+$ given by
\[
\tau_+ :=
\begin{cases}
(\tau_+^1, \tau_+^2), & \mbox{if $\bq$ is of type $\supera{k-1}{\theta-k}$}
\\
\tau_+^1, & \mbox{otherwise}.
\end{cases}
\]
It follows from \eqref{zqu-Laurent} that the homomorphism $\tau_+$ is surjective. It restricts to an isomorphism 
$\tau_+ : N(M_{\bq}^\geqslant) \stackrel{\simeq}{\longrightarrow} N_{\bq}^+$, where $N(M_{\bq}^\geqslant)$ is the unipotent radical of 
$M_{\bq}^\geqslant$. The homomorphism $\tau_+$ also restricts to a surjective homomorphism 
\begin{equation}
\label{eq:surj-tau}
\tau_+ : T(M_{\bq}^\geqslant)  \twoheadrightarrow \wt{T}_{\bq},
\end{equation}
where $T(M_{\bq}^\geqslant)$ is a maximal torus of $M_{\bq}^\geqslant$. The tori $T(M_{\bq}^\geqslant)$ and $\wt{T}_{\bq}$ are connected because 
$M_{\bq}^\geqslant$ and $\wt{B}_{\bq}^+$ are connected algebraic groups. In view of the Levi decompositions of 
$M_{\bq}^\geqslant$ and $\wt{B}_{\bq}^+$, in order to prove that $\tau_+$ is an isomorphism, it is sufficient to show that the restriction  
\eqref{eq:surj-tau} is an isomorphism. However, 
\begin{align*}
&\Cset[T(M_{\bq}^\geqslant)] \simeq \Cset[M_{\bq}^\geqslant/N(M_{\bq}^\geqslant)] \simeq \Cset[ G(\Cset[M_{\bq}^\geqslant])],
&&\Cset[\wt{T}_{\bq}] \simeq \Cset[\wt{B}_{\bq}^+/N_{\bq}^+] \simeq \Cset[ G(\Cset[\wt{B}_{\bq}^+])],
\end{align*}
where $G(H)$ denotes the group of group-like elements of a Hopf algebra $H$. 

The group of group-like elements of 
$\Cset[M_{\bq}^\geqslant] \simeq Z_{\bq}^\geqslant$ is the free abelian group with generators $K_\be^{N_\be}$, $\wbeta \in \varPi^{\bq}$.
The group of group-like elements of $\Cset[\wt{B}_{\bq}^+]$ equals the character lattice of $\wt{B}^+_\bq$, which is 
canonically identified with the lattice $\Zset \wt{\varPi}^{\bq}$.
The differentials at the identity element of the two generating sets are respectively $d_1 (K_\mu^{N_\mu})$ and $\wmu$, 
where $\wmu \in \wt{\varPi}^{\bq}$. Eq. \eqref{eq:dual-iso} implies that $\tau_+^* : G(\Cset[\wt{B}_{\bq}^+]) \to G(\Cset[M_{\bq}^\geqslant])$ is an isomorphism. 
Hence, $\tau_+^* : \Cset[\wt{T}_{\bq}] \to \Cset[T(M_{\bq}^\geqslant))]$ is an isomorphism and the same holds for \eqref{eq:surj-tau}. This completes the proof of the theorem. 
\end{proof}

\begin{example}\label{exa:bgl4-Ms-to-Bs}
Let $\bq$ be of type $\Bgl(4)$ and fix $N=\ord q$, $M=\ord (-q)$, see \S \ref{sec:by-diagram-modular-char3}. 
Let $\gamma=\alpha_1+2\alpha_2+3\alpha_3+\alpha_4$. Then $N_{\alpha_1}=N_{\alpha_2}=N$, $N_{\alpha_4}=N_{\gamma}=M$,
\begin{align*}
\wfO^{\bq}_+ &
=\{\walpha_1,\walpha_2, \walpha_1+\walpha_2, \walpha_4, \wgamma, \walpha_4+\wgamma\}.
\end{align*}
As shown previously, $\Delta(\Eq_\beta^{N_\beta})= \Eq_\beta^{N_\beta} \otimes 1+K_\beta^{N_\beta} \otimes \Eq_\beta^{N_\beta}$
for $\wbeta \in \varPi^{\bq}=\{\walpha_1,\walpha_2, \walpha_4, \wgamma\}$. We can check that
$\Eq_{\alpha_1+\alpha_2}=[\Eq_1,\Eq_2]_c$, $\Eq_{\alpha_4+\gamma}=[\Eq_{\gamma},\Eq_4]_c$ and
\begin{align*}
\Delta(\Eq_{\alpha_1+\alpha_2}^{N})= \Eq_{\alpha_1+\alpha_2}^{N} \otimes 1 +(q-1)^N \Eq_{\alpha_1}^{N} K_{\alpha_2}^N \otimes \Eq_{\alpha_2}^{N}
+K_{\alpha_1}^NK_{\alpha_2}^N \otimes \Eq_{\alpha_1+\alpha_2}^{N}, 
\\
\Delta(\Eq_{\alpha_4+\gamma}^M)= \Eq_{\alpha_4+\gamma}^M \otimes 1 + (q+1)^M \Eq_{\gamma}^M K_{\alpha_4}^M \otimes \Eq_{\alpha_4}^M
+K_{\alpha_4}^M K_{\gamma}^M \otimes \Eq_{\alpha_4+\gamma}^M.
\end{align*}

We now construct an explicit isomorphism between $Z_{\bq}^\geqslant$ and the algebra of functions over the Borel subgroup of $\operatorname{PSL}_3(\Cset)\times \operatorname{PSL}_3(\Cset)$. 
We consider the Levi decomposition $\widetilde{B}_3 \simeq N_3 \rtimes \widetilde{T}_3$ of the Borel subgroup of $\operatorname{SL}_3(\Cset)$, 
where 
\begin{align*}
\widetilde{T}_3 &= \{ \mathbf{t} =\diag(t_1, t_2, t_3) : t_i \in \Cset^\times, t_1 t_2 t_3 =1 \}, &
N_3 &= \Big\{ \mathbf{n} = 
\begin{pmatrix}
1 & t_{12} & t_{13} \\
0 & 1 & t_{12} \\
0 & 0 & 1 
\end{pmatrix}
: t_{ij} \in \Cset \Big\}.
\end{align*}
Let $a_i, x_{ij}: \widetilde{B}_3 \to \Cset$ be the coordinate functions sending $\mathbf{t} \mapsto t_i$
and $\mathbf{n} \mapsto t_{ij}$ respectively.
The coproducts of these coordinate functions are given by $\Delta(a_i) = a_i \otimes a_i$ and
\begin{align*}
\Delta(x_{12})&=  x_{12} \otimes 1 + a_1 a_2^{-1} \otimes x_{12}, \quad \quad 
\Delta(x_{23})= x_{23} \otimes 1 + a_2 a_3^{-1} \otimes x_{23}, 
\\
\Delta(x_{13})&= x_{13} \otimes 1 + x_{12} a_2 a_3^{-1} \otimes x_{23} + a_1 a_3^{-1} \otimes x_{13}.
\end{align*}
Denote $\Zset_3 = \langle ( \zeta, \zeta, \zeta) \rangle$, where $\zeta$ is a primitive 3rd root of unity. 
The Borel subgroup $B_3$ of 
$\operatorname{PSL}_3(\Cset)$ has Levi decomposition $B_3 \simeq N_3 \rtimes T_3$ where $T_3= \widetilde{T}/\Zset_3$, so 
\[
\Cset[B_3] = \Cset[N_3] \otimes \Cset[\widetilde{T}_3]^{\Zset_3} = 
\Cset[x_{12}, x_{23}, x_{13}, a_{12}^{\pm 1}, a_{23}^{\pm 1}],
\]
where $a_{12} := a_1 a_2^{-1}$ and $a_{23}= a_2 a_3^{-1}$. The coproducts of the coordinate functions on $B_3$ are given by $\Delta(a_{i i+1}) = a_{i i+1} \otimes a_{i i+1}$
and
\begin{align*}
\Delta(x_{12})&=  x_{12} \otimes 1 + a_{12} \otimes x_{12}, \quad \quad \Delta(x_{23})= x_{23} \otimes 1 + a_{23} \otimes x_{23}, 
\\
\Delta(x_{13})&= x_{13} \otimes 1 + x_{12} a_{23}  \otimes x_{23} + a_{12} a_{23} \otimes x_{13}.
\end{align*}
The Borel subgroup of $\operatorname{PSL}_3(\Cset)\times \operatorname{PSL}_3(\Cset)$ is isomorphic 
to $B_3 \times B_3$. We denote the coordinate functions $a_{i i+1}$ and $x_{ij}$
on the first and second copy of $B_3$ by superscripts 1 and 2.  
Now, clearly the map $\tau_+:Z_{\bq}^\geqslant \to \Cset[B_3 \times B_3]$ given by
\begin{align*}
K_{\alpha_1}^N&\mapsto a_{12}^1, &
K_{\alpha_2}^N&\mapsto a_{23}^1, &
K_{\alpha_3}^N&\mapsto a_{12}^2, &
K_{\gamma}^M&\mapsto a_{23}^2
\end{align*}
and
\begin{align*}
\Eq_{\alpha_1}^{N}&\mapsto x_{12}^1, & \Eq_{\alpha_2}^{N}&\mapsto x_{23}^1, & \frac{\Eq_{\alpha_1+\alpha_2}^{N}}{(q-1)^N}&\mapsto x_{13}^1,
\\
\Eq_{\alpha_4}^{M}&\mapsto x_{12}^2, & \Eq_{\gamma}^{M}&\mapsto x_{23}^2, & \frac{\Eq_{\alpha_4+\gamma}^{M}}{(q+1)^M}&\mapsto x_{13}^2
\end{align*}
is a Hopf algebra isomorphism. 
\end{example}
\section{Poisson geometry and representations}\label{sec:Poisson-geom}
In this section we describe the symplectic foliations and the torus orbits of symplectic leaves of the
Poisson algebraic groups $M_{\bq}$, $M_{\bq}^{\geqslant}$ and $M_{\bq}^{\leqslant}$, and the Poisson 
homogeneous spaces $M_{\bq}^{+}$ and $M_{\bq}^{-}$. 
Previous work in this direction dealt with the so called standard Poisson structures on simple algebraic groups (and their Borel subgroups) \cite{HL}, the dual Poisson algebraic groups \cite{DKP} 
and the related flag varieties \cite{GY}. See also \cite{CoV,DL,HLT}. 
The Poisson structures in Remark \ref{rem:bialg-str} are not of standard type in general and the results in this section 
can not be deduced from \cite{HL,DKP,GY}.   
For $z \in M_{\bq}$, respectively $M_{\bq}^{\geqslant}$, $M_{\bq}^{\leqslant}$, $M_{\bq}^{+}$, $M_{\bq}^{-}$, let 
$\hopf_z$, respectively $\hopf_z^{\geqslant}$, $\hopf_z^{\leqslant}$, $\hopf_z^{+}$, $\hopf_z^{-}$ be the algebra defined in
Theorem \ref{mainth:summary} \ref{item-mainth:findimalg-reps}, respectively \eqref{eq:def-small-algebra}.
The Poisson geometric results described above
provide information on the irreducible representations 
of the large quantum groups $U_{\bq}$ by reduction to the sheaf of algebras $\hopf_z$, $z \in M_{\bq}$.
Analogous results hold for $U_{\bq}^{\geqslant}$, $U_{\bq}^{\leqslant}$,
and  $U_{\bq}^{\pm}$.

\subsection{Representations of the large quantum groups  and symplectic foliations}\label{Poisson-geom-M} 
The Manin triple described in Theorem \ref{thm:mm*-Manin} and the identification $\mm_{\bq}^* \simeq \wt{\g}_{\bq} \oplus \wt{\h}_{\bq}$ 
equip $\wt{\g}_{\bq} \oplus \wt{\h}_{\bq}$ with a quasitriangular Lie bialgebra structure, which turns $\wt{G}_{\bq} \times \wt{T}_{\bq}$ into a Poisson algebraic group. 
The Poisson structure on $\wt{G}_{\bq} \times \wt{T}_{\bq}$ 
equals $L_g(r) - R_g(r)$ for $g \in \wt{G}_{\bq} \times \wt{T}_{\bq}$, where $r \in \wedge^2 (\wt{\g}_{\bq} \oplus \wt{\h}_{\bq})$ 
is the $r$-matrix for the Lie bialgebra structure on $\wt{\g}_{\bq} \oplus \wt{\h}_{\bq}$, 
and $L_g(-)$ and $R_g(-)$  refer to the left and right-invariant bivector fields on $\wt{G}_{\bq} \times \wt{T}_{\bq}$. 

\bigbreak
Let $\widetilde{M}_{\bq}^{\geqslant}$ and $\widetilde{M}_{\bq}^{\leqslant}$ be the connected Lie subgroups of $\wt{G}_{\bq} \times \wt{T}_{\bq}$ with Lie algebras 
$(\mm_{\bq}^{\leqslant})^{*}$ and $(\mm_{\bq}^{\geqslant})^{*}$. 
Proposition \ref{prop:m*} implies that $\widetilde{M}_{\bq}^{\leqslant}$ is an algebraic subgroup, while 
$\widetilde{M}_{\bq}^{\geqslant}$ is not necessarily a closed Lie subgroup. 
The projection onto the first component $\pi: \wt{G}_{\bq} \times \wt{T}_{\bq}  \twoheadrightarrow  \wt{G}_{\bq}$ gives the surjective Lie group homomorphisms 
\begin{align*}
&\pi_+ : \widetilde{M}_{\bq}^{\geqslant}  \twoheadrightarrow  \wt{B}_{\bq}^+,
&&\pi_- : \widetilde{M}_{\bq}^{\leqslant}  \twoheadrightarrow \wt{B}_{\bq}^-.
\end{align*}
Since $G_{\bq}$ is of adjoint type, the kernel of the exponential map $\exp : \h_{\bq} \twoheadrightarrow T_{\bq}$ equals $2 \pi i P_{\bq}\spcheck$, where $P_{\bq}\spcheck$
denotes the coweight lattice of $\g_{\bq}$. It follows from \eqref{exp-wt} that the kernel of the exponential map 
$\exp : \wt{\h}_{\bq} \twoheadrightarrow \wt{T}_{\bq}$ equals $2 \pi i \wt{P}_{\bq}\spcheck$, where $\wt{P}_{\bq}\spcheck$
is the coweight lattice of $\wt{\g}_{\bq}$ given by
\[
\wt{P}_{\bq}\spcheck = 
\begin{cases}
\wt{P}_{\bq}\spcheck \oplus \Zset h_\eta, & \mbox{if $\bq$ is of type $\supera{k-1}{\theta-k}$}
\\
\wt{P}_{\bq}\spcheck, & \mbox{otherwise}.
\end{cases}
\]

Denote the subgroup
\begin{align}
\label{eq:Cbq}
\wt{C}_{\bq}: = \exp \big( 2 \pi i \wt{\WP}^{-1} \wt{\WP}^{\, \mathrm{T}}(P_{\bq}\spcheck) \big) \subset \wt{T}_{\bq},
\end{align}
cf. \eqref{eq:def-PT}.
Proposition \ref{prop:m*} and the solvability of $\widetilde{M}_{\bq}^{\geqslant}$ and $\widetilde{M}_{\bq}^{\leqslant}$ give that 
\begin{align*}
\widetilde{M}_{\bq}^{\geqslant} &= (N_{\bq}^+ \times \{ 1 \}) \times \{ \exp (h, \wt{\WP}^{-1} \wt{\WP}^{\, \mathrm{T}}(h) ) : h \in \wt{\h}_{\bq} \}, 
\\
\widetilde{M}_{\bq}^{\leqslant} &= (N_{\bq}^- \times \{ 1 \}) \times \{ (t, t^{-1}) : t \in \wt{T}_{\bq} \}, 
\end{align*} 
from which one obtains that 
\begin{align*}
&\Ker \pi_+ = \{1\}, 
&& \Ker \pi_- = \{ 1\} \times \wt{C}_{\bq}.
\end{align*}
Composing $\pi_\pm$ with the isomorphisms from Theorem \ref{thm:Ms-to-Bs} leads to the isomorphisms 
\begin{align*}
&\tau_+^{-1} \pi_+  :  \widetilde{M}_{\bq}^{\geqslant} \stackrel{\simeq}{\longrightarrow} M_{\bq}^{\geqslant}, 
&\tau_-^{-1} \pi_-  :  \widetilde{M}_{\bq}^{\leqslant}/\Ker \pi_-  \stackrel{\simeq}{\longrightarrow} M_{\bq}^{\leqslant}.
\end{align*}
Their inverses give the canonical embeddings
\begin{align}
\label{eq:iota}
&j_+ : M_{\bq}^{\geqslant} \hookrightarrow \wt{G}_{\bq} \times \big( \wt{T}_{\bq}/ \wt{C}_{\bq} \big), 
&&j_- : M_{\bq}^{\leqslant} \hookrightarrow \wt{G}_{\bq} \times \big( \wt{T}_{\bq}/ \wt{C}_{\bq} \big).
\end{align}
Here we use that $\wt{G}_{\bq} \times \big( \wt{T}_{\bq}/ \wt{C}_{\bq} \big) \simeq \big( \wt{G}_{\bq} \times \wt{T}_{\bq} \big) / \Ker \pi_-$
and $\widetilde{M}_{\bq}^{\leqslant} \cap \Ker \pi_- = \{ 1 \}$.

\begin{remark}
If the matrix $\bqf$ is symmetric, then so is the matrix $\WP^{\bqf}$. This implies that
$\wt{\WP} = \wt{\WP}^{\, \mathrm{T}}$ and that the group $\wt{C}_{\bqf}$ is trivial. Then 
the continuous parameter accompanying the BD triple is as in Example \ref{exa:standard-poisson} and the Poisson structure is the standard one.
\end{remark}

\begin{theorem}\label{Poisson0}
Let $U_{\bq}$ be a large quantum group. For every choice of the specialization parameters 
$t_{ij} \in \mathbb Z$ satisfying the Non-degeneracy Assumption \ref{genericityassumption} the following hold:
\begin{enumerate}[leftmargin=*,label=\rm{(\alph*)}]
\item The symplectic leaves of the 
Poisson algebraic group $M_{\bq} \simeq M_{\bq}^{\geqslant} \times M_{\bq}^{\leqslant}$ are the inverse images $j^{-1} (\Oc \times {t} )$
under the map 
\begin{align*}
&j : M_{\bq} \to \wt{G}_{\bq} \times \big( \wt{T}_{\bq}/ \wt{C}_{\bq} \big),
&& j( m_+, m_-) := j_+(m_+)^{-1} j_- (m_-), 
&&&m_+ \in  M_{\bq}^{\geqslant}, m_- \in M_{\bq}^{\leqslant}, 
\end{align*}  
where $\Oc$ is a conjugacy class of $\wt{G}_{\bq}$ and $t \in \wt{T}_{\bq}/ \wt{C}_{\bq}$. The dimension of the symplectic leaf $j^{-1}(\Oc \times \{t\})$ equals $\dim \Oc$. 
\item If $j(z)$ and $j(z')$ are in the same conjugacy class of $\wt{G}_{\bq} \times \big( \wt{T}_{\bq}/ \wt{C}_{\bq} \big)$, then there is an algebra isomorphism
\[
\hopf_z \simeq \hopf_{z'}.
\]
\end{enumerate} 
\end{theorem}
Note that, since $\wt{T}_{\bq}/ \wt{C}_{\bq}$ is abelian, each conjugacy class of $\wt{G}_{\bq} \times \big( \wt{T}_{\bq}/ \wt{C}_{\bq} \big)$ has the form 
$\Oc \times \{ t \}$, where $\Oc$ is a conjugacy class of $\wt{G}_{\bq}$ and $t \in \big( \wt{T}_{\bq}/ \wt{C}_{\bq} \big)$.
\begin{proof} (a) By \cite{RSTS}, since the Poisson algebraic group $\wt{G}_{\bq} \times \wt{T}_{\bq}$ is quasitriangular, its 
double Poisson algebraic group is canonically isomorphic to
\[
D(\wt{G}_{\bq} \times \wt{T}_{\bq}) \simeq \big(\wt{G}_{\bq} \times \wt{T}_{\bq} \big) \times  \big(\wt{G}_{\bq} \times \wt{T}_{\bq} \big). 
\]
Theorem \ref{thm:mm*-Manin}(b) implies that the dual Poisson Lie group of $\wt{G}_{\bq} \times \wt{T}_{\bq}$ is 
\[
\widetilde{M}_{\bq}^{\geqslant} \times \widetilde{M}_{\bq}^{\leqslant} \hookrightarrow  \big(\wt{G}_{\bq} \times \wt{T}_{\bq} \big) \times  \big(\wt{G}_{\bq} \times \wt{T}_{\bq} \big)
\]
with the opposite Poisson structure to the restriction of the one of the double.  
Both $\widetilde{M}_{\bq}^{\geqslant} \times \widetilde{M}_{\bq}^{\leqslant}$ and $M_{\bq} \simeq M_{\bq}^{\geqslant} \times M_{\bq}^{\leqslant}$ have the same tangent 
Lie bialgebra, hence the map 
\[
\tau:=(\tau_+^{-1} \pi_+, \tau_-^{-1} \pi_-) : \widetilde{M}_{\bq}^{\geqslant} \times \widetilde{M}_{\bq}^{\leqslant} \twoheadrightarrow M_{\bq}^{\geqslant} \times M_{\bq}^{\leqslant} \simeq M_{\bq}
\]
is a Poisson covering map. By the Semenov-Tian-Shansky dressing method \cite{STS}, 
we get that the symplectic leaves of $\widetilde{M}_{\bq}^{\geqslant} \times \widetilde{M}_{\bq}^{\leqslant}$ are the connected components of the intersections 
\[
\widetilde{M}_{\bq}^{\leqslant} \cap \big( \diag (\wt{G}_{\bq} \times \wt{T}_{\bq} ) \cdot g \cdot \diag (\wt{G}_{\bq} \times \wt{T}_{\bq} ) \big), 
\]
where $\diag \big(\wt{G}_{\bq} \times \wt{T}_{\bq} \big)$ denotes the diagonal of $(\wt{G}_{\bq} \times \wt{T}_{\bq})^{\times 2}$ and $g \in (\wt{G}_{\bq} \times \wt{T}_{\bq})^{\times 2}$. 
Now we apply \cite[Theorem 1.10]{Y} to obtain that each such intersection is a dense, open and connected subset of 
$\diag \big(\wt{G}_{\bq} \times \wt{T}_{\bq} \big) \cdot g \cdot \diag \big(\wt{G}_{\bq} \times \wt{T}_{\bq} \big)$.
Consider the map
\begin{align*}
&\widetilde{j} : \widetilde{M}_{\bq}^{\geqslant} \times \widetilde{M}_{\bq}^{\leqslant}  \to \wt{G}_{\bq} \times \wt{T}_{\bq}, 
&& \widetilde{j}(m_+, m_-) := m_+^{-1} m_-, 
&&& m_+ \in \widetilde{M}_{\bq}^{\geqslant}, m_- \in \widetilde{M}_{\bq}^{\leqslant}.  
\end{align*}
By a direct argument we conclude that each symplectic leaf of $\widetilde{M}_{\bq}^{\geqslant} \times \widetilde{M}_{\bq}^{\leqslant}$ is of the form 
\[
\Ss_{\Oc'} := \big( \widetilde{M}_{\bq}^{\geqslant} \times \widetilde{M}_{\bq}^{\leqslant} \big) \cap \Oc',
\]
where $\Oc'$ is a conjugacy class of $G_{\bq} \times T_{\bq}$, and that
\[
\dim \Ss_{\Oc'} = \dim \Oc'.
\]
Since $\tau : \widetilde{M}_{\bq}^{\geqslant} \times \widetilde{M}_{\bq}^{\leqslant} \twoheadrightarrow M_{\bq}$ is a covering of Poisson 
Lie groups, each symplectic leaf of $M_{\bq}$ is of the form $\tau(\Ss_{\Oc'})$. One easily verifies that the diagram 
\begin{align*}
\begin{aligned}
\xymatrix@C=50pt{\widetilde{M}_{\bq}^{\geqslant} \times \widetilde{M}_{\bq}^{\leqslant} \ar@{->}[r]^{\widetilde{j}} \ar@{->}[d]_{\tau} & 
\wt{G}_{\bq} \times \wt{T}_{\bq} \ar@{->}[d]^{\psi}
\\ M_{\bq} \ar@{->}[r]^{j} & \wt{G}_{\bq} \times \big( \wt{T}_{\bq}/ \wt{C}_{\bq} \big)}
\end{aligned}\end{align*} 
commutes, where $\psi:  \wt{G}_{\bq} \times \wt{T}_{\bq} \twoheadrightarrow \wt{G}_{\bq} \times \big( \wt{T}_{\bq}/ \wt{C}_{\bq} \big)$ is the canonical projection.
Clearly, $\psi(\Oc') = \Oc \times \{ t \}$, where $\Oc$ is a conjugacy class of $\wt{G}_{\bq}$ and $t \in \wt{T}_{\bq}/ \wt{C}_{\bq}$. 
Therefore all symplectic leaves of $M_{\bq}$ are of the form $\tau(\Ss_{\Oc'}) = j^{-1} \psi(\Oc') = j^{-1}(\Oc \times \{ t \})$
and $\dim j^{-1}(\Oc \times \{ t \}) = \dim \Oc' = \dim \Oc$. 

\medbreak
Part (b) follows from part (a),  and Theorems \ref{theorem:BG-reps} and  \ref{thm:Poisson-orders}.
\end{proof}
In regard to the irreducible representations of $U_{\bq}$ we wonder whether the De Concini--Kac--Procesi conjecture 
could be extended to the setting of Theorem \ref{Poisson0}, see \cite{DKP}.
\begin{question} Let $\Oc$ be a conjugacy class of $\wt{G}_{\bq}$, $t \in \wt{T}_{\bq}/ \wt{C}_{\bq}$ and $z \in j^{-1}(\Oc \times \{ t \})$. Does $\ell^{\dim \Oc /2}$ divide
the dimension of any  irreducible representation of $\hopf_z$?
\end{question} 
\subsection{The torus orbits of symplectic leaves and the representations of the large quantum Borel algebras} 
\label{Poisson-geom-Borel} The algebras $U_{\bq}$, $U_{\bq}^{\geqslant}$,  $U_{\bq}^{\leqslant}$ and $U_{\bq}^\pm$ are 
$\Z^{\I}$-graded with grading $\deg e_i = - \deg f_i =\alpha_i$, $\deg K_i = \deg L_i =0$ for $i\in \I$. This  leads to a 
canonical action of the torus $(\Cset^\times)^\I$ on these algebras by algebra automorphisms, which preserves the 
central subalgebras $Z_{\bq}$, $Z_{\bq}^{\geqslant}$,  $Z_{\bq}^{\leqslant}$ and $Z_{\bq}^\pm$.  

By a direct  comparison, one obtains that the $(\Cset^\times)^\I$-action on $Z_{\bq}^{\geqslant}$ corresponds to the left action of 
$\tau_+^{-1}(T_{\bq})$ on $M_{\bq}^{\geqslant}$ in the sense that every automorphism from the first one corresponds to
an automorphism from the second and vice versa. Similarly, the $(\Cset^\times)^\I$-action on $Z_{\bq}^{\leqslant}$ corresponds to the left action of 
$\tau_-^{-1}(T_{\bq})$ on $M_{\bq}^{\leqslant}$. Theorem \ref{Poisson0}(a) implies that the induced action of $(\Cset^\times)^\I$ on 
$M_{\bq}$ preserves the symplectic leaves of $M_{\bq}$. So, in regard to irreps of $U_{\bq}$, the $(\Cset^\times)^\I$-automorphisms 
of $U_{\bq}$ do not provide any additional information to that in Theorem \ref{Poisson0}(a). 

However, for $U_{\bq}^{\geqslant}$ and $U_{\bq}^{\leqslant}$, we do obtain additional representation theoretic information from the $(\Cset^\times)^\I$-action, 
as stated in next theorem. Let  $W_{\bq}$ be the Weyl group of $G_{\bq}$, i.e., that of  $\wt{G}_{\bq}$.  
\begin{theorem}
\label{Poisson+}
For every choice of the specialization parameters 
$t_{ij} \in \mathbb Z$ satisfying the Non-degeneracy Assumption \ref{genericityassumption} the following hold:
\begin{enumerate}[leftmargin=*,label=\rm{(\alph*)}]
\item The Poisson structure on $M_{\bq}^{\geqslant}$ is invariant under the left and right actions of $\tau^{-1}_+(\wt{T}_{\bq})$. 
The $\tau^{-1}_+(\wt{T}_{\bq})$-orbits of symplectic leaves of $M_{\bq}^{\geqslant}$ are the double Bruhat cells
\begin{align*}
&\tau^{-1}_+( \wt{B}^+_{\bq} \cap \wt{B}^-_{\bq} w \wt{B}^-_{\bq}), 
&& w \in W_{\bq}.
\end{align*}
\item If $\tau_+(z)$ and $\tau_+(z')$ are in the same double Bruhat cell, then there is an algebra isomorphism
\[
\hopf_z^{\geqslant} \simeq \hopf_{z'}^{\geqslant}.
\]
\end{enumerate}
\end{theorem}
\begin{proof}
(a) For a Lie subalgebra of $\wt{\g}_{\bq} \oplus \wt{\h}_{\bq}$, denote by $N(-)$ its normalizer in $\wt{G}_{\bq} \times \wt{T}_{\bq}$. By \cite[Lemma 2.12]{LY}, 
the left and right actions of $\widetilde{M}_{\geqslant} \cap N( (\mm_{\bq}^{\geqslant})^{*})$ on the Lie group $\widetilde{M}_{\geqslant}$ preserve its Poisson structure.  
By the definition of $\tau_+$, these actions correspond to the left and right actions of $\tau_+^{-1}(T_{\bq})$ on $M_{\bq}^{\geqslant}$, so the latter 
preserve the Poisson structure on $M_{\bq}^{\geqslant}$, because $\widetilde{M}_{\bq}^{\geqslant} \twoheadrightarrow M_{\bq}^{\geqslant}$ is a Poisson map. 

Applying \cite[Theorem 2.7 and Proposition 2.15]{LY} and the Bruhat decomposition of $\wt{G}_{\bq}$,
we obtain that the $\widetilde{M}_{\geqslant} \cap N( (\mm_{\bq}^{\geqslant})^{*})$-orbits of symplectic leaves of 
$\widetilde{M}_{\geqslant}$ (with respect to either action) are the intersections
\[
\widetilde{M}_{\geqslant} \cap \big( ( \wt{G}_{\bq} \times \wt{T}_{\bq}) w ( \wt{G}_{\bq} \times \wt{T}_{\bq} ) \big)
\]
for $w \in W_{\bq}$. Since $\widetilde{M}_{\bq}^{\geqslant} \twoheadrightarrow M_{\bq}^{\geqslant}$ is a Poisson covering map
and $\tau_+ : M_{\bq}^\geqslant \stackrel{\simeq}{\longrightarrow} \wt{B}_{\bq}^+$ is an isomorphism (Theorem \ref{thm:Ms-to-Bs}), 
the $\tau^{-1}_+(\wt{T}_{\bq})$-orbits of symplectic leaves of $M_{\bq}^{\geqslant}$ (with respect to either action) are the double Bruhat cells
$\tau^{-1}_+( \wt{B}^+_{\bq} \cap \wt{B}^-_{\bq} w \wt{B}^-_{\bq})$ for $w \in W_{\bq}$.

\medbreak
Part (b) follows from part (a),  Theorems  \ref{theorem:BG-reps} and \ref{thm:Poisson-orders},  and the fact that the left action of $\tau_+^{-1}(\wt{T}_{\bq})$ on 
$M_{\bq}^{\geqslant}$ comes from the $(\Cset^\times)^\I$-action on $U_{\bq}^{\geqslant}$ by algebra automorphisms.
\end{proof}

\begin{example}\label{exa:reps}
Let $\bq$ be of type $\Bgl(4)$. By Example \ref{exa:bgl4-Ms-to-Bs}, the corresponding algebraic group $\wt{G}_{\bq}$ is isomorphic to 
$\operatorname{PSL}_3(\Cset)\times \operatorname{PSL}_3(\Cset)$ whose Weyl group is $S_3 \times S_3$. 
Theorem \ref{Poisson+} implies that among the quotients $U_{\bq}^\geqslant/ \Mg^{\leqslant}_z U_{\bq}^\geqslant$ 
for $z$ in the maximal spectrum of $Z_{\bq}^{\geqslant}$, there are at most $|S_3 \times S_3| = (3!)^2 = 36$
isomorphism classes of finite dimensional algebras.
\end{example}

Analogously to Theorem \ref{Poisson+} one proves the following:
\begin{proposition}
\label{Poisson-}
For every choice of the specialization parameters 
$t_{ij} \in \mathbb Z$ satisfying the Non-degeneracy Assumption \ref{genericityassumption} the following hold:
\begin{enumerate}[leftmargin=*,label=\rm{(\alph*)}]
\item The Poisson structure on $M_{\bq}^{\leqslant}$ is invariant under the left and right actions of $\tau^{-1}_-(\wt{T}_{\bq})$. 
The $\tau^{-1}_-(\wt{T}_{\bq})$-orbits of symplectic leaves of $M_{\bq}^{\leqslant}$ are the double Bruhat cells
\begin{align*}
&\tau^{-1}_-( \wt{B}^-_{\bq} \cap \wt{B}^+_{\bq} w \wt{B}^+_{\bq}), 
&& w \in W_{\bq}.
\end{align*}
\item If $\tau_-(z)$ and $\tau_-(z')$ are in the same double Bruhat cell, then $\hopf_z^{\leqslant} \simeq \hopf_{z'}^{\leqslant}$ as algebras.
\end{enumerate}
\end{proposition}
\subsection{Poisson homogeneous spaces and irreps of large quantum unipotent algebras} 
\label{Poisson-geom-unip} 
Since
$Z_{\bq}^+$ is the algebra of coinvariants for the coaction of $Z_{\bq}^{0+}$ on $Z_{\bq}^{\geqslant}$ obtained by restricting the coaction 
of  $U_{\bq}^{0+}$ on $Z_{\bq}^{\geqslant}$, and analogously for the negative part, we have isomorphisms of Poisson algebras
\begin{align}
\label{eq:P-hom-iso}
&Z_{\bq}^+ \simeq \Cset[ M_{\bq}^{\geqslant}/ \tau_+^{-1}(\wt{T}_{\bq})],  
&&Z_{\bq}^- \simeq \Cset[ M_{\bq}^{\leqslant}/ \tau_-^{-1}(\wt{T}_{\bq})]. 
\end{align}
As shown in the previous subsection, the left and right actions of $\tau_+^{-1}(\wt{T}_{\bq})$ and $\tau_-^{-1}(\wt{T}_{\bq})$ 
on the Poisson algebraic groups $M_{\bq}^{\geqslant}$ and $M_{\bq}^{\leqslant}$ preserve their Poisson structures.
The right hand sides 
of the isomorphisms \eqref{eq:P-hom-iso} involve the coordinate rings of the resulting 
Poisson homogeneous spaces $M_{\bq}^{\geqslant}/ \tau_+^{-1}(\wt{T}_{\bq})$ and $M_{\bq}^{\leqslant}/ \tau_-^{-1}(\wt{T}_{\bq})$
obtained by taking quotients with respect to the right actions. 
The Poisson structures on $M_{\bq}^{\geqslant}/ \tau_+^{-1}(\wt{T}_{\bq})$ and $M_{\bq}^{\leqslant}/ \tau_-^{-1}(\wt{T}_{\bq})$
are invariant under the induced left actions of $\tau_+^{-1}(\wt{T}_{\bq})$ and $\tau_-^{-1}(\wt{T}_{\bq})$.
By Theorem \ref{thm:Ms-to-Bs}, $\tau_+$ restricts to the isomorphism of homogeneous spaces 
$\tau_+ : M_{\bq}^{\geqslant} / j_+^{-1}(\wt{T}_{\bq}) \stackrel{\simeq}{\longrightarrow} \wt{B}_{\bq}^+ /\wt{T}_{\bq}$. Denote the canonical isomorphism
\[
\upsilon : \wt{B}_{\bq}^+ /\wt{T}_{\bq} \stackrel{\simeq}{\longrightarrow} \wt{B}_{\bq}^+ \wt{B}_{\bq}^-/\wt{B}_{\bq}^{-} \subset \wt{G}_{\bq}/\wt{B}_{\bq}^-.
\]
\begin{theorem}
\label{Poisson-unip}
 For every choice of the specialization parameters 
$t_{ij} \in \mathbb Z$ satisfying the Non-degeneracy Assumption \ref{genericityassumption} the following hold:
\begin{enumerate}[leftmargin=*,label=\rm{(\alph*)}]
\item  
The $\tau^{-1}_+(\wt{T}_{\bq})$-orbits of symplectic leaves of $M_{\bq}^{\geqslant}/  \tau_+^{-1}(\wt{T}_{\bq})$ are the open Richardson varieties
\begin{align*}
&\tau^{-1}_+ \upsilon^{-1} \big( (\wt{B}_{\bq}^+ \wt{B}_{\bq}^- \cap \wt{B}^-_{\bq} w \wt{B}^-_{\bq} ) / \wt{B}_{\bq}^- \big), 
&& w \in W_{\bq}.
\end{align*}
\item If $\upsilon \tau_+(z)$ and $\upsilon \tau_+ (z')$ are in the same open Richardson variety, then there is an isomorphism of algebras
\[
\hopf_z^{+} \simeq \hopf_{z'}^{+}.
\]
\end{enumerate}
\end{theorem}
\pf
Part (a) is proved arguing as in the proof of Theorem \ref{Poisson+}(a). 
Then (b) is a consequence of (a), and
Theorems \ref{theorem:BG-reps} and \ref{thm:Poisson-orders}.
\epf
An analogous result holds for the large quantum unipotent algebra $U_{\bq}^-$
and the torus orbits of symplectic leaves of the Poisson homogeneous space $M_{\bq}^{\leqslant}/ \tau_-^{-1}(\wt{T}_{\bq})$.

\appendix
\section{Families of finite-dimensional Nichols algebras}\label{sec:Nichols-specialization}
Let $\theta \in \N$, $\I = \I_{\theta}$. We fix a matrix $\bq = (q_{ij}) \in \Cset^{\I \times \I}$
such that $\dim \toba_{\bq} < \infty$. To insure centrality of $Z_{\bq}$ we require

\begin{property}\label{property:condition-cartan-roots-simple}
The matrix $\bq$ satisfies \eqref{eq:condition-cartan-roots-simple}, i.e.,
$q_{\alpha_i\beta}^{N_\beta}=1$,  for all $i\in\I$, $\wbeta\in\varPi^{\bq}$.
\end{property}

\begin{remark}\label{rem:condition-cartan-roots-simple-app}
If the Dynkin diagram of $\bq'$ is as in Tables \ref{tab:Appendix-cartan}, \ref{tab:Appendix-super} and \ref{tab:Appendix-modular}, then there is $\bq$ with the same Dynkin diagram that satisfies \eqref{eq:condition-cartan-roots-simple}; the proof is straightforward.
\end{remark}

If $\bq$ satisfies \eqref{eq:condition-cartan-roots-simple}, then
any matrix in its Weyl-equivalence class also does.
Let $\Cset[\nu^{\pm 1}]$ be the algebra of Laurent polynomials; its group of units is 
$\Cset[\nu^{\pm 1}]^\times = \Cset^{\times} \nu^{\Z}$. 
Let
\begin{align}\label{eq:braiding-matrix-parameter.app}
\bqf = ( \qf_{ij}) \in  \big( \Cset[\nu^{\pm 1}]^\times \big)^{\I \times \I}
\end{align}
For $x \in \Cset^{\times}$, we denote by $\bqf(x)$
the matrix obtained by the evaluation $\ev: \Cset[\nu^{\pm 1}] \to \Cset$, $\ev(\nu) = x$.
We seek for matrices \eqref{eq:braiding-matrix-parameter.app} with the following properties \ref{item:property-a}
and \ref{item:property-specialization-integer-3}. 

\begin{property}\label{item:property-a}
The Nichols algebra of the $\Cset(\nu)$-braided vector space of diagonal type with braiding matrix \eqref{eq:braiding-matrix-parameter.app}  has the same  arithmetic root system as $\bq$.
\end{property}

By inspection of the list in \cite{H-classif RS}--see also the exposition in \cite{AA-diag-survey}--we conclude that the only possible matrices
\eqref{eq:braiding-matrix-parameter.app} are those  Weyl-equivalent to the ones with Dynkin diagrams as
in Tables \ref{tab:Appendix-cartan}, \ref{tab:Appendix-super}
and \ref{tab:Appendix-modular} and that the following property holds.

\begin{property}\label{item:property-c} There exists an open subset $\emptyset \neq O \subseteq \ku^{\times}$ 
such that for any $x \in O$, the root systems and Weyl groupoids 
associated to $\bqf$ and $\bqf(x)$ are isomorphic.
Also there exists $\xi \in \G_{\infty}' \cap O$ 
with $N := \ord \xi \in [2, \infty)$ such that $\bq = \bqf(\xi)$. 
\end{property}

\begin{remark}\label{rem:twisting-generic}
(i). The Dynkin diagrams of the matrices $\bqf$ and $\bq$
locally have the form 
$\xymatrix{\overset{\qf_{ii}}{\underset{\ }{\circ}}\ar  @{-}[r]^{\widetilde{\bqf}_{ij}}  &
\overset{\bqf_{jj}}{\underset{\ }{\circ}}}$, 
respectively $\xymatrix{\overset{q_{ii}}{\underset{\ }{\circ}}\ar  @{-}[r]^{\widetilde{q}_{ij}}  &
\overset{q_{jj}}{\underset{\ }{\circ}}}$, 
where $\widetilde{\bqf}_{ij} = \qf_{ij} \qf_{ji}$, $\widetilde{q}_{ij} = q_{ij} q_{ji}$,
i.e., the Dynkin diagram does not determine completely the braiding matrix. We deal with this as follows.
Let  $\bp = (p_{ij}) \in \Cset^{\I \times \I}$ with the same Dynkin diagram as $\bq$. Then there exists 
$\bpf \in  \big( \Cset[\nu^{\pm 1}]^\times \big)^{\I \times \I}$  with the same Dynkin diagram as $\bqf$ such that
$\bp = \bpf(\xi)$. For, take $\bpf_{ii} = \bqf_{ii}$ and $\bpf_{ij} \in \Cset[\nu^{\pm 1}]^\times$ such that 
$p_{ij} = \bpf_{ij}(\xi)$ for $i < j$;
then  $\bpf_{ji} =\widetilde{\bqf}_{ij} \bpf_{ij}^{-1}$.

\medbreak
(ii). Assume that $\bqf$ satisfies \ref{item:property-a}.
Let  $\bpf$ be another matrix with the same diagram as \eqref{eq:braiding-matrix-parameter.app}.
Then $\bqf_{ij} = \bpf_{ij}\nu^{h_{ij} N}$, $i < j$ for a unique family  $(h_{ij})_{i<j\in\I}$ with $h_{ij}\in\mathbb Z$. 
\end{remark}

\begin{property}\label{item:property-specialization-integer-3} 
 $\WP^{\bqf}$ defined in \eqref{eq:matrix-P} is invertible.
\end{property}

Let $\NN$ be the diagonal matrix with entries $N_{\beta}$, $\wbeta\in\varPi^\bq$.
The matrix $\WP^{\bqf}$ is invertible if and only if the auxiliary matrix $\TT^{\bqf}$ is so,
where
\begin{align*}
\WP^{\bqf} = -\xi^{-1}\NN \TT^{\bqf} \NN.
\end{align*}

\begin{proposition}\label{prop:specialization-integer}
There exist matrices $\Matr = (\matr_{ij}) \in \Z^{\I \times \I}$
and $(\pt_{ij}) \in  (\Cset^{\times})^{\I \times \I}$ such that $C$ is symmetric and:

\begin{enumerate}[leftmargin=*,label=\rm{(\roman*)}]
\item\label{item;prop-AppA-1} There are infinitely many matrices $T = (t_{ij}) \in \Z^{\I \times \I}$ fulfilling
\begin{align}\label{eq:def-T-conditions}
t_{ii} &= \matr_{ii}, & t_{ij} + t_{ji} &= \matr_{ij} & &\text{ for all } i \neq j \in \I
\end{align}
such that the matrix $\bqf = (\bqf_{ij})$ defined by
\begin{align}\label{eq:def-qbf-T}
\bqf_{ij} &= \pt_{ij}\nu^{t_{ij}}, &\text{ for all } i, j \in \I
\end{align}
satisfies \ref{item:property-a}.

\item\label{item;prop-AppA-2} Among those $T$ in \ref{item;prop-AppA-1}, there infinitely many such that 
 $\bqf$ satisfies \ref{item:property-specialization-integer-3}.
\end{enumerate}
\end{proposition}

\pf It suffices to fix one matrix for each Weyl-equivalence class, see Lemma \ref{lem:matrix-P-Weyl-gpd}. 
We check below \ref{item;prop-AppA-1} by case-by-case considerations computing also $\TT^{\bqf}$
and proving that it is invertible for infinitely many $T$.

\subsection{Cartan type} Let $\bq$ be in this class; then there is a Cartan matrix $A = (a_{ij})_{i,j \in \I}$
such that $q_{ij} q_{ji} = q_{ii}^{a_{ij}}$. We fix $d_i \in \I_3$ such that
$d_ia_{ij} = d_j a_{ji}$ for all $i,j \in \I$. 
The Lie algebra $\g_{\bq}$ has the same type except when $N$ is even and $A$ is 
of type $B_{\theta}$ or $C_{\theta}$, when they are interchanged.
In this case $\varPi^{\bq} = \{N_i\alpha_i: i\in \I\}$, so \eqref{eq:condition-cartan-roots-simple} becomes:
\begin{align}\label{eq:condition-cartan-roots-simple-app1}
q_{ij}^{N_j} &= 1,& &\text{ for all } i,j\in\I.
\end{align} 
The matrix $\bqf$ we are looking for should also satisfy
$\qf_{ij}\qf_{ji} = \qf_{ii}^{a_{ij}}$ for all $i\neq j$.
In all cases we take $\xi = q_{11}$ except for $B_{\theta}$, where $\xi = q_{\theta\theta}$;
see Table \ref{tab:Appendix-cartan}.  Set $t_{ii}= d_{i}$ and $\bqf_{ii}= \nu^{t_{ii}}$. 
Thus $\bqf_{ii}(\xi) = q_{ii}$ for all $i\in \I$. 
Recall that
\begin{align*}
(\nu - \xi) \wp^{\bqf}_{\alpha_i\alpha_j}(\nu) &= 1-\bqf_{ij}^{N_iN_j}.
\end{align*}
For instance $\wp^{\bqf}_{\alpha_i\alpha_i}(\nu) =\frac{ 1- \nu^{d_iN_i^2}}{\nu - \xi}$ hence
\begin{align}\label{eq:pab-Cartan-ii}
\wp^{\bqf}_{\alpha_i\alpha_i}(\xi)&= - \xi^{-1}d_{i}N_i^2 = - \xi^{-1}t_{ii}N_i^2.
\end{align}

Let  $i<j$. 
We see that there exists $d_j \in \I_3$ such that $N_j = N/d_j$. By \eqref{eq:condition-cartan-roots-simple-app1},
$q_{ij}$ is a power of $\xi^{d_j}$; choose $t_{ij}\in d_j\Z$ such that $\bqf_{ij} = \nu^{t_{ij}}$ 
satisfies $\bqf_{ij}(\xi) = \xi^{t_{ij}} = q_{ij}$.
Set $t_{ji}= d_{i}a_{ij}-t_{ij}$ and $\bqf_{ji}= \nu^{t_{ji}}$. 
We have defined $T$ satisfying \eqref{eq:def-T-conditions} and 
$\bqf$ turns out to be given by \eqref{eq:def-qbf-T} with $\pt_{ij} = 1$ for all $i,j$, 
i.e.,  \ref{item;prop-AppA-1} holds. Also for all $i\neq j$, 
$(\nu - \xi) \wp^{\bqf}_{\alpha_i\alpha_j}(\nu) = 1- \nu^{t_{ij}N_iN_j}$ and
\begin{align} \label{eq:pab-Cartan-ij}
\wp^{\bqf}_{\alpha_i\alpha_j}(\xi)&=- \xi^{-1}t_{ij}N_iN_j.
\end{align}
Therefore $\TT^{\bqf} = T$. Observe that if $t_{ij} =0$ for $i<j$, then 
$\det \TT^{\bqf} \neq 0$. By a standard argument, \ref{item;prop-AppA-2} holds.

\begin{table}[ht]
\caption{Cartan type}\label{tab:Appendix-cartan}
\begin{center}
\begin{tabular}{c  c c }\hline
Type & $\bqf$ & $N$
\\ \hline
$A_{\theta}$ &  $\xymatrix{ \overset{\nu}{\underset{\ }{\circ}}\ar  @{-}[r]^{\nu^{-1}}  &
\overset{\nu}{\underset{\ }{\circ}}\ar  @{-}[r]^{\nu^{-1}} &  \overset{\nu}{\underset{\
}{\circ}}\ar@{.}[r] & \overset{\nu}{\underset{\ }{\circ}} \ar  @{-}[r]^{\nu^{-1}}  &
\overset{\nu}{\underset{\ }{\circ}}}$ & 
\\ \hline
$B_{\theta}$,  $\theta \ge 2$
& $\xymatrix{ 
\overset{\,\,\nu^2}{\underset{\ }{\circ}}\ar  @{-}[r]^{\nu^{-2}} &
\overset{\,\,\nu^2}{\underset{\ }{\circ}}\ar@{.}[r] & \overset{\,\,\nu^2}{\underset{\
}{\circ}} \ar  @{-}[r]^{\nu^{-2}}  & \overset{\nu}{\underset{\ }{\circ}}}$ & {\small $>2$} 
\\ \hline
$C_{\theta}$,  $\theta \ge 3$
& $\xymatrix{ \overset{\nu}{\underset{\ }{\circ}}\ar  @{-}[r]^{\nu^{-1}}  &
\overset{\nu}{\underset{\ }{\circ}}\ar  @{-}[r]^{\nu^{-1}} &  \overset{\nu}{\underset{\
}{\circ}}\ar@{.}[r] & \overset{\nu}{\underset{\ }{\circ}} \ar  @{-}[r]^{\nu^{-2}}  &
\overset{\,\,\nu^2}{\underset{\ }{\circ}}}$ & {\small  $>2$} 
\\ \hline
$D_{\theta}$,  $\theta \ge 4$
& $\xymatrix@R-15pt{  & & &  \overset{\nu}{\circ} &\\
\overset{\nu}{\underset{\
}{\circ}}\ar  @{-}[r]^{\nu^{-1}} &  \overset{\nu}{\underset{\ }{\circ}}\ar@{.}[r] &
\overset{\nu}{\underset{\ }{\circ}} \ar  @{-}[r]^{\nu^{-1}}  & \overset{\nu}{\underset{\
}{\circ}} \ar @<0.7ex> @{-}[u]_{\nu^{-1}}^{\qquad} \ar  @{-}[r]^{\nu^{-1}} &
\overset{\nu}{\underset{\ }{\circ}}}$ & 
\\ \hline
$E_{\theta}$,  $\theta \in \I_{6,8}$
& $\xymatrix@R-15pt{ &  &   \overset{\nu}{\circ} &\\
\overset{\nu}{\underset{\ }{\circ}}\ar  @{-}[r]^{\nu^{-1}}  &  \overset{\nu}{\underset{\
}{\circ}}\ar@{.}[r]  & \overset{\nu}{\underset{\ }{\circ}} \ar @<0.7ex> @{-}[u]_{\nu^{-1}} \ar
@{-}[r]^{\nu^{-1}} &  \overset{\nu}{\underset{\ }{\circ}}  \ar  @{-}[r]^{\nu^{-1}} &
\overset{\nu}{\underset{\ }{\circ}}}$ &  
\\ \hline
$F_{4}$ 
& $\xymatrix{ \overset{\,\,\nu}{\underset{\ }{\circ}}\ar  @{-}[r]^{\nu^{-1}}  &
\overset{\,\,\nu}{\underset{\ }{\circ}}\ar  @{-}[r]^{\nu^{-2}} &   \overset{\nu^2}{\underset{\
}{\circ}} \ar  @{-}[r]^{\nu^{-2}}  &  \overset{\nu^2}{\underset{\ }{\circ}} }$  
& {\small $>2$}
\\ \hline
$G_2$  & $\xymatrix{  \overset{\,\,\nu}{\underset{\ }{\circ}} \ar  @{-}[r]^{\nu^{-3}} &
\overset{\nu^3}{\underset{\ }{\circ}}}$ & {\small $>3$} 
\\
\hline \end{tabular}
\end{center}
\end{table}

\subsection{Super type} 

Assume that the braiding matrix  $\bq$ is of super type; see \cite{AA-diag-survey} for details and below for $\superda{\alpha}$. 
Going over the list, we see that there exist

\begin{itemize}[leftmargin=*]\renewcommand{\labelitemi}{$\circ$}
\item $\xi \in \Cset^{\times}$, a root of 1 of order $N > 1$;

\smallbreak
\item a symmetric matrix $\Mat = (\mat_{ij})_{i,j \in \I} \in \Z^{\I\times \I}$ with
$\mat_{ij} = 1$ for at least one pair  $(i,j)$;

\smallbreak
\item a parity vector $\pt = (\pt_1, \dots, \pt_{\theta}) \in \{\pm 1 \}^{\I}$ 
with $\pt_i =-1$ when $\mat_{ii} =0$; such that
\end{itemize}
\begin{align*}
q_{ij}q_{ji}  &=  \xi^{\mat_{ij}}, & i&\neq j;& q_{ii} &=  \pt_i\xi^{\mat_{ii}},& i &\in \I. 
\end{align*} 

We describe in Table \ref{tab:Appendix-super} matrices $\bqf$ of super type, one for each
Weyl-equivalence class (here $\alpha_{( ij)} := \alpha_{i} + \dots +\alpha_{j}$ for $i < j$). 
Since the matrix $\bq$ has an analogous shape, we may assume that

\begin{itemize}[leftmargin=*]\renewcommand{\labelitemi}{$\circ$}
\item there exists $k \in \I$ such that $\{i\in \I: \pt_i =-1 \} = \{k\}$;

\smallbreak
\item there exists $h \in \I$, $h \neq k$, such that $\xi = q_{hh}$.
\end{itemize}

Recall that $\widetilde{\varPi}$ was defined in \eqref{eq:def-wtPi-1} and  \eqref{eq:def-wtPi-2}. Therefore we have:

\begin{itemize}[leftmargin=*]\renewcommand{\labelitemi}{$\circ$}
\item either  $\widetilde{\varPi}^{\bq} = \{N_i \alpha_{i}: i \in \I, i \neq k\}\cup \{\Nt\alpha_k\}$ ($\Nt=N$ if $N$ is even and $\Nt=2N$ if $N$ is odd) for type $\supera{k-1}{\theta - k}$
or else there exists a unique positive non-simple root $\beta$ such that 
$\widetilde{\varPi}^{\bq} = \{N_i \alpha_{i}: i \in \I, i \neq k\} \cup \{N_{\beta} \beta\}$;

\smallbreak
\item for $i \in \I$, $i \neq k$, we may (and do) choose $b_{ii} \in \{\pm1, \pm2, \pm3 \}$.
Then $N_i = LCD(b_{ii}, N)$; set $d_i = N / N_i$.
\end{itemize}

We start defining the  matrix $\bqf$.
First we take $t_{ii} = b_{ii}$ and $\bqf_{ii} = \pt_{i}q^{b_{ii}}$ for all $i\in \I$. 

\smallbreak
Condition \eqref{eq:condition-cartan-roots-simple} says that $q_{ij}^{N_j}=1$, for all $i\in\I$, 
$j \in \I \backslash \{k\}$.
Let $i < j$ with $j \neq k$; choose $t_{ij}\in d_j\Z$ and set 
$t_{ji} = \mat_{ji} -t_{ij}$.
Then $\bqf_{ij} = \nu^{t_{ij}}$ and $\bqf_{ji} = \nu^{t_{ji}}$
satisfy $\bqf_{ij}(\xi) = \xi^{t_{ij}} = q_{ij}$ and  $\bqf_{ij}\bqf_{ji} = \nu^{\mat_{ij}}$.

\smallbreak
Similarly, $q_{ki}^{N_i}=1$ for $k > i$, so choose $t_{ik}\in d_i\Z$ and set $t_{ki} = \mat_{ki} -t_{ik}$,
$\bqf_{ki} = \nu^{\mat_{ki} - t_{ik}}$ and  $\bqf_{ik} = \nu^{t_{ik}}$ so that $\bqf_{ki}(\xi) = \xi^{t_{ki}} =q_{ki}$. 
We have defined $T$ satisfying \eqref{eq:def-T-conditions} and 
$\bqf$ turns out to be given by \eqref{eq:def-qbf-T} with $\pt_{ii} = \pt_{i}$ and $\pt_{ij} = 1$ for all $i\neq j$, 
i.e.,  \ref{item;prop-AppA-1} holds.

\smallbreak
It remains to compute the matrix $\TT^{\bqf}$. Arguing as in the Cartan case we see that
\begin{align*}
\wp^{\bqf}_{\alpha_i\alpha_i}(\xi)&= - \xi^{-1} \pt_{i}\mat_{ii} N_i^2,&
\wp^{\bqf}_{\alpha_i\alpha_j}(\xi)&= - \xi^{-1}t_{ij}N_iN_j,& i, j &\in \I \backslash \{k\}.
\end{align*}

Assume that there exists $\gamma  \in \varPi^{\bq} \backslash \I$ (a non-simple Cartan root).
Then there exist $\pt_{\gamma} \in \{\pm 1 \} $ and $\mat_{\gamma\gamma}, \mat_{i\gamma}\in \Z$ such that
\begin{align*}
\bqf_{\gamma\gamma} &= \pt_{\gamma}\nu^{\mat_{\gamma\gamma}}, & \bqf_{i\gamma}\bqf_{\gamma i} &= \nu^{\mat_{i\gamma}}.
\end{align*}
Extend $(t_{ij})$ to a bilinear form $t: \Z^{\I} \times \Z^{\I} \to \Z$. 
Then for $k \neq i \in \I$,
\begin{align*}
 \wp^{\bqf}_{\gamma\gamma}(\xi) &= - \xi^{-1} \pt_{\gamma}\mat_{\gamma\gamma} N_{\gamma}^2,& 
 \wp^{\bqf}_{\alpha_i\gamma}(\xi) &= - \xi^{-1} t_{i\gamma}N_iN_{\gamma}, &
(\nu - \xi) \wp^{\bqf}_{\gamma\alpha_i}(\nu) &= - \xi^{-1} t_{\gamma i}N_iN_{\gamma}.
\end{align*}

All in all,  
$\TT^{\bqf}$ is of the form $\left(t_{\alpha \beta}\right)\in\Z^{\varPi^{\bq}\times \varPi^{\bq}}$, 
where $t_{\alpha\alpha} = \pt_{\alpha}\mat_{\alpha\alpha}$, 
$t_{r\alpha \beta} + t_{\beta \alpha} = \mat_{\alpha \beta}$ for $\alpha \neq \beta$.
Arguing as in  the Cartan case, we conclude that \ref{item;prop-AppA-2} holds.

\begin{table}[ht]
\caption{Super type}\label{tab:Appendix-super}
\begin{center}
\begin{tabular}{c  c c c  c}\hline
Type & $\bqf$ & $N$ & $\widetilde{\varPi}^{\bq}$ & $\g_{\bq}$
\\ \hline
${\small\xymatrix@R-25pt{\supera{k-1}{\theta - k}, \\  k \in \I_{\lfloor\frac{\theta+1}{2} \rfloor}}}$ 
&$\xymatrix@C-10pt{ \overset{\nu^{-1}}{\underset{\ }{\circ}}\ar  @{-}[r]^{\nu}  &
\overset{\nu^{-1}}{\underset{\ }{\circ}}\ar@{.}[r] &  \underset{k}{\overset{{-1}}{\circ}}
\ar@{.}[r] & \overset{\nu}{\underset{\ }{\circ}} \ar  @{-}[r]^{\nu^{-1}}  &
\overset{\nu}{\underset{\ }{\circ}}}$ 
& $> 2$ & ${\small\xymatrix@R-25pt{\{N\alpha_j \, | j\ne k \} \\  \cup\{\Nt\eta\} }}$
& ${\small A_{k-1} \times A_{\theta-k}}$
\\ \hline
${\small \xymatrix@R-25pt{\superb{k}{\theta - k}, \\  k\in\I_{\theta-1} }}$ 
& $\xymatrix@C-10pt{ \overset{\nu^{-2}}{\underset{\ }{\circ}}\ar  @{-}[r]^{\nu^2}  &
\overset{\nu^{-2}}{\underset{\ }{\circ}}\ar@{.}[r] &  \underset{k}{\overset{{-1}}{\circ}}
\ar@{.}[r] & \overset{\nu^2}{\underset{\ }{\circ}} \ar  @{-}[r]^{\nu^{-2}}  &
\overset{\nu}{\underset{\ }{\circ}}}$ & {\small  $\neq2,4$} & {\small\eqref{eq:Piq-B-odd}} &  
${\small \xymatrix@R-25pt{C_{k} \times B_{\theta-k} \\  C_{k} \times C_{\theta-k} }}$
\\ \hline
${\small\superd{k}{\theta - k}}$, 
& $\xymatrix@C-10pt{ \overset{\nu^{-1}}{\underset{\ }{\circ}}\ar  @{-}[r]^{\nu}  &
\overset{\nu^{-1}}{\underset{\ }{\circ}}\ar@{.}[r] &  \underset{k}{\overset{{-1}}{\circ}}
\ar@{.}[r] & \overset{\nu}{\underset{\ }{\circ}} \ar  @{-}[r]^{\nu^{-2}}  &
\overset{\nu^2}{\underset{\ }{\circ}}}$ & {\small $>2$} & {\small \eqref{eq:Piq-D-odd} } & $D_{k} \times C_{\theta-k}$
\\
$k < \frac{\theta}{2}$ & &  &  & $D_{k} \times B_{\theta-k}$
\\ \hline
$\xymatrix@R-25pt{\superda{\alpha}, \\ d_1,d_3 \in \N}$
& $\xymatrix@C-2pt{ \overset{\nu^{d_1}}{\circ}\ar  @{-}[r]^{\nu^{-d_1}}  &
\overset{-1}{\circ} \ar  @{-}[r]^{\nu^{-d_3}}  & \overset{\nu^{d_3}}{\circ}}$ & & {\small  \eqref{eq:Piq-D21}}
& {\small $A_1\times A_1\times A_1$}
\\ \hline
$\superf$ 
& $\xymatrix@C-10pt{ \overset{\,\, \nu^2}{\circ}\ar  @{-}[r]^{\nu^{-2}}  & \overset{\,\,
\nu^2}{\circ} \ar  @{-}[r]^{\nu^{-2}}  & \overset{\nu}{\circ} &
\overset{-1}{\circ} \ar  @{-}[l]_{\nu^{-1}}}$  
& {\small  $>2$} & {\small  \eqref{eq:Piq-F-odd}} & {\small $A_1\times B_3$}
\\ \hline
$\superg$  & $\xymatrix@C-5pt{ \overset{-1}{\circ}\ar  @{-}[r]^{\nu^{-1}}  &
\overset{\nu}{\circ} \ar  @{-}[r]^{\nu^{-3}}  & \overset{\,\, \nu^3}{\circ},}$ & {\small $N>3$} 
& {\small  \eqref{eq:Piq-G}} & {\small $A_1\times G_2$}
\\\hline \end{tabular}
\end{center}
\begin{align}\label{eq:Piq-B-odd}
&\{N_j\alpha_j \, | j\ne k\}\cup \{ N_{\alpha_{(k\theta)}} \alpha_{(k\theta)}\},
\\\label{eq:Piq-D-odd}
&\{N_j\alpha_j \, | j\ne k \}\cup \{N_{\alpha_{(k-1 \, \theta)}+\alpha_{(k \,\theta-1)}}(\alpha_{(k-1 \, \theta)}+\alpha_{(k \,\theta-1)})\},
\\\label{eq:Piq-D21}
 &\{N_1\alpha_1, N_3\alpha_3, N_{\alpha_1+2\alpha_2+\alpha_3}(\alpha_1+2\alpha_2+\alpha_3)\},
\\
\label{eq:Piq-F-odd}
&\{N_1\alpha_1, N_2\alpha_2, N_3\alpha_3, N_{\alpha_1+ 2\alpha_2+ 3\alpha_3+ 2\alpha_4}(\alpha_1+ 2\alpha_2+ 3\alpha_3+ 2\alpha_4)\},
\\ \label{eq:Piq-G}
&\{N_{\alpha_1+2\alpha_2+\alpha_3} (\alpha_1+2\alpha_2+\alpha_3), N_2\alpha_2, N_3\alpha_3\}.
\end{align}
\end{table}

\subsubsection*{Type $\superda{\alpha}$} 
The diagrams of this type are Weyl equivalent to the following one $\xymatrix@C-2pt{ \overset{r}{\circ}\ar  @{-}[r]^{r^{-1}}  & \overset{-1}{\circ} \ar  @{-}[r]^{s^{-1}}  & \overset{s}{\circ}}$, with $r,s, rs\ne 1$.
The corresponding Nichols algebra has finite dimension if and only if $r, s\in\G'_{\infty}$, $rs\ne 1$.
Let $\bq$ be a braiding matrix with this diagram satisfying \eqref{eq:condition-cartan-roots-simple}.
Fix a generator $\xi$ of the subgroup of $\G_{\infty}$ generated by $r, s$; we choose $d_1, d_3\in\N$ minimal such that $r=\xi^{d_1}$, $s=\xi^{d_3}$.
Then there exists a braiding matrix $\bqf$ as in Table \ref{tab:Appendix-super} such that $\bq=\bqf(\xi)$.

\subsection{Modular type}\label{sec:by-diagram-modular-char3}
The Nichols algebras in this family could be thought of as quantizations 
in $\charr 0$ of the 34-dimensional Lie algebras in $\charr 2$ from \cite{KW-exponentials}, respectively
the 10-dimensional  Lie algebras in $\charr 3$ introduced in \cite{BGL}.
The information on  this type is given in Table \ref{tab:Appendix-modular}.
The matrices $T$ and $\TT^{\bqf}$ are worked out as in the super case.
\epf

\begin{table}[ht]
\caption{Modular type}\label{tab:Appendix-modular}
\begin{center}
\begin{tabular}{c  c c c  c}\hline
Type & $\bqf$ & $N$ & $\widetilde{\varPi}^{\bq}$ & $\g_{\bq}$
\\ \hline
$\Bgl(4)$ 
&$\xymatrix@C-15pt{\overset{\nu}{\underset{\ }{\circ}}\ar  @{-}[r]^{\nu ^{-1}}  &
\overset{\nu}{\underset{\ }{\circ}}
\ar  @{-}[r]^{\nu^{-1}}  & \overset{-1}{\underset{\ }{\circ}}
\ar  @{-}[r]^{-\nu}  & \overset{-\nu^{-1}}{\underset{\ }{\circ}}}$ 
& $> 2$ & {\small \eqref{eq:Piq-Bgl-4}} & ${\small A_{2} \times A_{2}}$
\\ \hline
${\small \Brown(2)}$
& $\xymatrix@C-15pt{\overset{\zeta}{\underset{\ }{\circ}} \ar  @{-}[r]^{\nu^{-1}}  &
\overset{\nu}{\underset{\ }{\circ}}}$,  $\zeta \in \G'_3$
& {\small $\neq 3$} & {\small $\{2M\alpha_1+M\alpha_2,N\alpha_2\}$} & {\small $A_1\times A_1$}
\\ \hline
\end{tabular}
\end{center}
\begin{align}\label{eq:Piq-Bgl-4}
&\{N\alpha_1,N\alpha_2,M\alpha_4,M\alpha_1+2M\alpha_2+3M\alpha_3+M\alpha_4\}.
\end{align}
\end{table}

\section{Lie bialgebras and Poisson algebraic groups}\label{subsec:B.1}
We gather minimal background material on Lie bialgebras and Poisson algebraic groups  
for Sections  \ref{sec:PLgrps} and \ref{sec:Poisson-geom}. We refer to \cite[Section 2-7]{ES} for a full treatment.
\subsection{Lie bialgebras} 
Recall that a Lie bialgebra is a Lie algebra $\g$ equipped with a linear map $\delta : \g \to \wedge^2 \g$ such that
\begin{enumerate}[leftmargin=*,label=\rm{(\roman*)}]
\item the dual of the map $\delta$ defines a Lie algebra structure of $\g^*$ and
\item $\delta$ is a 1-cocycle, i.e., $\delta([a,b]) = \ad_a(\delta(b)) - \ad_b(\delta(a))$ for all $a, b \in \g$. 
\end{enumerate} 

The Lie bialgebras with opposite cobracket (same bracket) and opposite bracket (same cobracket) 
will be denoted by $\g_{\mathrm{op}}$ and $\g^{\mathrm{op}}$, respectively.
The dual Lie bialgebra $\g^*$ of $\g$ is the Lie bialgebra with Lie bracket and cobracket given by
\begin{align*}
&\langle [f, g], a \rangle = \langle f \otimes g, \delta(a) \rangle,
&&\langle \delta (f), a \otimes b \rangle = \langle f, [a,b] \rangle, 
&&&\forall a, b \in \g, f, g \in \g^*.
\end{align*}

The Drinfeld double $D(\g)$ of the Lie bialgebra $\g$ is a Lie bialgebra which is isomorphic to $\g \oplus \g^*$ as a vector space and 
is uniquely defined by the conditions:
\begin{enumerate}[leftmargin=*,label=\rm{(\alph*)}]
\item The canonical embeddings $\iota : \g \hookrightarrow D(\g)$ and $\iota^* : (\g^*)^{\mathrm{op}} \hookrightarrow D(\g)$ are embeddings of Lie bialgebras;
\item For $a \in \g \subset D(\g)$, $f \in \g^* \subset D(\g)$, $[x, f] = \ad_x^*(f) - \ad_f^*(x)$ in terms of the coadjoint actions of $\g$ and $\g^*$. 
\end{enumerate}  

A quadratic Lie algebra is a Lie algebra $\g$ equipped with an non-degenerate invariant  symmetric bilinear form $(.,.)$. A Manin triple is a triple $(\g, \g_+, \g_-)$ 
consisting of a quadratic Lie algebra $(\g, (.,.))$ and a pair of isotropic Lie subalgebras $\g_\pm \subset \g$.  

\begin{remark}
\label{Manin-Drinfeld}
The notions of Drinfeld double and Manin triple are equivalent in the case of finite dimensional Lie algebras:
\begin{enumerate}[leftmargin=*,label=\rm{(\alph*)}]
\item Each Drinfeld double $D(\g)$ is a quadratic Lie algebra with symmetric bilinear form
\begin{align*}
&( a + f, b +g) = \langle f, b \rangle + \langle g, a \rangle, && a, b \in \g, f, g \in \g^*.
\end{align*} 
With respect to this form, $(D(\g), \g, \g^*)$ is a Manin triple.
\item For a Manin triple $(\g, \g_+, \g_-)$, $\g_\pm$ have canonical Lie bialgebra structures given by
\begin{align*}
& (  \delta(a), f \otimes g ) = ( a,  [f, g] ), 
&&( \delta (f), a \otimes b ) = - (f, [a,b] ), 
&&&\forall a, b \in \g_+, f, g \in \g_-.
\end{align*}
Then $\g$, equipped with the Lie cobracket $\delta_{\g_+} + \delta_{\g_-}$, is isomorphic to the Drinfeld double of $\g_+$, 
and $\g_- \simeq (\g_+^*)^{\mathrm{op}}$.
\end{enumerate}
\end{remark}

Here is an important class of Lie bialgebras:
$(\g,\delta)$ is \emph{quasitriangular} if $\delta (x) = \ad x (r)$ for all $x\in \g$
where $r = \sum_i r_i \otimes
r^i \in \g \otimes \g$
satisfies the classical Yang-Baxter equation:
$$[r^{12},r^{13}]+[r^{12},r^{23}]+[r^{13},r^{23}]= 0;$$
here   $r^{12} =  r \otimes 1$,  $r^{13} =
\sum_i r_i \otimes 1 \otimes r^i $, $r^{23} =  1 \otimes r$.  In this case we set
$(\g,r):= (\g, \delta)$.  The Drinfled double is the archetypical example of a quasitriangular Lie bialgebra.

\medbreak
Let $r^{21} =\sum_i  r^i \otimes r_i$. 
Recall that a quasitriangular Lie bialgebra $(\g, r)$ is called
\emph{factorizable}  if $r + r^{21} \in S^2 \g$
defines a nondegenerate inner product on $\g^*$ \cite{RSTS}.

\subsection{Poisson algebraic groups} 
\label{subsec:B.2}
A (complex) Poisson algebraic group is an algebraic group $G$ equipped with 
a bivector field $\pi$ such that the product map
\[
(G, \pi) \times (G, \pi) \to (G, \pi)
\]
is Poisson. The coordinate ring $\Cset[G]$ has a canonical structure of 
commutative Poisson-Hopf algebra with Poisson bracket given by 
\begin{align*}
&\{ f, g \} := \langle df \otimes d g, \pi \rangle, && f, g \in \Cset[G],
\end{align*}
where $d f$ denotes the differential of $f$. Conversely, every finitely generated commutative
Poisson-Hopf algebra $H$ gives rise to the Poisson algebraic group $\MaxSpec H$.

The tangent Lie algebra $\g = T_1 G$ of every Poisson algebraic group $G$ has a canonical Lie bialgebra structure. 
The Poisson structure $\pi$ automatically vanishes at the identity element $1$ of $G$. 
The Lie cobracket on $\g$, or equivalently the Lie bracket on $\g^* \simeq T_1^* G$, is defined 
as the linearization of $\pi$ at $1$:
\begin{align}
\label{eq:linear}
&[d_1 (f), d_1 (g)] := d_1 \big( \{ f, g \} \big), &&f, g \in \Cset[G].
\end{align}

In Hopf algebra situations it is advantageous to describe the tangent Lie algebra $\g$ of an algebraic group $G$ by describing the corresponding Lie cobracket 
on $\g^* = T_1^* G$. 

\begin{lemma}\label{lem:tanget-coalg} Let $G$ be a complex algebraic group; as usual 
$\Delta(f) = f_{(1)} \otimes f_{(2)}$ for $f \in \Cset[G]$. Then the canonical Lie coalgebra structure on 
$T_1^* G \simeq \g^*$ is given by
\begin{align*}
&\delta( d_1 f) = d_1 f_{(1)} \wedge d_1 f_{(2)}, 
&& f \in \Cset[G].
\end{align*}
\end{lemma}

\subsection{The classification of Belavin and Drinfeld} 
\label{subsec:B.3}

We fix  a complex finite-dimensional simple Lie algebra $\g$. Pick  a Cartan subalgebra $\h \subset \g$ and 
a set $\Delta \subset \h^*$  of simple roots.
The Casimir element $\Omega \in \g \otimes \g$ of $\g$ is the symmetric tensor 
associated to the Killing form of $\g$; the component of $\Omega$ in $\h \otimes \h$
is denoted by $\Omega_{0}$.

\begin{definition}
A Belavin-Drinfeld triple (BD-triple for short) is a triple
$(\Gamma_1, \Gamma_2, T)$ where $\Gamma_1,$ $\Gamma_2$ are subsets of 
$\Delta$ and
$T: \Gamma_1 \to \Gamma_2$ is a bijection that preserves the inner 
product
and satisfies the nilpotency condition:
for any $\alpha \in \Gamma_1$ there exists a positive integer $n$ for 
which
$T^n (\alpha)$ belongs to $\Gamma_2$ but not to $\Gamma_1$.
\end{definition}

Given  a BD-triple $(\Gamma_1, \Gamma_2, T)$, we denote by
$\widehat{\Gamma_i}$ the set of positive roots
lying in the subgroup generated by $\Gamma_i$, for $i = 1,2$.
There is an associated partial ordering on $\Phi^+$ given by
$\alpha \prec \beta$ if
$\alpha \in \widehat{\Gamma_1}$, 
$\beta \in \widehat{\Gamma_2}$, and $\beta = T^n(\alpha)$
for a positive integer $n$.

\medbreak
A \emph{continuous parameter} for the BD-triple 
$(\Gamma_1, \Gamma_2, T)$
is an element $\lambda \in \h^{\otimes 2}$ such that
\begin{align}
(T (\alpha) \otimes 1) \lambda + (1 \otimes \alpha) \lambda &= 0,
\quad \text{for all } \alpha \in \Gamma_1,
\\
\lambda + \lambda^{21} &= \Omega_0.
\end{align}

Let $\mathfrak{a}_1$, $\mathfrak{a}_2$ be the reductive subalgebras of $\g$ with Cartan 
subalgebras
generated by $h_{\alpha}$, $\alpha$ in $\Gamma_1$, resp. in $\Gamma_2$, 
and
with Dynkin diagrams $\Gamma_1$, respectively $\Gamma_2$. We extend $T$ to
a Lie algebra isomorphism $\widehat T: \mathfrak{a}_1\to \mathfrak{a}_2$.

\begin{theorem}  \cite{BD}.
Let $(\g, r)$ be  a factorizable  Lie bialgebra with underlying simple Lie algebra $\g$. 
Then there exist a Cartan subalgebra $\h$,
a set of simple roots $\Delta$, a BD-triple $(\Gamma_1, \Gamma_2, T)$,  a continuous parameter $\lambda$
and $t\in \ku - 0$ such that the $r$ is given by
\begin{equation}\label{BD}
r =t\Big(\lambda + \sum_{\alpha \in \Phi^+} x_{-\alpha} 
\otimes x_\alpha
+ \sum_{\alpha, \beta \in \Phi^+,  \alpha \prec \beta}
x_{-\alpha} \wedge x_\beta\Big),
\end{equation}
where $x_\beta \in \g_\beta$, $\beta \in \pm\Phi^{+}$, are root vectors  
normalized by
\begin{align}\label{norm1}
(x_\beta \vert x_{-\beta}) &= 1,& \text{for all }\beta &\in \Phi^+, &&
\\ \label{norm2}
\widehat T(x_\beta) &= x_{T(\beta)}, & \text{for all }\beta &\in 
\Gamma_1.
\end{align}
Reciprocally the matrix $r$ defined by \eqref{BD} satisfies the classical Yang-Baxter equation, hence defines 
a factorizable Lie bialgebra structure on $\g$.
\qed
\end{theorem}

\begin{example}\label{exa:standard-poisson}
 We say that a BD-triple $(\Gamma_1, \Gamma_2, T)$ is \emph{empty} if 
$\Gamma_1 = \Gamma_2 = \emptyset$. 
In this case any $\lambda \in \h^{\otimes 2}$ is a continuous parameter; 
the choice $\lambda = \frac{1}{2}(\sum_i h_i \otimes h_i)$, for an orthonormal basis $h_i$ of $\h$, gives rise to the \emph{standard} Poisson structure.
\end{example}

\section{Symplectic cores and symplectic leaves}\label{subsec:symplectic-core}
Let $Z$ be an affine commutative Poisson algebra and $M \coloneqq \MaxSpec Z$.
As usual the point $x\in M$ corresponds to the ideal $\Mg_x$. The material below
is extracted from \cite{BG}.

\medbreak
The largest Poisson ideal  contained in an ideal $I$ of $Z$ is called the  
\emph{Poisson core} of $I$ and denoted $\Pg(I)$; it exists because the sum of
Poisson ideals is again a Poisson ideal. 
 If $I$ is prime, then so is $\Pg(I)$. 
  If $\Mg$ is maximal, then we say that $\Pg(\Mg)$ is Poisson primitive. 
Every prime Poisson ideal of $Z$ is an intersection of Poisson primitive ideals.

\begin{definition}
A \emph{symplectic core} is a class of the equivalence relation $\sim$   given by 
\begin{align*}
x \sim y    &\iff \Pg(\Mg_x) = \Pg(\Mg_y), & x,y& \in M.
\end{align*}
\end{definition}
 The equivalence class of $x\in M$ is  denoted by $\CC(x)$ and called the symplectic core of $x$.
Any symplectic core  is locally closed and smooth in its closure \cite[3.3]{BG}.

\medbreak
Assume for simplicity that $Z$ is regular, i.e., $M$ is smooth, see \cite[3.5]{BG} for the general case. 
Then $M$ becomes a  complex analytic Poisson manifold. Given $x\in M$, the \emph{symplectic leaf} $\Ls(x)$  
 is the maximal connected complex analytic submanifold of $M$ 
such that $x \in \Ls(x)$  and the restriction of the Poisson bracket to $\Ls(x)$
is nondegenerate at every point.
Concretely, the symplectic leaf $\Ls(x)$ is formed by the points which can be reached from $x$ 
 by a piecewise smooth curve, each segment of which is a trajectory of a hamiltonian vector field.
Symplectic leaves might be not algebraic  but they  determine the symplectic cores. 
Below the closure is relative to the Zariski topology. 

\begin{theorem}\cite[Th. 7.4]{Go} Let $\Ls$ be a symplectic leaf.
There is a unique symplectic core $\CC$ in $M$ with $\Ls \subset \CC \subseteq \overline{\Ls}$ and
  $\CC$ is the unique symplectic core dense in $\overline{\Ls}$.
In fact 
\begin{align}\label{eq:symplectic-core=leaf}
\CC = \overline{\Ls} \ \backslash \bigcup_{\substack {\Ks \text{ symplectic leaf} \\ \overline{\Ks} \subsetneq \overline{\Ls}}} \overline{\Ks}.
\end{align}
Each symplectic core $\CC$ in $M$  can be obtained as in \eqref{eq:symplectic-core=leaf}.
\end{theorem}

{\sc  Conflicts of interest and Data availability.}
On behalf of all authors, the corresponding author states that there is no conflict of interest. 
Data sharing not applicable to this article as no datasets were generated or analysed during the current study.

\end{document}